\documentclass[a4paper,12pt,leqno]{amsart} 

\usepackage{amssymb, amsmath, amsthm}
\usepackage{amsfonts} 
\usepackage{color}
\usepackage{mathrsfs}
\usepackage{graphicx}
\usepackage{comment}

\numberwithin{equation}{section}
\newcommand{\e}{\varepsilon}
\newcommand{\R}{\mathbb R}
\newcommand{\Z}{\mathbb Z}

\newcommand{\N}{{\mathbb N}}

\newcommand{\dv}{{\rm{div}}}

\newcommand{\tenchi}{{}^t\!}

\newcommand{\wtts}{\overset{2,2}{\rightharpoonup}}

\renewcommand{\d}{\mathrm{d}}

\newcommand{\vierint}{\int_{0}^1\int_0^T\int_{\square}\int_{\Omega}}

\newcommand{\vep}{\varepsilon}
\newcommand{\lam}{\lambda}
\newcommand{\T}{\mathbb{T}}
\newcommand{\Vs}{V}

\newtheorem{thm}{Theorem}[section]
\newtheorem{lem}[thm]{Lemma}
\newtheorem{rmk}[thm]{Remark}
\newtheorem{prop}[thm]{Proposition}
\newtheorem{defi}[thm]{Definition}

\newtheorem{cor}[thm]{Corollary}

\newtheorem{remark}[thm]{Remark}

\DeclareMathOperator*{\esssup}{ess\,sup}

\allowdisplaybreaks[4]

\title{Space-time homogenization for nonlinear diffusion}

\author{Goro Akagi}
\address[Goro Akagi]{Mathematical Institute and Graduate School of Sciences, Tohoku University, Aoba, Sendai 980-8578 Japan}
\email{goro.akagi@tohoku.ac.jp}

\author{Tomoyuki Oka}
\address[Tomoyuki Oka]{Graduate School of Sciences, Tohoku University, Sendai 980-8579 Japan}
\email{tomoyuki.oka.q3@dc.tohoku.ac.jp}

\date{\today}


\begin{document}


\maketitle

\begin{abstract}
The present paper is concerned with a \emph{space-time homogenization problem} for nonlinear diffusion equations with periodically oscillating (in space and time) coefficients. Main results consist of a homogenization theorem (i.e., convergence of solutions as the period of oscillation goes to zero) as well as a characterization of homogenized equations. In particular, homogenized matrices are described in terms of solutions to cell-problems, which have different forms depending on the log-ratio of the spatial and temporal periods of the coefficients. At a critical ratio, the cell problem turns out to be a parabolic equation in microscopic variables (as in linear diffusion) and also involves the limit of solutions, which is a function of macroscopic variables. The latter feature stems from the nonlinearity of the equation, and moreover, some strong interplay between microscopic and macroscopic structures can be explicitly seen for the nonlinear diffusion. As for the other ratios, the cell problems are always elliptic (in micro-variable only) and do not involve any macroscopic variables, and hence, micro- and macrostructures are weakly interacting each other. Proofs of the main results are based on the two-scale convergence theory (for space-time homogenization). Furthermore, finer asymptotics of gradients,  diffusion fluxes and time-derivatives with certain corrector terms  will be provided, and a qualitative analysis on homogenized matrices will be also performed. 
\end{abstract}

\section{Introduction and main results}

\emph{Homogenization} is a method of asymptotic analysis for complex structures and systems in physics and engineering. For instance, it is often used for modeling various composite materials consisting of a large number of microstructures with some equivalent homogeneous materials at the macroscopic scale. On the other hand, \emph{mathematical homogenization problems} are concerned with rigorous derivations of macroscopic equations, often called \emph{homogenized equation}, by passage to the limit of solutions for microscopic equations with rapidly oscillating coefficients as the oscillation period goes to zero. 

Throughout this paper, let $\Omega$ be a bounded domain of $\R^N$ with smooth boundary $\partial \Omega$. A typical (\emph{periodic}) \emph{homogenization problem} concerns asymptotic behavior as $\vep \to 0_+$ of the (weak) solution $u_\vep = u_\vep(x): \Omega \to \R$ to the homogeneous Dirichlet problem,
\begin{equation}\label{lin-ell}
-\mathrm{div}\,\left(a(\tfrac x \vep) \nabla u_\vep \right) = f \ \mbox{ in } \Omega, \quad u_\vep|_{\partial \Omega} = 0, 
\end{equation}
where $a : \T^N \to \R^{N \times N}$ is an $N \times N$ symmetric matrix field (hence, $a(\tfrac x \vep)$ describes an $\vep$-periodic microstructure) with a torus $\T^N = \R^N / \Z^N$ of dimension $N$ (i.e., $x \in \R^N$ is identified with $x + y$ for any $y \in \Z^N$) such that $a(y)$ is \emph{uniformly elliptic} at each $y \in \T^N$ (see~\eqref{ellip} below without $s$-variable), and moreover, $f=f(x)$ is a given function. A classical (and intuitive) approach to this issue would be a method of \emph{asymptotic expansion}, where $u_\vep$ is (formally) expanded as a series,
\begin{equation}\label{asym_exp}
u_\vep(x) = \sum_{j = 0}^\infty \vep^j u_j(x,\tfrac x\vep) 
\end{equation}
and then, the first few functions $u_j = u_j(x,y) : \Omega \times \T^N \to \R$ for $j=0,1,2,\ldots$ are specified (or characterized) by substituting \eqref{asym_exp} to \eqref{lin-ell} and by focusing on the order of $\vep$ in each term (see, e.g.,~\cite{ben}). Then, at a formal level, $u_0$ turns out to be independent of the second variable $y$ (often called \emph{microscopic variable}), and hence, we can expect (and indeed prove) that $u_\vep$ converges to $u_0$ strongly in $L^2(\Omega)$ as $\vep \to 0_+$. On the other hand, $u_1$ does depend on $y$ as well as $x$, and therefore, $u_\vep$ is no longer strongly convergent in $H^1_0(\Omega)$ due to the oscillation of the gradient,
$$
\nabla u_\vep(x) = \nabla u_0(x) + \vep\nabla u_1(x,\tfrac x \vep) + \nabla_y u_1(x,\tfrac x \vep) + O(\vep),
$$
where $\nabla_y$ stands for the gradient operator with respect to the second variable $y$, as $\vep \to 0_+$. Moreover, $u_1 : \Omega \times \T^N \to \R$ is characterized as a solution of the equation,
$$
- \mathrm{div}_y \left( a(y) \left[ \nabla_y u_1(x,y) + \nabla u_0(x)\right] \right) = 0 \ \mbox{ in } \Omega \times \T^N, 
$$
(where $\mathrm{div}_y$ denotes the divergence in $y$), and hence, it can be written as a linear combination
$$
u_1(x,y) = \sum_{k=1}^N \frac{\partial u_0}{\partial x_k}(x) \Phi_k(y)
$$
of solutions $\Phi_k : \T^N \to \R$ (for $k = 1,2,\ldots,N$) to the following \emph{cell problems},
\begin{equation}\label{cell-pb}
-\mathrm{div}_y \left( a(y) \left[\nabla_y \Phi_k(y)+ e_k \right] \right) = 0 \ \mbox{ in } \T^N,
\end{equation}
where $\{e_k\} = \{[\delta_{jk}]_{j=1,2,\ldots,N}\}$ stands for a canonical basis of $\R^N$. Furthermore, $u_0$ turns out to be a solution of the \emph{homogenized equation},
\begin{equation*}
-\mathrm{div} \big( a_{\rm hom} \nabla u_0) = f \ \mbox{ in } \Omega, \quad u_0|_{\partial \Omega} = 0, 
\end{equation*}
where $a_{\rm hom}$ is the so-called \emph{homogenized matrix} (describing a homogeneous structure equivalent to the original one at the macroscopic scale) given by
\begin{equation}\label{hom_mat}
a_{\rm hom} e_k = \int_{\square} a(y)(\nabla_y \Phi_k(y)+e_k) \, \d y \quad \mbox{ for } \ k=1,2,\ldots,N
\end{equation}
and $\square := (0,1)^N$ is a unit cell. 

Although the argument so far is formal (indeed, \eqref{asym_exp} is an ansatz only), these observations can be justified with the aid of a couple of theories such as \emph{two-scale convergence theory}, which was first proposed by G.~Nguetseng~\cite{ng} and then developed by G.~Allaire~\cite{al1,al2} (see also, e.g.,~\cite{cd2000,ln}) and which enables us to analyze how strong compactness of bounded sequences in Sobolev spaces fails due to their oscillatory behaviors. Indeed, in view of Functional Analysis, homogenization problem could be regarded as a precise analysis on ``breaking of strong compactness'' of (non-trivially) oscillating sequences in Sobolev spaces. 

Homogenization problems have been studied for various types of linear and nonlinear PDEs as well as systems (see, e.g.,~\cite{DeFr75},~\cite{Spa76},~\cite{ben},~\cite{Mar78},~\cite{ng},~\cite{al1,al2},~\cite{MuTa97},~\cite{cd2000},~\cite{cdg02,cdg08,cdg18},~\cite{Vis06}), and moreover, they are not limited to periodic coefficients, but also to almost periodic ones (see, e.g.,~\cite{Koz2},~\cite{Bra,Bra2},~\cite{arm,arm3}) as well as random (stochastic) ones (see, e.g.,~\cite{JKO},~\cite{Koz1},~\cite{PV},~\cite{arm},~\cite{n2018}). In most of them, homogenized equations (or homogenized matrices) are represented in terms of solutions $\{\Phi_k\}$ to a linear elliptic cell-problem  (see, e.g.,~\eqref{cell-pb}) or minimizers of some functionals, whose Euler-Lagrange equations may be elliptic PDEs. Moreover, in homogenization for \emph{nonlinear} PDEs, cell-problems are not always linear (see, e.g.,~\cite{ln}), and solutions $\{\Phi_k\}$ to the cell-problems (or their analogues) are not always irrelevant to the homogenized solution $u_0$, i.e., the limit of solutions $u_\vep$ as $\vep \to 0_+$ (see, e.g.,~\cite{MPP05}). In particular, such a dependence can be regarded as an interaction between micro- and macrostructures due to the nonlinearity through the homogenization. Generally, it is not so easy to explicitly reveal such a dependence of the cell-problem on the homogenized limit $u_0$.

Bensoussan, Lions and Papanicolaou~\cite{ben} studied a \emph{space-time homogenization problem}, which is a homogenization for linear evolution equations involving coefficient matrix fields oscillating both in space and time, based on the method of asymptotic expansion. Moreover, developing two-scale convergence theory, Holmbom~\cite{hol} justified formal observations in~\cite{ben} on a space-time homogenization problem for the linear diffusion equation,
\begin{equation*}
\partial_t u_\vep -\mathrm{div}\,\left(a(\tfrac x \vep,\tfrac t {\vep^r}) \nabla u_\vep \right) = f \ \mbox{ in } \Omega \times (0,T), \quad u_\vep|_{\partial \Omega} = 0, \quad u_\vep|_{t=0} = u^0,
\end{equation*}
where $f = f(x,t)$ is a prescribed data. Here $a : \T^N \times \T \to \R^{N\times N}$ ($\T := \T^1$) is an $N \times N$ matrix field such that $a(y,s)$ is uniformly elliptic at each $(y,s) \in \T^N \times \T$, i.e.,
\begin{equation}\label{ellip}
 \lam|\xi|^2 \leq a(y,s) \xi \cdot \xi \leq |\xi|^2 \quad \mbox{ for } \ \xi \in \R^N
\end{equation}
for some constant $\lam \in (0,1]$. We emphasize that the periods in $x$ and $t$ of $a(\tfrac x \vep, \tfrac t {\vep^r})$ are $\vep$ and $\vep^r$, respectively. Then the form of the corresponding cell-problem changes significantly at a critical value of $r$, and moreover, in contrast to standard homogenization, it is not always elliptic, but parabolic at the critical scale (see also~\cite{DAM97,flo1,flo2,flo3,flo4}). 

In the present paper, we restrict ourselves to a space-time homogenization problem for the nonlinear diffusion equation,
\begin{equation}\label{intro}
\partial_tu_{\e}
-\dv\left( a(\tfrac x \vep, \tfrac t {\vep^r})\nabla |u_{\e}|^{p-1}u_{\e} \right)=f_{\e} \ \text{ in } \Omega\times (0,T), 
\end{equation}
where $r > 0$ is a scale parameter, $1 < p < +\infty$, $f_\vep = f_\vep(x,t): \Omega \times (0,T) \to \R$ is a prescribed function and $a : \T^N \times \T \to \R^{N\times N}$ is a symmetric matrix field satisfying \eqref{ellip}, furnished with the  homogeneous  Dirichlet boundary and initial conditions. By analogy to classical Fick's law, the diffusion coefficient (matrix) of \eqref{intro} is proportional to $|u_\vep|^{p-1}$. In case $p=1$, \eqref{intro} is nothing but the classical (linear) normal diffusion. In case $0<p<1$, Equation \eqref{intro} is called a \emph{fast diffusion equation}, where $|u_\vep|^{p-1}$ is decreasing in $|u_\vep|$ and divergent as $u_\vep \to 0$; hence, it is also classified as a \emph{singular diffusion}. In case $1<p<+\infty$, \eqref{intro} is called a \emph{porous media equation}, where $|u_\vep|^{p-1}$ increases in $|u_\vep|$ and vanishes at $u_\vep =0$; i.e., it is a \emph{degenerate diffusion}. The classical porous medium and fast diffusion equations (that is, \eqref{intro} with the identity matrix $a \equiv \mathbb{I}$) have been well studied in a great deal of literature. For instance, we refer the reader to the books~\cite{vaz,vaz2} for more details. 

Now, let us consider the Cauchy-Dirichlet problem,
\begin{equation}\label{eq}
  \left\{
    \begin{aligned}
    \partial_tu_{\e}&=\dv\left( a\left(\tfrac{x}{\e},\tfrac{t}{\e^r}\right)\nabla |u_{\e}|^{p-1}u_{\e} \right)+f_{\e} 
     &&\text{ in } \Omega\times (0,T), \\
    |u_{\e}|^{p-1}u_{\e}&=0 &&\text{ on } \partial\Omega\times (0,T), \\
    u_{\e}&=u^0  &&\text{ in } \Omega \times \{0\},
    \end{aligned}
    \right.
\end{equation}
where $a = [a_{ij}]_{i.j=1,2,\ldots,N}$ is an $N \times N$ symmetric matrix field over $\R^N \times [0,+\infty)$  satisfying \eqref{ellip}  and $u_0 \in H^{-1}(\Omega)$ and $f_\vep : (0,T) \to H^{-1}(\Omega)$ are given. Throughout the present paper, we are concerned with \emph{weak solutions} of \eqref{eq} defined by
\begin{defi}[Weak solution of \eqref{eq}]\label{D:sol}
A function $u_\e = u_\e(x,t) : \Omega \times (0,T) \to \R$ is called a {\rm(\emph{weak})} \emph{solution} to \eqref{eq}, if the following conditions are all satisfied\/{\rm:}
\begin{itemize}
\rm
 \item[(i)] $u_{\e}\in W^{1,2}(0,T;H^{-1}(\Omega)) \cap L^{p+1}(\Omega\times(0,T))$, $|u_\vep|^{p-1}u_\e \in L^{2}(0,T;H^{1}_0(\Omega))$ and $u_{\e}(t,0)\to u^0$ strongly in $H^{-1}(\Omega)$ as $t\to 0_+$, 
 
 \item[(ii)] it holds that
\begin{equation*}
\left\langle \partial_t u_{\e}(t),w\right\rangle_{H^1_0(\Omega)} + \int_{\Omega}a\left(\tfrac{x}{\e},\tfrac{t}{\e^r}\right) \nabla (|u_\e|^{p-1}u_\e)(x,t)\cdot \nabla w(x)\ dx =\langle f_{\e}(t),w\rangle_{H^{1}_0(\Omega)}
\end{equation*}
for a.e.~$t\in (0,T)$ and all $w\in H^1_0(\Omega)$.
\end{itemize}
\end{defi}
We begin with the well-posedness of \eqref{eq} (see also Notation below).
\begin{thm}[Well-posedness of \eqref{eq}]\label{T:wp}
Let $0 < p, r, \vep < +\infty$ and let $a=[a_{ij}]_{i,j=1,2,\ldots,N}$ be an $N \times N$ symmetric matrix field satisfying \eqref{ellip} as well as $(x,t) \mapsto a_{ij}(\tfrac x\e,\tfrac t{\e^r}) \in W^{1,1}(0,T;L^{\infty}(\Omega))$ for $i,j =1,2,\ldots,N$. Then for any $f_{\e} \in W^{1,2}( 0,T ;H^{-1}(\Omega)) \cap L^1( 0,T;L^2(\Omega))$ and $u^0\in L^2(\Omega) \cap L^{p+1}(\Omega)$, the Cauchy-Dirichlet problem \eqref{eq} admits a unique weak solution $u_\vep = u_\vep(x,t): \Omega \times (0,T) \to \R$ such that 
\begin{align*}
 u_{\e} &\in W^{1,\infty}_{\mathrm{loc}}((0,T];H^{-1}(\Omega)) \cap  C_{\rm w}([0,T];L^2(\Omega)) \cap C([0,T];L^{p+1}(\Omega)),\\
 |u_{\e}|^{p-1}u_{\e} &\in  L^{\infty}_{\mathrm{loc}}((0,T];H^{1}_0(\Omega))\cap C([0,T];L^{(p+1)/p}(\Omega)).
\end{align*}
Furthermore, the weak solution continuously depends on the initial datum in the following sense\/{\rm :} let $u^{0,1}$, $u^{0,2}\in H^{-1}(\Omega)$ and let $u_1$, $u_2$ be weak solutions of \eqref{eq} for the initial data $u^{0,1}$, $u^{0,2}$, respectively. Then there exists a constant $C_T\ge 0$ depending on $T$ but independent of $t$, $u^{0,1}$ and $u^{0,2}$ such that  
\begin{align*}
\sup_{t\in [0,T]}\left\|u_1(t)-u_2(t)\right\|_{H^{-1}(\Omega)}^2  \le C_T\|u^{0,1}-u^{0,2}\|_{H^{-1}(\Omega)}^2.
\end{align*}
\end{thm}

The existence part of the assertion can be proved by applying a general theory on the existence of gradient flows for time-dependent convex energies in reflexive Banach spaces (see~\cite{akagi}). On the other hand, the rest of the assertions (i.e., uniqueness and continuous dependence on data) still seems non-trivial due to the presence of the time-dependent coefficient $a(\tfrac x \vep, \tfrac t {\vep^r})$. 

We next give a homogenization theorem.  Here and henceforth, $J := (0,1)$ stands for the unit interval.
\begin{thm}[Homogenization theorem for \eqref{eq}]\label{thm1}
Let $0 < p, r < +\infty$ and let $\e_n \to 0_+$ be an arbitrary sequence in $(0,+\infty)$. In addition to the assumptions in Theorem \ref{T:wp} for $\e = \e_n$,  suppose that
\begin{itemize}
 \item 
 $f_{\e_n} \to f$ weakly in $L^2( 0,T;H^{-1}(\Omega))$, 
 \item $(f_{\e_n})$ is bounded in $L^1( 0,T;L^2(\Omega))$ if $p\in(0,1)$,
 \item $a_{ij}$ is $(\square \times J)$-periodic for $i,j=1,2,\ldots,N$.
\end{itemize}
Let $u_{\e_n}$ be the unique weak solution to \eqref{eq} with $\e=\e_n$. Then there exist a {\rm (}not relabeled{\rm )} subsequence of $(\e_n)$ and functions
\begin{align*}
u_{0}&\in W^{1,2}(0,T;H^{-1}(\Omega)) \cap L^{p+1}(\Omega\times(0,T)) \cap W^{1,\infty}_{\mathrm{loc}}((0,T];H^{-1}(\Omega)) \cap C_{\rm w}([0,T];L^2(\Omega)),\\
z&\in L^{2}(\Omega\times (0,T) ;L^{2}(J;H^{1}_{\mathrm{per}}(\square)/\R))
\end{align*}
{\rm (}see Notation at the end of this section{\rm )} such that $|u_{0}|^{p-1}u_{0} \in L^2(0,T;H^1_0(\Omega))$,
\begin{alignat}{4}
|u_{\e_n}|^{p-1}u_{\e_n} &\to |u_{0}|^{p-1}u_{0} \quad &&\text{ weakly in }\ L^2(0,T;H^1_0(\Omega)), \label{thm1pf1}\\
u_{\e_n} &\to u_{0} \quad &&\text{ strongly in }\ L^\rho(0,T;L^{p+1}(\Omega))\label{thm1pf3}
\end{alignat}
for any $\rho \in [1,+\infty)$ and
\begin{align}
a(\tfrac{x}{\e_n},\tfrac{t}{\e_n^r})\nabla |u_{\e_n}|^{p-1}u_{\e_n} &\wtts a(y,s) \left(\nabla |u_{0}|^{p-1}u_0+\nabla_y z\right)\quad \mbox{ in } \  [L^2(\Omega \times  (0,T)  \times \square \times J)]^N \label{thm1pf2}
\end{align}
where $\wtts$ denotes the notion of weak two-scale convergence and it will be defined in Section \ref{Ss:wtts} below. Moreover, the limit $u_0$ solves the weak form of the homogenized equation,
\begin{equation}\label{homeq}
\left\{
\begin{aligned}
&\langle \partial_tu_{0}(t),\phi\rangle_{H^1_0(\Omega)}+\int_{\Omega} j_{\rm hom}(x,t)\cdot\nabla\phi(x)\, dx  =\int_{\Omega} f(x,t)\phi(x)\, dx \ \mbox{ for } \, \phi\in H^1_0(\Omega),\\
&u_{0}( \cdot,0)=u^0 \ \mbox{ in } \Omega
\end{aligned}
 \right.
\end{equation}
with a homogenized diffusion flux $j_{\rm hom}\in[L^2(\Omega\times (0,T))]^{N}$ given by 
\begin{equation}\label{jhom}
j_{\rm hom}(x,t)=\int_0^1\int_{\square}a(y,s)\left(\nabla |u_{0}|^{p-1}u_0(x,t)+\nabla_yz(x,t,y,s)\right)\, dyds
\end{equation}
for a.e.~$t \in (0,T)$.
\end{thm}

We now move on to a qualitative analysis of the space-time homogenization. In the next theorem, the homogenized diffusion flux $j_{\rm hom}$ will be represented as
\begin{equation}\label{j-a}
j_{\rm hom}=a_{\rm hom}\nabla |u_0|^{p-1}u_0 \ \mbox{ in } \Omega \times (0,T) 
\end{equation}
for a homogenized matrix $a_{\rm hom}$, and moreover, $a_{\rm hom}$ will be characterized depending on the scale parameter $0 < r < +\infty$.
\begin{thm}[Characterization of homogenized matrices]\label{thm2}
Let $p \in(0,2)$. In addition to all the assumptions of Theorem \ref{thm1}, suppose that
\begin{equation}\label{t2:hyp}
\begin{cases}
(f_{\e_n}) \mbox{ is bounded in } L^1( 0,T;L^{3-p}(\Omega)),\\
u^0 \in L^{3-p}(\Omega) \ \mbox{ if } \ p \in (0,1)\,{\rm ;} \ \
u^0 \in L^{p+1}(\Omega)\ \mbox{ if } \ p \in (1,2).
\end{cases}
\end{equation}
Let $u_0$ be a {\rm (}homogenized{\rm )} limit of weak solutions $(u_{\e_n})$ to \eqref{eq} along a sequence $\e_n \to 0_+$ such that \eqref{thm1pf1}--\eqref{thm1pf2} are fulfilled, and hence, $u_0$ is a weak solution of the homogenized equation \eqref{homeq}. Then the homogenized flux $j_{\rm hom}(x,t)$ is represented as \eqref{j-a} for a homogenized matrix $a_{\rm hom}$. 

Moreover, $a_{\rm hom}$ can be characterized as follows\/{\rm :}
\begin{description}
\item[\rm{(i)}] In case $0<r<2$, $a_{\rm hom}$ is a constant $N\times N$ matrix given by 
\begin{align}\label{ahomfast}
a_{\rm hom}e_k=\int_0^1\int_{\square}a(y,s) \left(\nabla_y\Phi_k(y,s)+e_{k}\right)\, dyds \ \mbox{ for } k = 1,2,\ldots,N,
\end{align}
where $\Phi_k\in L^2(J;H^1_{\mathrm{per}}(\square)/\R)$ is the unique weak solution to the cell-problem\/{\rm :} 
\begin{equation}\label{local1}
-\dv_y \left(a(y,s) \left[ \nabla_y\Phi_k(y,s)+e_{k}\right]\right)=0\ \text{ in }\ \T^N \times \T.
\end{equation}
Furthermore, the pair $(u_0,z)$ satisfying \eqref{thm1pf1}--\eqref{homeq} is uniquely determined. Hence $(u_{\e_n})$ converges to $u_{0}$ {\rm(}without taking any subsequence{\rm )}. Moreover, the function $z=z(x,t,y,s)$ can be written as
\begin{equation}\label{z}
z(x,t,y,s) = \sum_{k=1}^N \left( \partial_{x_k} |u_0|^{p-1}u_0(x,t) \right) \Phi_k(y,s).
\end{equation} 
\item[\rm{(ii-FDE)}] In case $r=2$ and $p \in (0,1)$, the homogenized matrix $a_{\rm hom}(x,t)$ is characterized by
\begin{equation}\label{ahom:c:fd}
a_{\rm hom}(x,t) e_k = \int^1_0 \int_\square a(y,s) \left( \nabla_y \Phi_k(x,t,y,s) + e_k \right) \, d y d s, 
\end{equation}
where $\Phi_k = \Phi_k(x,t,y,s)$ is a function belonging to $L^{\infty}(\Omega \times (0,T)  ; L^2(J ; H^1_{\rm per}(\square)/\R))$ such that 
\begin{align*}
|u_0|^{1-p}\Phi_k &\in L^{\infty}(\Omega \times (0,T) ; W^{1,2}(J;[H^1_{\rm per}(\square)/\R]^*)),\\
|u_0|^{(1-p)/2}\Phi_k &\in L^{\infty}(\Omega \times (0,T) ; C(\overline{J};L^2(\square)/\R))
\end{align*}
and solves the cell-problem,
\begin{equation}\label{local2}
\left\{
\begin{array}{ll}
\frac{1}{p}|u_0(x,t)|^{1-p}\partial_s\Phi_k(x,t,y,s)=\dv_y\left(a(y,s)\left[\nabla_y\Phi_k(x,t,y,s)+e_{k}\right]\right) \hspace{-1mm}&\mbox{in } \T^N \times \T,\\
 \Phi_k(x,t,y,0)=\Phi_k(x,t,y,1) &\mbox{in } \T^N
\end{array}
\right.
\end{equation}
for each $(x,t) \in \Omega \times (0,T)$. Moreover, $z$ is given by \eqref{z} with $\Phi_k=\Phi_k(x,t,y,s)$.
\item[\rm (ii-PME)] In case $r=2$ and $p\in(1,2)$, the homogenized matrix $a_{\rm hom}(x,t)$ is characterized by \eqref{ahom:c:fd} with $\Phi_k$ given by
\begin{align}\label{ahom:c:pm}
\Phi_k(x,t,y,s)
= \begin{cases}
p |u_0(x,t)|^{p-1} \Psi_k(x,t,y,s) &\mbox{if } \ u_0(x,t) \neq 0,\\
0 &\mbox{if } \ u_0(x,t) = 0,
   \end{cases}
\end{align}
where $\Psi_k = \Psi_k(x,t,y,s)$ is a function lying on  $L^\infty ( [u_0\neq0]; W^{1,2}(J;[H^1_{\rm per}(\square)/\R]^*))$ with the measurable set $[u_0\neq0] := \{(x,t) \in \Omega\times (0,T)\colon u_0(x,t) \neq 0\}$  such that 
\begin{align*}
|u_0|^{p-1}\Psi_k &  \in L^{\infty}([u_0\neq0] ; L^2(J; H^1_{\rm per}(\square)/\R)),\\
 |u_0|^{(p-1)/2} \Psi_k &  \in L^\infty([u_0\neq0] ; C(\overline{J};L^2(\square)/ \R)) 
\end{align*} 
and solves the cell-problem,
\begin{equation}\label{local21}
\left\{
\begin{array}{ll}
\partial_s\Psi_k(x,t,y,s)=\dv_y\left(a(y,s)\left[ p|u_0(x,t)|^{p-1}\nabla_y\Psi_k(x,t,y,s)+e_{k}\right]\right) \hspace{-1mm}&\mbox{in } \T^N \times \T,\\
\Psi_k(x,t,y,0) = \Psi_k(x,t,y,1) &\mbox{in } \T^N
\end{array}
\right.
\end{equation}
for each $(x,t) \in  [u_0\neq0]$. Furthermore, $z$ is given as in {\rm (ii-FDE)}.
\item[\rm{(iii)}] In case $2<r<+\infty$, 
$a_{\rm hom}$ is a constant $N\times N$ matrix given by 
\begin{equation}\label{a_hom3}
a_{\rm hom}e_k=\int_{\square}\Bigl(\int_0^1a(y,s)\ ds\Bigl)(\nabla_y\Phi_k(y)+e_{k})\, dy,
\end{equation}
where $\Phi_k\in H^1_{\mathrm{per}}(\square)/\R$ is the unique weak solution to the cell problem,
\begin{equation}\label{local3}
-\dv_y\biggl(\Bigl(\int_0^1a(y,s)\ ds\Bigl)(\nabla_y\Phi_k(y)+e_{k})\biggl)=0\ \text{ in }\ \T^N.
\end{equation}
Furthermore, the pair $(u_0,z)$ satisfying \eqref{thm1pf1}--\eqref{homeq} is uniquely determined. Hence $(u_{\e_n})$ converges to $u_{0}$ {\rm(}without taking any subsequence{\rm )}. Finally, $z$ is independent of $s$ and given by \eqref{z} with $\Phi_k = \Phi_k(y)$.
\end{description}
\end{thm}

\begin{remark}[Interpretation of the assertions]\label{R:interpre}
{\rm
\begin{enumerate}
 \item[(i)] In case $r \in (0,2)$, the oscillation in space of the coefficient field $a(\tfrac x \e, \tfrac t {\e^r})$ gets much faster than that in time as $\e$ gets smaller. Intuitively speaking, the assertion for this case can be understood as follows: the homogenization seems to be first performed only in space with the fixed microscopic variable $s$, and then, that in time follows. Therefore the cell-problem \eqref{local1} has a similar form to that for the time-independent case $a = a(\tfrac x \e,s)$ with  the parameter  $s$ fixed, and hence, $\Phi_k=\Phi_k(y,s)$ depends on $s$ as well as $y$. Furthermore, the homogenized matrix $a_{\rm hom}$ given by \eqref{ahomfast}  can be seen  as an average in $s$ of  a homogenized matrix for the (spatially) oscillating coefficient $x \mapsto a(\tfrac x \e,s)$ (see \eqref{hom_mat}). Indeed, if $a=a(y)$ depends only on $y$, then one can prove the homogenization theorem for \eqref{eq} with the homogenized matrix represented by \eqref{hom_mat} with the unique weak solution $\Phi_k = \Phi_k(y)$ to \eqref{cell-pb} for $k=1,2,\ldots,N$. In case $r \in (2,+\infty)$, the oscillation in time gets much faster than that in space. Hence the homogenization in time is first performed, and then, that in space follows. Therefore the cell-problem and the homogenized matrix $a_{\rm hom}$ look like time-independent ones but with the averaged matrix field of $a(y,s)$ in $s$ (hence, $\Phi_k=\Phi_k(y)$ depends only on $y$). On the other hand, the case $r = 2$ is a critical case, where the oscillation speeds in space and time are balanced, and accordingly, the cell-problem has a parabolic form.
\item[(ii)] In the critical case (i.e., $r = 2$), it is noteworthy that the cell-problem and its solutions depend on macroscopic variables $x,t$ as well, although they do not for $r \neq 2$. In Theorem \ref{thm2}, the strong interplay between macroscopic and microscopic structures are explicitly observed through the homogenization for the case $r = 2$. On the other hand, microscopic and macroscopic structures interact weakly (in particular, the homogenized matrix is constant over $\Omega \times (0,T)$) for the other case $r \neq 2$. Moreover, in the critical case, the degeneracy and singularity of nonlinear diffusion also emerge in the characterization of $\Phi_k = \Phi_k(x,t,y,s)$ (and hence, of the homogenized matrix). To see this, let $(x,t) \in \Omega \times (0,T)$ be such that $u_0(x,t) = 0$, where the diffusion may be degenerate or singular. For the singular diffusion (i.e., $0 < p < 1$), the cell-problem turns out to be elliptic by \eqref{local2} and similar to the case $r \in (0,2)$. On the the hand, for the degenerate diffusion (i.e., $1 < p < 2$), the homogenized matrix $a_{\rm hom}(x,t)$ is simply given as an average,
$$
a_{\rm hom}(x,t) = \int^1_0 \int_\square a(y,s) \, dy ds.
$$
\item[(iii)] The uniqueness of the limit $(u_0,z)$  is still open for the critical case $r = 2$. 
\end{enumerate}
}
\end{remark}

The next theorem exhibits finer asymptotics  of  gradients with a certain corrector. As we have seen, the gradients $\nabla |u_\vep|^{p-1}u_\vep$ cannot converge strongly in $L^2(\Omega)$ as $\vep \to 0_+$, and therefore, we need a corrector (see a summation term in \eqref{cor} below), which may oscillate and destroy the strong compactness of gradients, in order to  describe  a precise asymptotic behavior of the gradients.

\begin{thm}[Corrector for gradient convergence]\label{T:cor}
Let $p \in(0,2)$. In addition to all the assumptions in Theorem \ref{thm2}, suppose that 
\begin{itemize} 
 \item[(i)] $f_{\e_n} \to f$ strongly in $L^2( 0,T;H^{-1}(\Omega))$ or weakly in $L^\sigma( 0,T;L^{(p+1)/p}(\Omega))$ for some $\sigma > 1$,
 \item[(ii)] $a \in L^\infty(J;C^\alpha_{\rm per}(\square))$ for some $\alpha \in (0,1)$ if $r \neq 2$\,{\rm ;} $a$ is smooth in $(y,s)$ if $r=2$.
\end{itemize}
Let $u_0$ be a {\rm (}homogenized{\rm )} limit of weak solutions $(u_{\e_n})$ to \eqref{eq} along a sequence $\e_n \to 0_+$ such that \eqref{thm1pf1}--\eqref{thm1pf2} are satisfied. Moreover, let $\Phi_{k}$ be the weak solution of the cell-problem depending on $r$, $p$ {\rm (}more precisely, \eqref{local1} if $r \in (0,2)$\/{\rm ;} \eqref{local2} if $r =2$ and $p \in (0,1)$\/{\rm ;} \eqref{local21} along with \eqref{ahom:c:pm} if $r =2$ and $p \in (1,2)$\/{\rm ;} \eqref{local3} if $r \in (2,+\infty)${\rm )} associated with the limit $u_0$ for the case $r = 2$. Then it holds that
\begin{equation}\label{cor}
\lim_{\e_n\to 0_+} \int_{0}^T\int_{\Omega} \Bigl|\nabla |u_{\e_n}|^{p-1}u_{\e_n}-\nabla  |u_0|^{p-1}u_0-\sum_{k=1}^N\left(\partial_{x_k} |u_0|^{p-1}u_0\right)\nabla_y\Phi_k\left(x,t,\tfrac{x}{\e_n},\tfrac{t}{\e_n^r}\right)\Bigl|^2\, dxdt=0.
\end{equation}
Here $\Phi_k$ depends only on $(y,s)$ for $r \in (0,2)$ and on $y$ for $r \in (2,+\infty)$, respectively.
\end{thm}

 Moreover,  we have the following corollary (see also Remark \ref{R:corr-non0} in \S \ref{S:cor}):
\begin{cor}[Corrector results for diffusion flux and time-derivative]\label{C:cor}
Under the same assumptions as in Theorem \ref{T:cor}, it holds that
\begin{align}
\label{cor-f}
\lim_{\e_n\to 0_+} \int_{0}^T \bigg\| j_{\e_n} - j_{\rm hom} 
- \Bigl[ a_{\e_n} \Bigl( 
	\nabla |u_0|^{p-1}u_0 + \sum_{k=1}^N \left( \partial_{x_k} |u_0|^{p-1}u_0 \right) \nabla_y\Phi_k \left(x,t,\tfrac{x}{\e_n},\tfrac{t}{\e_n^r}\right) \Bigl) \\
- j_{\rm hom} 
\Bigl] \bigg\|_{L^2(\Omega)}^2\, dt=0,\nonumber\\
\label{cor-t}
\lim_{\e_n\to 0_+} \int_{0}^T \bigg\| \partial_t u_{\e_n} - \partial_t u_0 
- \mathrm{div} \Bigl[ 
a_{\e_n} \Bigl( 
	\nabla |u_0|^{p-1}u_0 + \sum_{k=1}^N \left( \partial_{x_k} |u_0|^{p-1}u_0 \right) \nabla_y\Phi_k \left(x,t,\tfrac{x}{\e_n},\tfrac{t}{\e_n^r}\right) \Bigl) \\
- j_{\rm hom} 
\Bigl] \bigg\|_{H^{-1}(\Omega)}^2\, dt=0,\nonumber
\end{align}
where $a_{\vep_n} = a(\tfrac x {\e_n}, \tfrac t {\e_n^r})$, $j_{\e_n} = a_{\e_n} \nabla |u_{\e_n}|^{p-1}u_{\e_n}$ and $j_{\rm hom}$ is given by \eqref{jhom}.
\end{cor}

Finally, we shall discuss qualitative properties of the homogenized matrix $a_{\rm hom}$.

\begin{prop}[Qualitative properties of $a_{\rm hom}$]\label{P:ahom}
Let $p \in(0,2)$. Under the same assumptions as in Theorem \ref{thm2}, let $a_{\rm hom}$ and $\{\Phi_k\}_{k=1,2,\ldots,N}$ be defined as in Theorem \ref{thm2}. Then the following {\rm (i)} and {\rm (ii)} hold true\/{\rm:}
\begin{itemize}
\item[(i)]{\rm(}Improved uniform ellipticity{\rm)} Let $\lambda > 0$ be the ellipticity constant of $a(y,s)$ given in \eqref{ellip}. It then holds that
\begin{align*}
\lefteqn{
\lambda \sum_{k=1}^N\left(1+\int^1_0\|\Phi_k(x,t,\cdot,s)\|_{L^2(\square)}^2 \,ds\right)|\xi_k|^2
}\\
&\le a_{\rm hom}(x,t)\xi\cdot\xi \leq  \sum_{k=1}^N\left(1+\int^1_0\|\Phi_k(x,t,\cdot,s)\|_{L^2(\square)}^2 \,ds\right)|\xi_k|^2
\end{align*}
for any $\xi=[\xi_k]_{k=1,2,\ldots,N}\in \R^N$ and a.e.~$(x,t) \in \Omega \times (0,T)$. Here $a_{\rm hom}$ is a constant matrix if $r \neq 2$. Moreover, $\Phi_k$ depends only on $y,s$ if $r \in (0,2)$ and on $y$ if $r \in (2,+\infty)$.
\item[(ii)]{\rm(}Symmetry and asymmetry{\rm)} For $r\neq2$, $a_{\rm hom}$ is symmetric. On the other hand, for $r = 2$, $a_{\rm hom}(x,t)$ is not symmetric {\rm (}respectively, symmetric{\rm )} when $u_0(x,t) \neq 0$ {\rm (}respectively, $u_0(x,t)=0${\rm )}.
\end{itemize}
\end{prop}

In qualitative properties of $a_{\rm hom}$ mentioned above, one can observe effects of degeneracy or singularity of the diffusion as well as the ratio $r$ of frequencies. Firstly, we stress that the ellipticity of the matrix field seems to be improved by the homogenization. Indeed, the ellipticity constant of the homogenized matrix is bigger than $\lambda$, which quantifies the uniform ellipticity of the matrix field $a(y,s)$, whenever $r \neq 2$.  For the fast diffusion case at the critical ratio $r=2$,  one can also observe a similar improvement. On the other hand,  for the porous medium case  at the critical ratio, the uniform ellipticity is not improved through the homogenization at every \emph{degenerate} point $(x,t)$ (at which $u_0(x,t)$ vanishes and then $\Phi_k(x,t,\cdot,\cdot) \equiv 0$ in $\square \times (0,1)$ by Theorem \ref{thm2}). Secondly, the symmetry of the homogenized matrix is inherited from that of the matrix field when $r \neq 2$; however, at the critical ratio $r=2$, the homogenized matrix $a_{\rm hom}(x,t)$ violates the symmetry at which $u_0(x,t) \neq 0$. On the other hand, it is still symmetric at which $u_0(x,t)=0$.  The skew-symmetric part of $a_{\rm hom}$ will turn out to make no contribution to the homogenized diffusion (see Remark \ref{R:no-contr} in \S \ref{S:hmat}). 

\bigskip
\noindent
{\bf Structure of the paper.} This paper is composed of seven sections. 
In the next section, we briefly review the relevant material on space-time two-scale convergence and maximal monotone operator. In particular, we prove lemmas on gradient two-scale compactness as well as very weak two-scale convergence. Indeed, they are tiny extensions of the original ones proved in~\cite{hol}; however, the original ones do not seem available for our analysis later. Section \ref{S:wp} is devoted to proving Theorem \ref{T:wp} on the well-posedness of \eqref{eq}. In Section \ref{S:est}, we shall establish uniform (in $\e>0$) estimates for $u_\e$ as well as $|u_\e|^{p-1}u_\e$. We shall further derive their strong convergence in a couple of different topologies. Section \ref{S:pf} provides proofs of Theorems \ref{thm1} and \ref{thm2}. 
In section \ref{S:cor}, we shall prove Theorem \ref{T:cor}  and Corollary \ref{C:cor}.  In the final section, we shall discuss qualitative properties of the homogenized matrix $a_{\rm hom}$ to prove Proposition \ref{P:ahom}.

\vspace{2mm}
\noindent
{\bf Notation.}\ 
Throughout this paper, we shall use the following notation\/:
\begin{itemize}
\item We shall often denote by $\vep \to 0$ an arbitrary sequence $\vep_n \to 0_+$.
\item Let $N\in\N$ and let $\Omega$ be a bounded domain of $\R^N$ with smooth boundary $\partial \Omega$.
\item We simply write $I=(0,T)$ for $T>0$ and $dZ=dxdydtds$. Moreover, $\square=(0,1)^N$ and $J=(0,1)$ are the unit cell and interval, respectively. Furthermore, $\T^N = \R^N /\Z^N$ and $\T = \R /\Z$ denote the $N$- and $1$-dimensional tori, respectively.
\item The vector $e_k= [\delta_{jk}]_{j=1,2,\ldots,N}$ denotes the $k$-th vector of the canonical basis of $\R^N$. Here $\delta_{jk}$ denotes the Kronecker delta.  Let $a = [a_{ij}]_{i,j=1,2,\ldots,N}$ be an $N \times N$ real matrix and denote by $\tenchi a = [(\tenchi a)_{ij}]_{i,j=1,2,\ldots,N}$ its transposition, that is, $(\tenchi a)_{ij} = a_{ji}$ for $i,j=1,2,\ldots,N$.
\item Moreover, $\|\cdot\|_{H^1_0(\Omega)}$ is the norm of $H^1_0(\Omega)$ given by $\|\cdot\|_{H^1_0(\Omega)}:=\| |\nabla\cdot| \|_{L^2(\Omega)}$  (we shall simply write $\|\nabla \cdot\|_{L^2(\Omega)}$ instead of $\| |\nabla\cdot| \|_{L^2(\Omega)}$ below). 
\item Define the set of smooth $\square$-periodic functions by
\begin{align*}
C^{\infty}_{\rm per}(\square) 
&= \{w\in C^{\infty}(\R^N) \colon w(\cdot+e_k)=w(\cdot) \text{ in } \R^N \ \text{ for }\ 1\leq k \leq N\}.
\end{align*}
Furthermore, $C^\alpha_{\rm per}(\square)$ is also defined analogously for $\alpha \in (0,1)$. 
\item We also define $W^{1,q}_{\mathrm{per}}(\square)$ and $L^q_{\mathrm{per}}(\square)$ as closed subspaces of $W^{1,q}(\square)$ and $L^q(\square)$ by
$$
W^{1,q}_{\mathrm{per}}(\square) = \overline{C^\infty_{\rm per}(\square)}^{\|\cdot\|_{W^{1,q}(\square)}}, \quad L^q_{\mathrm{per}}(\square) = \overline{C^\infty_{\rm per}(\square)}^{\|\cdot\|_{L^q(\square)}} \simeq L^q(\square),
$$ 
respectively, for $1\leq q < +\infty$. In particular, set $H^1_{\rm per}(\square) := W^{1,2}_{\rm per}(\square)$. We shall simply write $L^q(\square)$ instead of $L^q_{\rm per}(\square)$, unless any confusion may arise.
\item Define the \emph{mean} $\langle w \rangle_y := \int_\square w(y) \, d y$ in $y$ of $w \in L^1(\square)$. We set $\Vs = H^1_{\mathrm{per}}(\square)/\R = \{w \in H^1_{\mathrm{per}}(\square) \colon \langle w\rangle_{y} =0\}$ equipped with norm $\|\cdot\|_{\Vs} := \|\nabla \cdot\|_{L^2(\square)}$. 
\item Moreover, $\nabla_y$ and $\mathrm{div}_y$ stand for the gradient and divergence operators in $y$.
\item Furthermore, let $X$ be a normed space  with a norm $\|\cdot\|_X$ and a duality pairing $\langle \cdot, \cdot \rangle_X$ between $X$ and its dual space $X^*$  and denote by $C_{\rm w}(\overline{I};X)$ the set of all weakly continuous functions defined on $\overline{I}$ with values in $X$.  Moreover, we write $X^N = X \times X \times \cdots \times X$ ($N$-product space), e.g., $[L^2(\Omega)]^N = L^2(\Omega;\R^N)$.
\end{itemize}

To clarify variables of integration, we shall often write, e.g., 
$\|u(\tfrac{x}{\e})\|_{L^{q}(\Omega)}$ and $\|u(x,\tfrac{x}{\e})\|_{L^{q}(\Omega)}$ instead of $\|u(\tfrac{\cdot}{\e})\|_{L^{q}(\Omega)}$ and $\|u(\cdot,\tfrac{\cdot}{\e})\|_{L^{q}(\Omega)}$, respectively.
Indeed, it is convenient to handle functions depending on $x$, $x/\e$, $t$ and $t/\e^r$. Moreover, $C^{\infty}_{\rm per}(J)$, $L^{q}(J)$, $W^{1,q}_{\rm per}(J)$ and $C_{\rm per}(\square \times J)$, $1\le q\le \infty$ are defined in an analogous way. We also denote by $\langle \cdot \rangle_s$ and $\langle \cdot \rangle_{y,s}$ the means of functions in $s$ and $(y,s)$, respectively. We often write $u(t)$ instead of $u(\cdot,t)$ for each $t\in I$ and $u:\Omega\times I\to\R$.  Finally,  we denote by $C$ a non-negative constant which is independent of the elements of the corresponding space or set and may vary from line to line.



\section{Preliminaries}

In this section, we summarize preliminary facts to be used later.

\subsection{Weak space-time two-scale convergence}\label{Ss:wtts}

In this subsection, we briefly review a \emph{space-time two-scale convergence theory} developed in~\cite{hol}.  Some of them will be customized for our analysis later.  Throughout this subsection, we always assume that
\[
  0<r<+\infty,\quad 1\le q\le +\infty,
\]
unless noted otherwise.
Moreover, $q'$ denotes the H\"older conjugate of $q$, i.e., $1/q+1/q'=1$ if $1<q<+\infty$; $q'=1$ if $q=+\infty$; $q'=+\infty$ if $q=1$. 
The notion of \emph{weak space-time two-scale convergence} is defined by
 
\begin{defi}[Weak space-time two-scale convergence]
A sequence $(u_{\e})$ in $L^q(\Omega\times I)$ is said to \emph{weakly space-time two-scale converge} to a function $u$ in $L^q(\Omega\times I\times\square\times J)$ 
 as $\e\to 0_+$, if $(u_{\e})$ is bounded in $L^{q}(\Omega\times I)$ and it holds that
\begin{align}\label{wtsc}
 \lefteqn{\lim_{\e\to 0_+}\int_0^T\int_{\Omega}u_{\e}(x,t)\phi(x)b\left(\tfrac{x}{\e}\right)\psi(t)c\left(\tfrac{t}{\e^r}\right)\, dxdt}\\
 &= \vierint u(x,t,y,s)\phi(x)b(y)\psi(t)c(s)\, dZ\nonumber
\end{align}
for any $\phi\in C^{\infty}_{\rm c}(\Omega)$, $b\in C^{\infty}_{\mathrm{per}}(\square)$, $\psi\in C^{\infty}_{\rm c}(I)$ and $c\in C^{\infty}_{\rm per}(J)$.
Then we write 
\[
   u_{\e}\wtts u\quad \text{ in } L^q(\Omega\times\square\times I\times J),
\]
or simply $u_{\e}\wtts u$, unless any confusion way arise. 
\end{defi}

The following fact is well known and will  often be  used to discuss weak convergence of periodic test functions.

\begin{prop}[Mean-value property]\label{mean}
  Let $g\in L^q(\square\times J)$ and set $g_{\e}(x,t)=g(x/\e,t/\e^r)$ for $\e>0$. 
  Then for any bounded domain $\Omega\subset\R^N$ and any bounded interval $I\subset \R$, it holds that 
\begin{align}\label{mean1}
\left\{
\begin{array}{ll}
 g_{\e}\to \langle g(y,s)\rangle_{y,s} \ \text{ weakly in } L^q(\Omega\times I) &\text{ if }\ q\in[1,+\infty), \\
 g_{\e}\to \langle g(y,s)\rangle_{y,s}\ \text{ weakly star in } L^{\infty}(\Omega\times I) &\text{ if }\ q=+\infty 
\end{array}
\right.
\end{align}
as $\e\to 0_+$. Here $\langle g(y,s)\rangle_{y,s}$ is defined by 
\[
\langle g(y,s)\rangle_{y,s}:=\int_0^1\int_{\square}g(y,s)\, dyds.
\] 
\end{prop}
  
\begin{proof}
See~\cite[Theorem 2.6]{cd2000},~\cite[Proposition 2.10]{n2018}. 
\end{proof}
 
\begin{rmk}
\rm 
We note that $u_{\e}\wtts u$ in $L^{q}(\Omega\times I\times \square\times J)$ implies $u_{\e}\to \langle u(y,s)\rangle_{y,s}$ weakly in $L^q(\Omega\times I)$. Indeed, one can immediately check it by taking constant test functions $b \equiv 1$ and $c \equiv 1$ in \eqref{wtsc}.
\end{rmk}

The following theorem is concerned with weak two-scale compactness of bounded sequences in $L^q(\Omega\times I)$ (see~\cite{al1},~\cite{hol},~\cite{ln},~\cite{ng}).
\begin{thm}[Weak space-time two-scale compactness]\label{multicpt}
Let $q\in(1,+\infty]$. For any bounded sequence $(u_{\e})$ in $L^q(\Omega\times I)$, there exist a subsequence  $\e_n \to 0_+$  of $(\e)$  and a function $u$$\in$$L^q(\Omega\times I\times \square\times J ) $ such that 
\[
  u_{\e_n}\wtts u\quad \text{ in } L^q(\Omega\times I\times \square\times J) .
\] 
\end{thm}

Theorem \ref{multicpt} can be obtained as a corollary of more general result (see Lemma \ref{multi} below). To this end, 
we first set up a notion of \emph{admissible test functions}. 

\begin{defi}[Admissible test function]\label{admissible}
Let $q\in [1,+\infty]$ and let $X\subset L^{q'}(\Omega\times I\times \square\times J)$ be a separable normed space equipped with norm $\|\cdot\|_X$.
Then $(X,\|\cdot\|_X)$ is called an \emph{admissible test function space} {\rm (}for the weak space-time two-scale convergence in $L^q(\Omega\times I\times \square \times J)${\rm)}, if every $\Phi \in X$ satisfies that  $(x,t) \mapsto \Phi(x,t,\tfrac x \e, \tfrac t {\e^r})$ is measurable in $\Omega \times I$ for $\e > 0$, and moreover,  there exists a constant $C>0$ such that    
\begin{align}
\lim_{\e\to 0_+}\left\|\Phi\left(x,t,\tfrac{x}{\e},\tfrac{t}{\e^r}\right)\right\|_{L^{q'}(\Omega\times I)}  &=  \left\|\Phi(x,t,y,s)\right\|_{L^{q'}(\Omega\times I\times \square\times J)},\label{ad1} \\
 \left\|\Phi\left(x,t,\tfrac{x}{\e},\tfrac{t}{\e^r}\right)\right\|_{L^{q'}(\Omega\times I)} &\le  C\left\|\Phi(x,t,y,s)\right\|_{X}\quad \text{ for }\ \e>0\label{ad2}
\end{align}
for any $\Phi\in$ $X$. Moreover, each $\Phi\in$ $X$ is called an \emph{admissible test function} {\rm (}for the weak space-time two-scale convergence in $L^q(\Omega\times I\times \square \times J)${\rm)}.
\end{defi}

\begin{rmk}\label{rem2.5}
\rm
 \begin{itemize}
 \item[\rm(i)]  We stress that, in Definition \ref{admissible}, the constant $C>0$ is independent of $\Phi\in X$ and $\vep > 0$. Moreover, admissible test function spaces are not unique.

 \item[\rm (ii)]  If $q\in (1,+\infty]$, then one can always check \eqref{ad1} and \eqref{ad2} for functions of the form 
     \[\Phi(x,t,y,s)=\phi(x)b(y)\psi(t)c(s)\]
     for any $\phi\in C^{\infty}_{\rm c}(\Omega)$, $b\in C^{\infty}_{\mathrm{per}}(\square)$, $\psi\in C^{\infty}_{\rm c}(I)$ and $c\in C^{\infty}_{\rm per}(J)$ by setting
     $X=L^{2q'}(\Omega\times I\times \square\times J)$ equipped with norm $\|\cdot\|_{X}=\|\cdot\|_{L^{2q'}(\Omega\times I\times \square\times J)}$. Indeed, by Proposition \ref{mean}, it follows that
$$
\int_{\Omega}|\phi(x)|^{q'} \left|b \left(\tfrac{x}{\e}\right)\right|^{q'}\, dx\to \int_{\Omega}|\phi(x)|^{q'}\left(\int_{\square}|b(y)|^{q'}\, dy\right)\, dx
$$
as $\e\to 0$. Moreover, with the aid of H\"older's inequality and the periodicity of $b(y)$, we have
\begin{align*}
\left\|\phi(x)b\left(\tfrac{x}{\e}\right)\right\|_{L^{q'}(\Omega)} 
\le \left\|\phi\right\|_{L^{2q'}(\Omega)}\left\|b\left(\tfrac{x}{\e}\right)\right\|_{L^{2q'}(\Omega)} 
\le C(\Omega)\left\|\phi\right\|_{L^{2q'}(\Omega)}\left\|b\right\|_{L^{2q'}(\square)}
\end{align*}
for some constant $C(\Omega)>0$ depending only on $\Omega$. Indeed, setting $y=x/\e$, one observes that
\begin{align*}
\int_{\Omega}|b(\tfrac{x}{\e})|^{2q'}\, dx  =  \e^{N}\int_{\Omega/\e}|b(y)|^{2q'}\, dy
\le  \e^NN(\Omega/\e)\int_{\square} |b(y)|^{2q'}\, dy,
\end{align*}
where $N(\Omega/\e)$ denotes the minimum number of unit boxes covering $\Omega/\e:=\{\tfrac{x}{\e}\in\R^N \colon x\in\Omega\}$ for $\e>0$ and it is proportional to $\e^{-N}$.
     Hence one can choose a constant $C(\Omega)$ independent of $\e$ such that $\e^NN(\Omega/\e)\le C(\Omega)$.
     A similar observation also holds for $\psi(t)$ and $c(s)$. 
 \end{itemize}
\end{rmk}

\begin{rmk}[Measurability of admissible functions]\label{R:meas-osci}
\rm
In general, it is not true that $(x,t)\mapsto \Phi(x,t,\tfrac{x}{\e},\tfrac{t}{\e^r})$ is Lebesgue measurable in $\Omega\times I$, even if so is $\Phi$ in $\Omega\times I\times \square\times J$. 
However, the following facts can be used to compensate for the gap\/:
 \begin{itemize}
  \item[\rm(i)] If $\Phi$ is of class $L^1(\Omega\times I;C_{\rm per}(\square\times J))$,~then $(x,t)\mapsto \Phi(x,t,\tfrac{x}{\e},\tfrac{t}{\e^r})$ is Lebesgue measurable in $\Omega\times I$ for $\e>0$.
     Indeed, $\Phi$ is Carath\'eodry, i.e., measurable in $(x,t)$ and continuous in $(y,s)$.
     
  \item[\rm(ii)] Let $q\in (1,+\infty]$. For $\Phi\in L^{q'}(\Omega\times I;C_{\rm per}(\square\times J))$, it holds that
\begin{align*}
\lim_{\e\to 0_+} \int_0^T\int_{\Omega} \left| \Phi\left(x,t,\tfrac{x}{\e},\tfrac{t}{\e^r}\right) \right|^{q'}\, dxdt
&= \vierint  | \Phi\left(x,t,y,s\right)  | ^{q'}\, dZ,\\
 \int_0^T\int_{\Omega} \left| \Phi\left(x,t,\tfrac{x}{\e},\tfrac{t}{\e^r}\right) \right|^{q'}\, dxdt &\le  \int_0^T\int_{\Omega}\sup_{(y,s)\in\square\times J}\left|\Phi(x,t,y,s)\right|^{q'}\,dxdt.
\end{align*}
We refer the reader to~\cite[Theorems 1 and 2]{ln} for more details. Hence, $L^{q'}(\Omega\times I ; C_{\rm per}(\square \times J))$ is an admissible function space.
\end{itemize} 
\end{rmk}

Now, we have 
\begin{lem}[\noindent{\cite[Theorem 2.3]{hol}}]\label{multi}
Let $q\in  ( 1,+\infty]$ and let
$(u_{\e})$ be a bounded sequence in $L^q(\Omega\times I)$.
Let $(X,\|\cdot\|_X)\subset L^{q'}(\Omega\times I\times \square\times J)$ be an admissible test function space.
Then there exist a subsequence $\e_n\to 0_+$  of $(\e_n)$  and a function $u\in L^{q}(\Omega\times I\times \square\times J)$ such that
\begin{align*}
\lefteqn{
\lim_{\e_n\to 0_+}\int_0^T\int_{\Omega}u_{\e_n}(x,t)\Phi\left(x,t,\tfrac{x}{\e_n},\tfrac{t}{\e_n^r}\right)\, dxdt
}\\
 &= \vierint u(x,t,y,s)\Phi(x,t,y,s)\, dZ\nonumber
\end{align*}
for any $\Phi\in X$.
\end{lem}
 Theorem \ref{multicpt} can be obtained as a corollary by setting $\Phi(x,t,y,s)=\phi(x)b(y)\psi(t)c(s)$ for any $\phi\in C^{\infty}_{\rm c}(\Omega)$, $b\in C^{\infty}_{\mathrm{per}}(\square)$, $\psi\in C^{\infty}_{\rm c}(I)$ and $c\in C^{\infty}_{\rm per}(J)$ (see also (ii) of Remark \ref{rem2.5}). Let us move on to space-time two-scale compactness of gradients. In contrast to elliptic homogenization (in space only), boundedness of gradients $(\nabla u_\e)$ does not imply strong precompactness of $(u_\e)$ in any Lebesgue spaces. Indeed, the time-derivative $(\partial_t u_\e)$ may not be bounded (in particular for nonlinear diffusion). In~\cite[Theorem 3.1]{hol}, space-time two-scale weak precompactness for gradients is proved via the Aubin-Lions compactness theorem under certain boundedness of time-derivatives as well as gradients. Here we shall provide a slightly different version without assuming boundedness of time-derivatives with a proof for the convenience of the reader. 
\begin{thm}[Weak space-time two-scale compactness for gradients]\label{gradientcpt}\ 
Let $q\in (1,+\infty)$ and let $(u_{\e})$ be a bounded sequence in $L^{q}(I;W^{1,q}(\Omega))$ such that $u_\vep \wtts u$ in $L^q(\Omega \times I \times \square \times J)$ for a limit $u$.
Suppose that $u\in L^q(I\times \square\times J;  W^{1,q}  (\Omega))$. 
Let $(X,\|\cdot\|_X)\subset L^q(\Omega\times I\times \square\times J)$ be an admissible function space in the sense of Definition \ref{admissible}. Then there exist a subsequence $\e_n\to 0_+$ and a function $z\in L^q(\Omega\times I;L^{q}(J; W^{1,q}_{\mathrm{per}}(\square)/\R))$ such that 
\begin{align}\label{adgracpt}
 \lefteqn{\lim_{\e_n\to 0_+}\int_0^T\int_{\Omega} \nabla u_{\e_n}(x,t)\cdot \Phi\left(x,t,\tfrac{x}{\e_n},\tfrac{t}{\e_n^r}\right)\, dxdt}\\
  &= \vierint \left(\nabla u(x,t,y,s)+\nabla_y z(x,t,y,s)\right)\cdot \Phi(x,t,y,s)\, dZ\nonumber
\end{align}
for any admissible test function $\Phi\in X^N$. 
In particular, setting $\Phi(x,t,y,s)=\phi(x)  B  (y)\psi(t) c(s)$ for any $\phi\in C^{\infty}_{\rm c}(\Omega)$, $B\in [C^{\infty}_{\mathrm{per}}(\square)]^N$, $\psi\in C^{\infty}_{\rm c}(I)$ and $c\in C^{\infty}_{\rm per}(J)$, then we have 
\begin{align}\label{gradientwtts}
  \nabla u_{\e_n}\wtts \nabla u+\nabla_y z\quad \text{ in }\ [L^q(\Omega\times I\times \square\times J)]^N.
\end{align}
Here and henceforth, $\nabla$ and $\nabla_y$ denote gradients with respect to $x$ and $y$, respectively. Moreover, $B$ can be replaced by a scalar one $b \in C^\infty_{\rm per}(\square)$ and one of the other  test functions  can be vector-valued. 
\end{thm}

\begin{rmk}\label{ysindeprmk}
\rm
Let $q\in [1,+\infty]$. If $(u_{\e})$  is bounded in $L^q(\Omega\times I)$ and  converges to $u$ strongly in $L^{1}(\Omega\times I)$, then  $(u_{\e})$ weakly space-time two-scale converges to $u$ in $L^q(\Omega\times I \times \square \times J)$. 
Indeed,  by \eqref{mean1} of Proposition \ref{mean}, we obtain
\begin{align*}
\lim_{\e\to 0_+}\int_0^T\int_{\Omega} u_{\e}(x,t)\phi(x)b\left(\tfrac{x}{\e}\right)\psi(t)c\left(\tfrac{t}{\e^r}\right)\, dxdt 
  &= \int^T_0 \int_\Omega  u(x,t) \phi(x)\langle b(y)\rangle_y\psi(t)\langle c(s)\rangle_s\, dxdt \\
  &= \vierint u(x,t)\phi(x)b(y)\psi(t)c(s)\, dZ.
\end{align*}
Therefore the weak space-time two-scale limit of $(u_{\e})$ coincides with the strong limit $u=u(x,t)$,  which is  independent of $(y,s)\in \square\times J$.
\end{rmk}

To prove Theorem \ref{gradientcpt}, we use the following 
\begin{lem}\label{zhikov}
Let $q\in(1,+\infty)$ and let $(u_{\e})$  be a sequence in  $L^q(\Omega\times I)$ such that $u_{\e}\wtts u$ in $L^{q}(\Omega\times I\times \square\times J)$. 
Then it holds that
\begin{align*}
\lim_{\e\to 0_+}\int_0^T\int_{\Omega}u_{\e}(x,t)\Phi\left(x,t,\tfrac{x}{\e},\tfrac{t}{\e^r}\right)\, dxdt
 = \vierint u(x,t,y,s)\Phi(x,t,y,s)\, dZ\nonumber
\end{align*}
for any admissible test function $\Phi$ in the sense of Definition \ref{admissible}.
\end{lem}

\begin{proof}[Sketch of proof]
 Let $\Phi$ be an admissible test function. As in the proof of Proposition 2.7 in~\cite{hol}, where only an $L^2$ setting is treated, using the uniform convexity of Lebesgue spaces (by $1 < q < +\infty$) instead of a Hilbert structure, we can prove
\begin{equation}\label{Phi-wts}
\Phi(x,t,\tfrac{x}{\e},\tfrac{t}{\e^r})\wtts \Phi\quad \text{ in }\ L^{q'}(\Omega\times I\times \square\times J).
\end{equation}
 As a by-product, it is also proved under \eqref{ad1} along with \eqref{Phi-wts} that
\[
 \lim_{\e\to0_+}\int_0^T\int_{\Omega}\Phi(x,t,\tfrac{x}{\e},\tfrac{t}{\e^r}) \Psi_{\e}(x,t)\, dxdt  
  =
 \vierint \Phi(x,t,y,s)\Psi(x,t,y,s)\, dZ
\]
for all bounded sequence $(\Psi_{\e})$ in $L^{q}(\Omega\times I)$ satisfying $\Psi_{\e}\wtts \Psi$ in $L^{q}(\Omega\times I\times \square\times J)$  (see~\cite[Theorem 2.5]{hol} and~\cite[Lemma 4.4]{Z2003}). Then $\Phi(x,t,x/\e,t/\e^r)$ is often said to \emph{strongly} two-scale converge to $\Phi$ in $L^{q'}(\Omega\times I\times \square\times J)$.  Thus Lemma \ref{zhikov} is proved by setting $\Psi_{\e}=u_{\e}$.
\end{proof}

Now, we are ready to prove Theorem \ref{gradientcpt}.
\begin{proof}[Proof of Theorem \ref{gradientcpt}]
Thanks to Lemma \ref{zhikov}, it suffices to prove \eqref{adgracpt} only for test functions of the form
 \[
 \Phi\left(x,t,\tfrac{x}{\e},\tfrac{t}{\e^r}\right)=\phi(x)B\left(\tfrac{x}{\e}\right)\psi(t)c\left(\tfrac{t}{\e^r}\right)
 \]
 for all $\phi\in C_{\rm c}^{\infty}(\Omega)$, $B\in [C^{\infty}_{\mathrm{per}}(\square)]^N$, $\psi\in C_{\rm c}^{\infty}(I)$ and $c\in C^{\infty}_{\mathrm{per}}(J)$. 
By the boundedness of $(\nabla u_{\e})$ in $[L^q(\Omega\times I)]^N$ and Theorem \ref{multicpt}, 
there exist a subsequence $(\e_n)$ of $(\e)$ and $U\in [L^q(\Omega\times I\times \square\times J)]^N$ such that  
\begin{align*}
\lefteqn{\lim_{\e_n\to 0_+}  \int_0^T\int_{\Omega}\nabla u_{\e_n}(x,t)\cdot \phi(x)B\left(\tfrac{x}{\e_n}\right)\psi(t)c\left(\tfrac{t}{\e_n^r}\right)\, dxdt}\\
 &= \vierint U(x,t,y,s)\cdot \phi(x)B(y)\psi(t)c(s)\, dZ.
\end{align*}
On the other hand, choose $B\in [C^{\infty}_{\mathrm{per}}(\square)]^N$ as a solution to 
\begin{equation}\label{divassump}
\dv_y B(y) =0 \ \text{ in }\ \square.
\end{equation}
Using the assumption that $u_{\e}\wtts u$ in $L^q(\Omega\times I\times \square\times J)$ and $u\in L^{q}(I\times \square\times J; W^{1,q}(\Omega))$, we can also deduce that
\begin{align*}
\lefteqn{
\lim_{\e_n\to 0_+} \int_0^T\int_{\Omega}\nabla u_{\e_n}(x,t)\cdot\phi(x)B\left(\tfrac{x}{\e_n}\right)\psi(t)c\left(\tfrac{t}{\e_n^r}\right)\, dxdt
}\\
 &= -\lim_{\e_n\to 0_+}\int_0^T\int_{\Omega}u_{\e_n}(x,t) \Bigl[\nabla\phi(x)\cdot B\left(\tfrac{x}{\e_n}\right)+\tfrac{1}{\e_n}\phi(x)\underbrace{\nabla_y\cdot B\left(\tfrac{x}{\e_n}\right)}_{=\,0}\Bigl]\psi(t)c\left(\tfrac{t}{\e_n^r}\right)\, dxdt\\
 &= -\vierint u(x,t,y,s)\nabla\phi(x)\cdot B(y)\psi(t)c(s)\, dZ\\
 &= \vierint \nabla u(x,t,y,s)\cdot \phi(x)B(y)\psi(t)c(s)\, dZ. 
\end{align*}
Setting $W:=U-\nabla u\in [L^q(\Omega\times \square\times I\times J)]^N$, we have 
\begin{equation*}
 \vierint W(x,t,y,s)\cdot \phi(x)B(y)\psi(t)c(s)\, dZ=0, 
\end{equation*}
provided that $\mathrm{div}_y B = 0$ in $\square$. Furthermore, the arbitrariness of $\phi\in C_{\rm c}^{\infty}(\Omega)$, $\psi\in C_{\rm c}^{\infty}(I)$ and $c\in C^{\infty}_{\mathrm{per}}(J)$ yields 
\begin{equation}\label{helmeq}
 \int_{\square}W(x,t,y,s)\cdot B(y)\, dy=0\quad \text{ a.e.~in }\ \Omega\times I\times J, 
\end{equation}
which implies that $W(x,t,\cdot,s)\in \mathbb{V}^{q}_{\rm pot}(\square):=\{\nabla w\in [L^q(\square)]^N\colon w\in W^{1,q}_{\mathrm{per}}(\square) /\R \}$ equipped with norm $\|\cdot\|_{\mathbb{V}^{q}_{\rm pot}(\square)}:=\|\cdot\|_{L^q(\square)}$. Indeed, let $\mathbb{V}^q_{\rm sol}(\square)$ be the closure of $\mathbb{C}^{\infty}_{\rm \sigma}(\square):=\{ W\in [C^{\infty}_{\rm per}(\square)]^{N}\colon \dv\, W=0 \text{ in } \square\}$ in $[L^q(\square)]^N$. As in Theorem 2 of~\cite{FM1977}, one can ensure that $[L^q(\square)]^N=\mathbb{V}^q_{\rm sol}(\square)\oplus \mathbb{V}^q_{\rm pot}(\square)$ and $\mathbb{V}^q_{\rm pot}(\square) = \mathbb{V}^{q'}_{\rm sol}(\square)^\perp$. 
Thus we see by \eqref{divassump} and \eqref{helmeq} that $W(x,t,\cdot,s)\in \mathbb{V}^{q}_{\rm pot}(\square)$. Hence there exists $z(x,t,\cdot,s)\in W^{1,q}_{\mathrm{per}}(\square)/\R$ such that
\begin{equation}\label{zexist}
\nabla_{y}z(x,t,\cdot,s)= W(x,t,\cdot,s) \quad \text{ a.e.~in }\ \square.
\end{equation}
Moreover, such a function $z$ satisfying \eqref{zexist} is uniquely determined. Indeed,  let   $z_j(x,t,\cdot,s)\in W^{1,q}_{\rm per}(\square)/\R$ (for $j=1,2$) satisfy \eqref{zexist}. Then we have 
\[
 \nabla_y(z_1-z_2)(x,t,\cdot,s)=0\quad \text{ a.e.~in }\ \square, 
\]
i.e., $(z_1-z_2)(x,t,\cdot,s)\equiv C(x,t,s)$  a.e.~in $\square$. Moreover, we find by $z_j(x,t,\cdot,s)\in W^{1,q}_{\rm per}(\square)/\R$, $j=1,2$, that 
\[
 C(x,t,s)=\int_{\square}[z_1(x,t,y,s) -z_2(x,t,y,s)]\, dy=0. 
\]
Thus  $z_1(x,t,\cdot,s)\equiv z_2(x,t,\cdot,s)$ a.e.~in $\square$. 
Therefore, if $z$ lies on $L^q(\Omega\times I;L^{q}(J;W^{1,q}_{\mathrm{per}}(\square)/\R))$, then we can conclude that
\begin{equation*}
\nabla_y u_{\e_n}\wtts U=\nabla u+\nabla_{y}z \quad \text{ in }\ [L^q(\Omega\times I\times \square\times J)]^N.
\end{equation*}
Hence it remains to prove $z\in L^q(\Omega\times I;L^{q}(J;W^{1,q}_{\mathrm{per}}(\square)/\R))$,  and it will be done in  the following lemma.
\end{proof}

\begin{lem}
The function $z(x,t,\cdot,s)\in W^{1,q}_{\rm per}(\square)/\R$ appeared in \eqref{zexist} belongs to $L^q(\Omega\times I;L^{q}(J;W^{1,q}_{\mathrm{per}}$$(\square)/\R))$. 
\end{lem}

\begin{proof}
Let $W=U-\nabla u\in [L^q(\Omega\times I\times \square\times J)]^N \simeq [L^q(\Omega\times I \times J;L^q(\square))]^N$ be the function defined in the proof of Theorem \ref{gradientcpt}.
Since $W(x,t,\cdot, s)$ belongs to the closed subspace $\mathbb{V}^{q}_{\rm pot}(\square)$ of $[L^q(\square)]^N$ for a.e.~$(x,t,\cdot,s)\in \Omega\times I\times J$,   
we see that $W\in L^q(\Omega\times I\times  J;\mathbb{V}^{q}_{\rm pot}(\square))$. 
Hence 
for each $n\in\N$ 
there exist a family of finite disjoint measurable sets $(A_j^n)_{j=1,2,\ldots,m_n}$ in $\Omega\times  I\times J$ and a family $(C_j^n)_{j=1,2,\ldots,m_n}$ in $\mathbb{V}^{q}_{\rm pot}(\square)$ for some $m_n\in \N$ such that 
\begin{equation}\label{w_n}
W_{n}(x,t,\cdot,s):=\sum_{j=1}^{m_n}C_j^n(\cdot)\chi_{A_j^n}(x,t,s)\to W(x,t,\cdot,s)\quad\text{ in }\ \mathbb{V}^{q}_{\rm pot}(\square)
\end{equation}
a.e.~in $\Omega\times I\times J$ as $n\to+\infty$ and 
\begin{equation}\label{202003101}
 \lim_{n\to \infty}\int_0^1\int_0^T\int_{\Omega}\|W_n(x,t,\cdot,s)-W(x,t,\cdot,s)\|_{\mathbb{V}^{q}_{\rm pot}(\square)}^q\, dxdtds= 0
\end{equation}
(see~\cite[p21]{barbu}). Moreover, recalling that $C_j^n\in \mathbb{V}^{q}_{\rm pot}(\square)$, we can define $D_j^n\in W^{1,q}_{\rm per}(\square)/\R$ satisfying $\nabla_y D_j^n= C_j^n$ a.e.~in $\square$. Furthermore, we put $z_n(x,t,y,s):=\sum_{j=1}^{m_n}D_j^n(y)\chi_{A_j^n}(x,t,s)$. Using these functions, we shall prove that $(x,t) \mapsto z(x,t,\cdot,\cdot)$ is strongly measurable over $\Omega \times I$ with values in $L^q(J;W^{1,q}_{\rm per}(\square) / \R)$.

To this end, we first prove that $z_n(x,t,\cdot,\cdot)$ and $z(x,t,\cdot,\cdot)$ belong to $L^{q}(J;W^{1,q}_{\rm per}(\square)/\R)$ for a.e.~$(x,t)\in \Omega\times I$.
Using Poincar\'e-Wirtinger's inequality along with \eqref{w_n} and recalling by \eqref{zexist} that $W(x,t,\cdot,s)=\nabla_yz(x,t,\cdot,s)$  for a.e.~$(x,t,s)\in\Omega\times I\times J$,  we have
\begin{align}
 \|z(x,t,\cdot,s)-z_n(x,t,\cdot,s)\|_{W^{1,q}_{\mathrm{per}}(\square)/\R}
 &\le C\|\nabla_y z(x,t,\cdot,s)-\nabla_y z_n(x,t,\cdot,s)\|_{L^{q}(\square)}\label{mstep1}\\
 &= C\|W(x,t,\cdot,s)-W_n(x,t,\cdot,s)\|_{L^{q}(\square)}  \to 0\nonumber
\end{align}
a.e.~in $\Omega\times I\times J$ as $n\to+\infty$.
Hence the limit $z$ of $(z_n)$ is a $W^{1,q}_{\rm per}(\square)/\R$-valued strongly measurable function in $\Omega\times I\times J$, since so is $z_n:\Omega\times I\times J\to W^{1,q}_{\rm per}(\square)/\R$. Accordingly, we further see by Fubini's lemma that $\chi_{A_j^n}(x,t,\cdot)$ is measurable in $J$ for a.e.~$(x,t) \in \Omega \times I$. Thus the function $s \mapsto z_n(x,t,\cdot,s)$ is $W^{1,q}_{\rm per}(\square)/\R$-valued strongly measurable in $J$ for a.e.~$(x,t)\in\Omega\times I$. Thanks to \eqref{mstep1}, $z(x,t,\cdot,\cdot):J\to W^{1,q}_{\rm per}(\square)/\R$ turns out to be strongly measurable for a.e.~$(x,t)\in\Omega\times I$. Hence using Fubini's theorem again, we obtain
\begin{align*}
\int_0^1\|z_n(x,t,\cdot,s)\|_{W^{1,q}_{\mathrm{per}}(\square)/\R}^q\, ds
&\le C\int_0^1 \|\nabla_y z_n(x,t,\cdot,s)\|_{L^{q}(\square)}^q\, ds\\
&= C\int_0^1\|W_n(x,t,\cdot,s)\|_{L^{q}(\square)}^q\, ds\\
&\le C(m_n) \sum_{j=1}^{m_n}\|C_j^n\|_{L^{q}(\square)}^q\int_0^1| \chi_{A_j^n}(x,t,s)|^q\, ds <+\infty
\end{align*}
and
\begin{align*}
\int_0^1 \|z(x,t,\cdot,s)\|_{W^{1,q}_{\mathrm{per}}(\square)/\R}^q\, ds
&\le C\int_0^1 \|\nabla_y z(x,t,\cdot,s)\|_{L^{q}(\square)}^q\, ds\\
&= C\int_0^1 \|W(x,t,\cdot,s)\|_{L^{q}(\square)}^q\, ds<+\infty
\end{align*}
for $(x,t)\in\Omega\times I$. We next show that the functions $(x,t)\mapsto$$z_n(x,t,\cdot,\cdot)$, $z(x,t,\cdot,\cdot)$ are $L^{q}(J;W^{1,q}_{\rm per}(\square)/\R)$-valued strongly measurable functions in $\Omega\times I$. 
Let $A$ be any measurable set in $\Omega\times I\times J$.
By Fubini's lemma, for any $v\in L^{q'}(J)$, the function
$$
(x,t)\mapsto \int_0^1v(s)\chi_{A}(x,t,s)\, ds
$$
is Lebesgue measurable in $\Omega\times I$.
Hence $(x,t)\mapsto \chi_{A}(x,t,\cdot)$ is an $L^{q}(J)$-valued weakly measurable function in $\Omega\times I$.
Since $L^q(J)$ is separable, Pettis's theorem ensures that $\chi_{A}(x,t,\cdot)$ is an $L^{q}(J)$-valued strongly measurable function in $\Omega\times I$.
Therefore, $z_n=\sum_{j=1}^{m_n}D_j^n\chi_{A_j^n}$ turns out to be an $L^{q}(J;W^{1,q}_{\rm per}(\square)/\R)$-valued strongly measurable function in $\Omega\times I$.
Furthermore, Poincar\'e-Wirtinger's inequality and \eqref{202003101} yield, up to (not relabeled) subsequence,   
\begin{align*}
    \|z(x,t,\cdot,\cdot)-z_n(x,t,\cdot,\cdot)\|_{L^{q}(J;W^{1,q}_{\rm per}(\square)/\R)}
    &\le C\|\nabla_y z(x,t,\cdot,\cdot)-\nabla_y z_n(x,t,\cdot,\cdot)\|_{L^{q}(J;L^{q}(\square))}\\
    &= C\|W(x,t,\cdot,\cdot)-W_n(x,t,\cdot,\cdot)\|_{L^{q}(\square\times J)}
    \to 0
\end{align*}
a.e.~in $\Omega\times I$ as $n\to+\infty$.
Hence the limit $z$ of $(z_n)$ is also $L^{q}(J;W^{1,q}_{\rm per}(\square)/\R)$-valued strongly measurable in $\Omega\times I$. \\
\quad
Finally, we prove that $z$ belongs to $L^{q}(\Omega\times I;L^{q}(J;W^{1,q}_{\rm per}(\square)/\R))$. Employing Poincar\'e-Wirtinger's inequality again, we obtain
\begin{align*}
\int_0^T\int_{\Omega}
\|z(x,t)\|_{L^{q}(J;W^{1,q}_{\mathrm{per}}(\square)/\R))}^q\, dxdt
&\le C
\int_0^T\int_{\Omega}
\|\nabla_yz(x,t)\|_{L^{q}(J;L^{q}(\square)}^q\, dxdt\\
&=
C\|W\|_{L^q(\Omega\times I\times \square\times J)}^q
<+\infty ,
\end{align*}
which completes the proof.
\end{proof}

In order to identify homogenized matrices, we shall employ the notion of \emph{very weak} two-scale convergence defined by \eqref{vwconv} below, 
for which the test function $b(y)$ is in particular chosen in such a way as to have zero mean, i.e., $\langle b(y)\rangle_y=\int_{\square}b(y)\, dy =0$.  The following corollary is a modified version of~\cite[Corollary 3.3]{hol} and is also better suited for later analysis. 

\begin{cor}[Very weak two-scale convergence, cf.~\cite{flo1,flo2,flo3,flo4}]\label{veryweak}
In addition to the same assumptions as in Theorem \ref{gradientcpt}.
Suppose that $u=u(x,t)$ is independent of $(y,s)$. 
Then it holds that
\begin{align}
\lim_{\e_n\to 0_+} &\int_{0}^{T}\int_{\Omega}\frac{u_{\e_n}(x,t)-u(x,t)}{\e_n} \phi(x)b\left(\tfrac{x}{\e_n}\right)\psi(t)c\left(\tfrac{t}{\e_n^{r}}\right)\, dxdt
\label{vwconv}\\
 &= \vierint z(x,t,y,s) \phi(x)b(y)\psi(t)c(t)\, dZ\nonumber
\end{align}
for any $\phi\in C^{\infty}_c(\Omega)$, $b\in C^{\infty}_{\mathrm{per}}(\square)/\R$ {\rm (}i.e., $\langle b \rangle_y = 0${\rm )}, $\psi\in C^{\infty}_c(I)$ and $c\in C^{\infty}_{\mathrm{per}}(J)$.
\end{cor}

\begin{proof}
We first note that, for any test function $b\in C^{\infty}_{\rm per}(\square)/\R$, there exists a unique solution $w\in C^{\infty}_{\mathrm{per}}(\square)/\R$ to 
\begin{equation*}\label{divdiv}
\Delta_y w(y)=b(y) \ \text{ in }\ \square.
\end{equation*}
In what follows, we put $B:=\nabla_y w\in [C^{\infty}_{\rm per}(\square)/\R]^N$ (hence, $\dv_yB(y)=b(y)$). 
Then since $u=u(x,t)$ is independent of $(y,s)$, Theorem \ref{gradientcpt} and Proposition \ref{mean} yield
\begin{align*}
\lefteqn{
\vierint \nabla_y z(x,t,y,s)\cdot\phi(x)B(y)\psi(t)c(s)\, dZ}\\
&=\vierint \left(\nabla u(x,t)+\nabla_y z(x,t,y,s)-\nabla u(x,t) \right)\cdot\phi(x)B(y)\psi(t)c(s)\, dZ\nonumber\\
&=
\lim_{\e_n\to 0_+}\int_0^T\int_{\Omega}\nabla (u_{\e_n}-u)\cdot\phi(x)B\left(\tfrac{x}{\e_n}\right)\psi(t)c\left(\tfrac{t}{\e_n^r}\right)\, dxdt\nonumber\\
&=
-\lim_{\e_n\to 0_+}\int_0^T\int_{\Omega}(u_{\e_n}-u)\nabla\phi(x)\cdot B\left(\tfrac{x}{\e_n}\right)\psi(t)c\left(\tfrac{t}{\e_n^r}\right)\, dxdt\nonumber\\
&\quad-\lim_{\e_n\to 0_+} \int_0^T\int_{\Omega} \frac{u_{\e_n}-u}{\e_n} \phi(x)b\left(\tfrac{x}{\e_n}\right)\psi(t)c\left(\tfrac{t}{\e_n^r}\right)\, dxdt=:I_1+I_2.
\nonumber
\end{align*}
Then we claim that $I_1=0$.
Indeed, we can derive from the assumption of Theorem \ref{gradientcpt} (in particular, $u_{\e}\wtts u$ in $L^q(\Omega\times I\times \square\times J)$) and Proposition \ref{mean} that
\begin{align*}
\lefteqn{\lim_{\e_n\to 0_+}\int_0^T\int_{\Omega}(u_{\e_n}-u)\nabla\phi(x)\cdot B\left(\tfrac{x}{\e_n}\right)\psi(t)c\left(\tfrac{t}{\e_n^r}\right)\, dxdt}\\
 &=
   \lim_{\e_n\to 0_+}\int_0^T\int_{\Omega} u_{\e_n}(x,t)\nabla\phi(x)\cdot B\left(\tfrac{x}{\e_n}\right)\psi(t)c\left(\tfrac{t}{\e_n^r}\right)\, dxdt\\
 &\quad -
   \lim_{\e_n\to 0_+}\int_0^T\int_{\Omega}u(x,t)\nabla\phi(x)\cdot B\left(\tfrac{x}{\e_n}\right)\psi(t)c\left(\tfrac{t}{\e_n^r}\right)\, dxdt\\
 &=
   \int_0^T\int_{\Omega}u(x,t)\nabla\phi(x)\cdot \langle B(y)\rangle_y\psi(t)\langle c( s)\rangle_s\,  dxdt\\
 &\quad -
   \int_0^T\int_{\Omega}u(x,t)\nabla\phi(x)\cdot \langle B(y)\rangle_y\psi(t)\langle c(s)\rangle_s\, dxdt =0.
\end{align*}
Thus we conclude that
\begin{align*}
\lefteqn{\lim_{\e_n\to 0_+}\int_0^T\int_{\Omega}\frac{u_{ \e_n  }-u}{\e_n}\phi(x)b\left(\tfrac{x}{\e_n}\right)\psi(t)c\left(\tfrac{t}{\e_n^r}\right)\, dxdt}\\
 &= \vierint -\nabla_y z(x,t,y,s)\cdot\phi(x)B(y)\psi(t)c(s)\, dZ\\
 &= \vierint z(x,t,y,s)\phi(x)b(y)\psi(t)c(s)\, dZ,
\end{align*}
which completes the proof.
\end{proof}

\subsection{Maximal monotone operator}
In this subsection, we recall the notion of {\it maximal monotone operators} (see, e.g.,~\cite{barbu} for more details). 

\begin{defi}[Monotone operator]
Let $X$ and $X^{\ast}$ be a Banach space and its dual space.
A set-valued operator $A:X\to 2^{X^{\ast}}$ is said to be \emph{monotone}, if it holds that
\[
 \langle u-v,\xi-\eta\rangle_{X^{\ast}}\ge0\quad \text{ for all }\ [u,\xi],\ [v,\eta]\in G(A),
\]
where $G(A)$ denotes the graph of $A$, i.e., $G(A)=\{[u,\xi]\in X\times X^{\ast}\colon\xi\in Au\}$.
\end{defi}

\begin{defi}[Maximal monotone operator]
A monotone operator $A:X\to 2^{X^{\ast}}$ is said to be \emph{maximal monotone}, if the graph $G(A)$ of $A$ is maximal among monotone operators, i.e., any monotone operator $B:X\to 2^{X^{\ast}}$
whose graph $G(B)$ involves $G(A)$ coincides with $A$.
\end{defi}

The following proposition is often used for  identifying  weak limits of nonlinear terms.
\begin{prop}[Minty's trick]\label{trick}
 Let $X$ be a reflexive Banach space and let $A:X\to 2^{X^{\ast}}$ be a maximal monotone operator.
Suppose that $[u_n,v_n]\in G(A)$ satisfies
$$
u_n\to u\ \text{ weakly in } X,\quad  v_n\to v\ \text{ weakly in } X^{\ast}
$$
and 
$$
 \limsup_{n\to\infty}\ \langle v_n,u_n\rangle_{X}\le \langle v,u\rangle_{X}.
$$
 Then $[u,v]\in G(A)$. Furthermore, it holds that
$$
   \langle v_n,u_n\rangle_{X}\to \langle v,u\rangle_{X}.
$$
 \end{prop}

Finally, let us recall the notion of 
{\it subdifferentials} for convex functionals. 

\begin{defi}[Subdifferential operator]
Let $X$ and $X^{\ast}$ be a Banach space and its dual space,  respectively.  Let $\phi: X\to (-\infty,+\infty]$ be a proper {\rm(}i.e., $D(\phi)\neq\emptyset${\rm)} lower semicontinuous and convex functional with \emph{effective domain} 
\[
  D(\phi):=\{u \in X\colon  \phi(u)<+\infty\}.
\]
The \emph{subdifferential} operator $\partial\phi:X\to 2^{X^{\ast}}$ of $\phi$ is defined by
\[
  \partial\phi(u)=\{ \xi\in X^{\ast}\colon \phi(v)-\phi(u)\ge \langle \xi ,v-u\rangle_X\ \text{ for all } v\in D(\phi)\}
\]
with domain $D(\partial\phi):=\{u\in D(\phi)\colon  \partial\phi(u)\neq \emptyset\}$.
\end{defi}

Subdifferential operators form a subclass of maximal monotone operators.

\begin{thm}[Minty]\label{minty}
Every subdifferential operator is maximal monotone.
\end{thm}


\section{Well-posedness}\label{S:wp}

This section is devoted to proving Theorem \ref{T:wp}. To this end, let us first set $v_{\e}:=|u_{\e}|^{p-1}u_{\e}$ (equivalently, $u_{\e}=|v_{\e}|^{(1-p)/p}v_{\e}$) and rewrite \eqref{eq} as
\begin{equation}\label{anal}
 \left\{
   \begin{aligned}
    \partial_tv_{\e}(x,t)^{1/p}&=\dv\left( a\left(\tfrac{x}{\e},\tfrac{t}{\e^r}\right)\nabla v_{\e}(x,t) \right)+f_{\e}(x,t),   &&(x,t)\in \Omega\times I, \\
    v_{\e}(x,t)&=0,                        &&(x,t)\in \partial\Omega\times I , \\
    v_{\e}(x,0)^{1/p}&=u^0(x),                      &&x\in \Omega, 
   \end{aligned}
 \right.
\end{equation}
where $f_{\e}:I\to H^{-1}(\Omega)$ and $u^0\in H^{-1}(\Omega)$ are given data. 
Here and henceforth, we simply write $v_{\e}^{1/p}$ instead of $|v_{\e}|^{(1-p)/p}v_{\e}$, unless any confusion may arise. According to Definition \ref{D:sol}, it suffices to prove the well-posedness for \eqref{anal} in the following sense:
\begin{defi}[Weak solution of \eqref{anal}]
A function $v_{\e} = v_\e(x,t) : \Omega \times I \to \R$ is called a {\rm({\emph weak})} {\emph solution} to \eqref{anal}, if the following conditions are all satisfied\/{\rm:}
\begin{itemize}
 \item[(i)] $v_\e\in L^{2}(I;H^{1}_0(\Omega)) \cap L^{(p+1)/p}(\Omega\times I)$, $v_{\e}^{1/p}\in W^{1,2}(I;H^{-1}(\Omega))$ and $v_{\e}^{1/p}(t,0)\to u^0$ strongly in $H^{-1}(\Omega)$ as $t\to 0_+$, 
 \item[(ii)] it holds that,  for all $w\in H^1_0(\Omega)$,
  \begin{equation*}
    \left\langle \partial_t v_{\e}^{1/p}(t),w\right\rangle_{H^1_0(\Omega)}+B^{t}(v_{\e}(t),w) =\langle f_{\e}(t),w\rangle_{H^{1}_0(\Omega)} \quad \text{ for a.e.~} t\in I,
    \end{equation*}
    where $B^{t}(\cdot,\cdot)$ is a coercive continuous bilinear form in $H^1_0(\Omega)$ defined by
    \begin{equation}\label{bilinearform}
     B^{t}(v,w)=\int_{\Omega}a\left(\tfrac{x}{\e},\tfrac{t}{\e^r}\right)\nabla v\cdot \nabla w\ dx\quad \text{ for }\ v,w\in H^1_0(\Omega).
    \end{equation}
\end{itemize}
\end{defi}

Then we shall prove the following theorem, which is equivalent to Theorem \ref{T:wp}\/{\rm:}
\begin{thm}[Well-posedness for \eqref{anal}]\label{well-posedness}
Let $0 < p,r,\e < +\infty$ and let $a=[a_{ij}]_{i,j=1,2,\ldots,N}$ be an $N \times N$ symmetric matrix field satisfying \eqref{ellip} as well as $(x,t) \mapsto a_{ij}(x,t)  \in W^{1,1}( I;L^{\infty}(\Omega))$ for $i,j =1,2,\ldots,N$. Then for any $f_{\e} \in W^{1,2}(I;H^{-1}(\Omega)) \cap L^1( I ; L^2(\Omega))$ and $u^0\in L^2(\Omega) \cap L^{p+1}(\Omega)$, the Cauchy-Dirichlet problem \eqref{anal} admits a unique weak solution $v_{\e} = v_{\e}(x,t) : \Omega \times I \to \R$ such that 
\begin{align*}
v_{\e}  &\in  L^{\infty}_{\mathrm{loc}}((0,T];H^{1}_0(\Omega)) \cap C(\overline{I};L^{(p+1)/p}(\Omega)),\\
v_{\e}^{1/p} &\in W^{1,\infty}_{\mathrm{loc}}((0,T];H^{-1}(\Omega)) \cap C_{\rm w}(\overline{I};L^2(\Omega)) \cap C(\overline{I};L^{p+1}(\Omega)).
\end{align*}
Furthermore, the weak solution continuously depends on the initial datum in the following sense\/{\rm :} let $u^{0,1}$, $u^{0,2}\in H^{-1}(\Omega)$ and let $v_1$, $v_2$ be weak solutions of \eqref{anal} for  the  initial data $u^{0,1}$, $u^{0,2}$, respectively. Then there exists a constant $C_T\ge 0$ depending on $T$ but independent of $t$, $u^{0,1}$ and $u^{0,2}$ such that  
\begin{align}\label{energy}
\sup_{t\in \overline{I}}\left\|v_{1}(t)^{1/p}-v_{2}(t)^{1/p}\right\|_{H^{-1}(\Omega)}^2  \le C_T\|u^{0,1}-u^{0,2}\|_{H^{-1}(\Omega)}^2.
\end{align}
\end{thm}

\begin{proof}
The existence of weak solutions for \eqref{anal} satisfying some energy inequalities have already been proved (see~\cite[Theorem 3.2]{akagi} and note that $R(\partial_H \psi) = L^2(\Omega) \cap L^{p+1}(\Omega)$ for our setting), and therefore, we shall only prove uniqueness and continuous dependence of weak solutions on initial data. For each $t\ge0$, define a bounded linear operator $A^{t}:H^1_0(\Omega)\to H^{-1}(\Omega)$ by 
$$
\left\langle A^{t} v, w\right\rangle_{H^{1}_0(\Omega)}
=
\int_{\Omega}a\left(\tfrac{x}{\e},\tfrac{t}{\e^r}\right)\nabla v(x)\cdot\nabla w(x)\, dx\quad \text{ for }\ v,w\in H^{1}_0(\Omega). 
$$
Then $A^{t}$ is bijective from $H^1_0(\Omega)$ to $H^{-1}(\Omega)$ for each $t\in \overline{I}$.
Indeed, since $a(y,s)$ is uniformly elliptic (see \eqref{ellip}), by the Lax-Milgram theorem, for each $t\in\overline{I}$ fixed and any $g\in H^{-1}(\Omega)$, 
there exists a unique solution $v_g \in H^{1}_0(\Omega)$ (which also depends on $t$) of
$$
\int_{\Omega}a\left(\tfrac{x}{\e},\tfrac{t}{\e^r}\right)\nabla v_g(x)\cdot\nabla w(x)\ dx
= \langle g, w\rangle_{H^{1}_0(\Omega)}\quad  \text{ for all }\ w\in H^1_0(\Omega).
$$
Therefore, $A^{t}:H^1_0(\Omega)\to H^{-1}(\Omega)$ is surjective. 
As for the injectivity, let $v_1, v_2\in H^1_0(\Omega)$ be such that
$\langle A^{t}v_1, w\rangle_{H^1_0(\Omega)}=\langle A^{t} v_2, w\rangle_{H^1_0(\Omega)}$
for all $w\in H^1_0(\Omega)$. Setting
 $w=v_1-v_2$ and using
 Poincar$\acute{\mathrm{e}}$'s inequality
 along with \eqref{ellip}, we find that
$$
\|v_1-v_2\|_{L^2(\Omega)}^2\le C\|\nabla(v_1-v_2)\|_{L^2(\Omega)}^2\stackrel{\eqref{ellip}}{\le}
\frac{C}{\lambda}\int_{\Omega}
a\left(\tfrac{x}{\e},\tfrac{t}{\e^r}\right)\nabla (v_1-v_2)\cdot\nabla (v_1-v_2)\ dx
=0, 
$$
which implies $ v_1 = v_2 $. Thus $A^{t}:H^1_0(\Omega)\to H^{-1}(\Omega)$ turns out to be injective.
Hence, for each $t\in \overline{I}$, one can define the inverse mapping $(A^{t})^{-1}:H^{-1}(\Omega)\to H^1_0(\Omega)$ of $A^{t}$. 
Furthermore, due to the {symmetry} of the matrix $a(y,s)$, we have
\begin{align}\label{adjoint1}
\langle A^{t} u, \left(A^{t}\right)^{-1}h \rangle_{H^1_0(\Omega)}
&=\int_{\Omega}a\left(\tfrac{x}{\e},\tfrac{t}{\e^r}\right)\nabla u\cdot\nabla(A^{t})^{-1}h\ dx\\
&=\int_{\Omega}\tenchi a\left(\tfrac{x}{\e},\tfrac{t}{\e^r}\right)\nabla(A^{t})^{-1}h\cdot\nabla u\ dx\nonumber\\
&= \langle A^{t}\left(A^{t}\right)^{-1}h,u \rangle_{H^1_0(\Omega)}\nonumber\\
&=\langle h,u\rangle_{H^1_0(\Omega)} \quad \text{ for } \ u\in H^1_{0}(\Omega),\ h\in H^{-1}(\Omega).\nonumber
\end{align}
Moreover, we claim that
\begin{equation}\label{equivnorm}
 \lambda \|\nabla (A^{t})^{-1}\cdot\|_{L^2(\Omega)}\le\|\cdot\|_{H^{-1}(\Omega)}  \le  \|\nabla (A^{t})^{-1}\cdot\|_{L^2(\Omega)}.
\end{equation}
Indeed, one can check the first inequality by noting that 
\begin{align*}
\lambda\|\nabla (A^{t})^{-1}h\|_{L^2(\Omega)}^2
 &\stackrel{\eqref{ellip}}{\le} \int_{\Omega} a(\tfrac{x}{\e},\tfrac{t}{\e^r})\nabla  (A^{t})^{-1} h\cdot \nabla  (A^{t})^{-1} h\, dx\\
 &=\langle A^{t}(A^{t})^{-1}h, (A^{t})^{-1}h\rangle_{H^1_0(\Omega)}\\
 &=\langle  h, (A^{t})^{-1}h\rangle_{H^1_0(\Omega)}\\
 &\le\|h\|_{H^{-1}(\Omega)}\|\nabla (A^{t})^{-1}h\|_{L^2(\Omega)}\quad \text{ for }\ h\in H^{-1}(\Omega).     
\end{align*}
On the other hand, the second inequality follows from the fact that
\begin{equation}\label{rayleigh}
a(y,s)\xi\cdot\zeta\le |\xi||\zeta|\quad \text{ for all }\ \xi,\zeta\in \R^N
\end{equation} 
(see \eqref{ellip} and Remark \ref{abdd}). Indeed, we have
\begin{align*}
\|h\|_{H^{-1}(\Omega)} &= \sup_{\scriptstyle v\in H^1_0(\Omega)\atop \scriptstyle \|v\|_{H^1_0(\Omega)}\le 1}\langle A^{t} (A^{t})^{-1}h, v\rangle_{H^1_0(\Omega)}
= \sup_{\scriptstyle v\in H^1_0(\Omega)\atop \scriptstyle \|v\|_{H^1_0(\Omega)}\le 1}
\int_{\Omega} a(\tfrac{x}{\e},\tfrac{t}{\e^r})\nabla  (A^t)^{-1}  h\cdot \nabla v\, dx\\
&\stackrel{\eqref{rayleigh}}{\le} \sup_{\scriptstyle v\in H^1_0(\Omega)\atop \scriptstyle \|v\|_{H^1_0(\Omega)}\le 1} \|\nabla  (A^t)^{-1}  h\|_{L^2(\Omega)}\|\nabla v\|_{L^2(\Omega)}
= \|\nabla  (A^t)^{-1}  h\|_{L^2(\Omega)}\quad \text{ for }\ h\in H^{-1}(\Omega).
\end{align*}
Thus \eqref{equivnorm} is proved.
Therefore, for each $t\in \overline{I}$, $A^{t}$ is an isomorphism between $H^1_0(\Omega)$ and $H^{-1}(\Omega)$.

Now, we are ready to prove \eqref{energy}. Let $v_1$, $v_2$ be weak solutions of \eqref{anal} for initial data $u^{0,1}$, $u^{0,2}$  (i.e., $v_j^{1/p}|_{t=0} = u^{0,j} \in H^{-1}(\Omega)$ for $j=1,2$),  respectively. Set $G:=(A^{t})^{-1}(v_1^{1/p}-v_2^{1/p})$, i.e., $v_1^{1/p}-v_2^{1/p}=A^{t}G$. We shall derive the following inequality\/:
\begin{equation}\label{groncondi}
\frac{d}{dt}\langle A^{t}G,G\rangle_{H^1_0(\Omega)}\le C\|\partial_sa(\tfrac\cdot\e,\tfrac{t}{\e^r})\|_{L^{\infty}(\Omega)}\langle A^{t}G,G\rangle_{H^1_0(\Omega)}\quad \text{ a.e.~in }\ I
\end{equation}
for some constant $C\ge 0$. Then applying Gronwall's inequality to \eqref{groncondi}, we have
\begin{align*}
\lefteqn{\lambda\int_{\Omega}\left|\nabla\left(A^{t}\right)^{-1}\left(v_1^{1/p}-v_2^{1/p}\right)(x,t)\right|^2\, dx}\\
 &\stackrel{\eqref{ellip}}{\le}
 \langle A^{t}G(t),G(t)\rangle_{H^1_0(\Omega)}\nonumber\\
 &\stackrel{\eqref{groncondi}}\le \langle  A^0 G(0),G(0)\rangle_{H^1_0(\Omega)}\exp\left(C\int_0^T\|\partial_sa(\tfrac\cdot\e,\tfrac{t}{\e^r})\|_{L^{\infty}(\Omega)}\, dt\right) \nonumber\\
 &\stackrel{\eqref{ellip}}{\le} \left( \int_{\Omega}\left|\nabla\left( A^0 \right)^{-1}\left(u^{0,1}-u^{0,2}\right)(x)\right|^2\, dx \right) \exp\left(C\int_0^T\|\partial_sa(\tfrac\cdot\e,\tfrac{t}{\e^r})\|_{L^{\infty}(\Omega)}\, dt\right)\nonumber
\end{align*}
for all $t\in\overline{I}$. Thus by \eqref{equivnorm}, we obtain \eqref{energy}. In particular, if $u^{0,1}=u^{0,2}$, it then follows that $v_1^{1/p}(t)=v_2^{1/p}(t)$ in $H^{-1}(\Omega)$ for all $t\in \overline{I}$. 

Hence, it remains to prove \eqref{groncondi}. By subtraction, we have 
$$
  \partial_t(A^{t} G)+A^{t}(v_1-v_2) = 0\ \text{ in } H^{-1}(\Omega),\quad t\in I. 
$$    
Testing it by $G$ and employing \eqref{adjoint1}, we obtain
\begin{align*}\label{pfgc1}
 0 &=  
    \left\langle\partial_t\left(A^{t}G\right), G\right\rangle_{H^1_0(\Omega)} + \left\langle A^{t}(v_1-v_2), G\right\rangle_{H^1_0(\Omega)} \\
  &=
    \left\langle\partial_t\left(A^{t}G\right), G\right\rangle_{H^1_0(\Omega)} + \left\langle v_1^{1/p}-v_2^{1/p},v_1-v_2\right\rangle_{H^1_0(\Omega)}.\nonumber
\end{align*}
Hence the monotonicity of $v\mapsto v^{1/p}$ yields
\begin{equation}\label{unistep1}
0\ge \left\langle
\partial_t\left(A^{t}G\right), G\right\rangle_{H^1_0(\Omega)}
= \frac{d}{dt}
\left\langle
A^{t}G, G
\right\rangle_{H^1_0(\Omega)}
- \left\langle
A^{t}G, \partial_tG
\right\rangle_{H^1_0(\Omega)},
\end{equation}
and moreover, the second term of the right-hand side in \eqref{unistep1} is rewritten as 
\begin{align}\label{3.19}
\lefteqn{\left\langle A^{t}G, \partial_tG\right\rangle_{H^1_0(\Omega)}}\\
 &= \int_{\Omega}a\left(\tfrac{x}{\e},\tfrac{t}{\e^r}\right)\nabla G\cdot\nabla\partial_t G\, dx \nonumber\\
 &= \frac{1}{2}\int_{\Omega} \left[ a\left(\tfrac{x}{\e},\tfrac{t}{\e^r}\right)\nabla G\cdot\nabla\partial_t G + a\left(\tfrac{x}{\e},\tfrac{t}{\e^r}\right)\nabla \partial_tG\cdot\nabla G\right]\, dx\nonumber\\
 &= \frac{1}{2}\int_{\Omega}\partial_t \left[ a\left(\tfrac{x}{\e},\tfrac{t}{\e^r}\right)\nabla G\cdot\nabla G\right]\, dx - 
   \frac{1}{2}\int_{\Omega}\partial_ta\left(\tfrac{x}{\e},\tfrac{t}{\e^r}\right)\nabla G\cdot\nabla G\, dx\nonumber\\
 &= \frac{1}{2}\frac{d}{dt}\langle A^{t}G,G\rangle_{H^1_0(\Omega)} - \frac{1}{2\e^r}\int_{\Omega}\partial_sa\left(\tfrac{x}{\e},\tfrac{t}{\e^r}\right)\nabla G\cdot\nabla G\, dx.\nonumber
\end{align}
Here we used the {symmetry} of the matrix $a(y,s)$. 
By combining \eqref{unistep1} with \eqref{3.19}, we can derive
\begin{align*}
\frac{1}{2}\frac{d}{dt}\langle A^{t}G,G\rangle_{H^1_0(\Omega)}
 &\le -\frac{1}{2\e^r}\int_{\Omega}\partial_s a\left(\tfrac{x}{\e},\tfrac{t}{\e^r}\right)\nabla G\cdot\nabla G\, dx\\
 &\le \frac{\|\partial_sa(\tfrac\cdot\e,\tfrac{t}{\e^r})\|_{L^{\infty}(\Omega)}}{2\e^r}\int_{\Omega}|\nabla G|^2\, dx\\
 &\stackrel{\eqref{ellip}}{\le} \frac{\|\partial_sa(\tfrac\cdot\e,\tfrac{t}{\e^r})\|_{L^{\infty}(\Omega)}}{2\lambda\e^r}\int_{\Omega}a\left(\tfrac{x}{\e},\tfrac{t}{\e^r}\right)\nabla G\cdot\nabla G\, dx\\
&= \frac{C}{2}\|\partial_sa(\tfrac\cdot\e,\tfrac{t}{\e^r})\|_{L^{\infty}(\Omega)}\langle A^{t}G,G\rangle_{H^1_0(\Omega)},
\end{align*}
where $C:=1/(\lambda\e^r)$. Thus \eqref{groncondi} follows.
\end{proof}

\begin{rmk}[Uniform ellipticity of $a(y,s)$]\label{abdd}\rm
Note that \eqref{rayleigh} is shown by the Rayleigh-Ritz  variational  principle. Indeed, $a(y,s)$ can be represented by $a(y,s)=Q(y,s)^{-1}\Lambda Q(y,s)$ for some unitary matrix $Q(y,s)\in\R^{N\times N}$ and a diagonal matrix $\Lambda(y,s) \in \R^{N\times N}$ whose diagonal components  consist of  all the eigenvalues $\{\lambda_i(y,s)\}_{i=1,2,\ldots,N}$ of $a(y,s)$. Then the Rayleigh-Ritz  variational  principle along with \eqref{ellip} yields $\max_{1\le i\le N}\lambda_i(y,s)\leq 1$. Hence we obtain
\begin{align*}
a{}\xi\cdot\zeta
=Q{}^{-1}\Lambda{} Q{}\xi\cdot\zeta
=\Lambda{} Q{}\xi\cdot Q{} \zeta\le\max_{1\le i\le N}|\lambda_i{}||\xi||\zeta|\le |\xi||\zeta|
\quad \text{ for all }\ \xi,\zeta\in \R^N,
\end{align*}
which implies \eqref{rayleigh}.
\end{rmk}


\section{Uniform estimates and convergence}\label{S:est}

In this section, we shall derive uniform estimates for  $(v_{\e})$ and $(v_{\e}^{1/p})$  and discuss their convergence.
\begin{lem}[Uniform estimates]\label{bdd}
Let $0<p,r<+\infty$. In addition to the same assumptions as in Theorem \ref{well-posedness} for $\e > 0$, assume that $(f_\e)$ is a bounded sequence in $L^2(I;H^{-1}(\Omega))$  as $\e \to 0_+$.  For each $\e > 0$ let $v_{\e}\in L^2(I;H^1_0(\Omega))$ be the unique weak solution of \eqref{anal}.
Then the following {\rm(i)-(iii)}  hold true\/{\rm :}
\begin{enumerate}
 \item[(i)] $(v_{\e}) $ is bounded in $L^2(I;H^1_{0}(\Omega))\cap L^{\infty}(I;L^{(p+1)/p}(\Omega))$, and $(v_{\e}^{1/p})$ is bounded in $L^{\infty}(I;L^{p+1}(\Omega))$,
 \item[(ii)] $(\partial_tv_{\e}^{1/p})$ is bounded in $L^2(I;H^{-1}(\Omega))$,
 \item[(iii)]  as for  $p \in (0,1)$, if $(f_\e)$ is bounded in $L^1(I;L^2(\Omega))$ as $\e \to 0_+$, then $(v_{\e}^{1/p})$ is bounded in $L^{\infty}(I;L^2(\Omega))$.
\end{enumerate}
In addition, if $p\in (0,2)$, $u^0 \in L^{3-p}(\Omega)$ and $(f_{\e})$ is bounded in $L^1(I;L^{3-p}(\Omega))$ as $\e\to0_+$, then it holds that
\begin{enumerate}
 \item[(iv)] $(v_{\e}^{1/p})$ is bounded in $ L^{2}(I;H^1_0(\Omega))$. 
\end{enumerate}
\end{lem}

\begin{proof}
Recall the weak formulation of 
\eqref{anal},
\begin{equation}\label{weakform}
\left\langle \partial_t v_{\e}(t)^{1/p},{w}\right\rangle_{H^1_0(\Omega)}+B^{t}(v_{\e}(t),{w}) = \langle f_{\e}(t),{w}\rangle_{H^1_0(\Omega)},
\end{equation}
where $B^{t}(\cdot,\cdot)$ is the bilinear form in $H^1_0(\Omega)$ defined by \eqref{bilinearform}, for all  ${w}\in H^1_0(\Omega)$ and a.e.~$t\in I$. We first prove (i).  
Put ${w}=v_{\e}(t)$ in \eqref{weakform}. Note that  
\begin{align*}
\left\langle \partial_t v_{\e}(t)^{1/p},v_{\e}(t)\right\rangle_{H^1_0(\Omega)}
&= \frac{1}{p+1}\frac{d}{dt}\|v_{\e}(t)^{1/p}\|^{p+1}_{L^{p+1}(\Omega)}.
    \end{align*}
Integrating both sides of \eqref{weakform} over $(0,t)$, we have
\begin{align}\label{4.39}
\frac{1}{p+1}\left(\|v_{\e}(t)^{1/p}\|^{p+1}_{L^{p+1}(\Omega)}-\left\|u^0\right\|^{p+1}_{L^{p+1}(\Omega)}\right)+\int_0^t B^{ \tau}(v_{\e}(\tau),v_{\e}(\tau)) \, d\tau
\\
=\int_0^t\int_{\Omega} f_{\e}(x,\tau)v_{\e}(x,\tau)\, dxd\tau, \nonumber
\end{align}
which along with H\"{o}lder and Young's inequalities implies that
\begin{align}\label{4.40}
\lefteqn{
\frac{1}{p+1}\|v_{\e}(t)^{1/p}\|_{L^{p+1}(\Omega)}^{p+1}+\lambda\int_0^t\|v_{\e}(\tau)\|_{H^1_0(\Omega)}^2\, d\tau
}\\
&\stackrel{\eqref{ellip}}{\le} \frac{1}{p+1}\|v_{\e}(t)^{1/p}\|_{L^{p+1}(\Omega)}^{p+1} + \int_0^t B^{ \tau}(v_{\e}(\tau),v_{\e}(\tau))\, d\tau\nonumber\\
&\stackrel{\eqref{4.39}}{\le} \frac{1}{p+1}\|u^0\|_{L^{p+1}(\Omega)}^{p+1}+\frac{\lambda}{2}\int_0^t\|v_{\e}(\tau)\|_{H^1_{0}(\Omega)}^2\, d\tau+\frac{1}{2\lambda}\int_0^t\|f_{\e}(\tau)\|_{H^{-1}(\Omega)}^2\, d\tau
\nonumber
\end{align}
for any $t\in\overline{I}$. By the boundedness of $(f_{\e})$ in $L^2(I; H^{-1}(\Omega))$, we obtain
$$
\|v_{\e}\|_{L^2(I;H^1_0(\Omega))}^2\le \frac{2}{\lambda(p+1)}\|u^0\|_{L^{p+1}(\Omega)}^{p+1}+\frac{1}{\lambda^2}\|f_{\e}\|_{L^2(I;H^{-1}(\Omega))}^2\le C,
$$
and hence, \eqref{4.40} yields
$$
\sup_{t \in  \overline{I} }\|v_{\e}(t)^{ 1/p}\|_{L^{p+1}(\Omega)}^{p+1}\le \|u^0\|_{L^{p+1}(\Omega)}^{p+1}+C.
$$
Finally, noting that $\|v_{\e}^{1/p}(t)\|_{L^{p+1}(\Omega)}^{p+1}=\|v_{\e}(t)\|_{L^{(p+1)/p}(\Omega)}^{(p+1)/p}$, we also obtain a uniform estimate for $(v_{\e})$ in $L^{\infty}(I;L^{(p+1)/p}(\Omega))$.

We next prove (ii).     
For any ${w}\in H^1_{0}(\Omega)$, \eqref{weakform} yields
\begin{align*}
\left\langle \partial_t v_{\e}(t)^{1/p},{w} \right\rangle_{H^1_{0}(\Omega)}
 &= -B^{t}(v_{\e}(t),{w})+\int_{\Omega}f_{\e}(x,t){w}(x)\, dx\\
 &\le \|{w}\|_{H^1_0(\Omega)} \left(\|v_{\e}(t)\|_{H^1_0(\Omega)}+\|f_{\e}(t)\|_{H^{-1}(\Omega)}\right) \quad \text{ a.e.~in } I,
\end{align*}
which implies
\begin{align*}
\left\|\partial_t v_{\e}(t)^{1/p}\right\|_{H^{-1}(\Omega)}
\le \|v_{\e}(t)\|_{H^1_0(\Omega)}+\|f_{\e}(t)\|_{H^{-1}(\Omega)}\quad \text{ a.e.~in } I.
\end{align*}
By (i) along with the boundedness of $(f_{\e})$ in $L^2(I;H^{-1}(\Omega))$, we have
\begin{align*}
\left\|\partial_t v_{\e}^{1/p}\right\|_{L^2(I;H^{-1}(\Omega))}
\le \left\|v_{\e}\right\|_{L^2(I;H^1_0(\Omega))}+\|f_{\e}\|_{L^2(I;H^{-1}(\Omega))} \le C.
\end{align*}

As for (iii), we shall prove the boundedness of $(v_{\e}^{1/p})$ in $L^{\infty}(I;L^{2}(\Omega))$. To this end, we may test \eqref{anal} by $v_\e^{1/p}$, which may not however lie on $H^1_0(\Omega)$. Hence we need some approximation. For $\ell\in (0,+\infty)\setminus\{1\}$, let us introduce approximate functions $(\xi_{\ell,n})$ in $C^1(\R)$ of the power function $s^{\ell} := |s|^{\ell-1}s$  by setting
\begin{alignat*}{4}
\xi_{\ell,n}(s) &=
\begin{cases}
s^{\ell} \ & \mbox{ for } \ s\in [0,n],\\
\ell n^{\ell-1}(s-n)+n^{\ell} \ & \mbox{ for } \ s > n
\end{cases}
&&\quad \text{ if }\ \ell\in(1,+\infty),\\
\xi_{\ell,n}(s) &=
\begin{cases}
\ell(\frac 2 n)^{\ell-1} s \ & \mbox{ for } \ s\in [0,\frac{1}{n}),\\
\rho_{\ell,n}(s) \ & \mbox{ for } \ s\in [\frac{1}{n},\frac{2}{n}],\\
s^{\ell} \ & \mbox{ for } \ s\in (\frac{2}{n},+\infty)
\end{cases}
&&\quad \text{ if }\ \ell\in(0,1),\\
\xi_{\ell,n}(s) &= \xi_{\ell,n}(-s) \quad \mbox{ for } \ s < 0,
\end{alignat*}
where $\rho_{\ell,n}$ is some smooth increasing function in $\R$ such that $|\rho_{\ell,n}(s)|\le |s|^{\ell}$ for $s\in\R$. In particular, we note that
\begin{align*}
 & \xi_{\ell,n}(0)=0,\quad 0\le \xi_{\ell,n}'\le C_{n} \ \mbox{ in } \R, \quad
   \xi_{\ell,n}(s)\to s^{\ell} \ \mbox{ as } n \to +\infty,\\
 & |\xi_{\ell,n}(s)| \leq |s|^\ell \ \mbox{ for } s \in \R
\end{align*}
for some constants $C_{n}\ge 0$ depending on $n$.  Then $\xi_{\ell,n}(v_{\e})\in L^2(I;H^1_0(\Omega))$ and $|\xi_{\ell,n}(v_{\e})|\le |v_{\e}|^{\ell}$. 
In what follows, we shall simply write $\xi_n$ instead of $\xi_{\ell,n}$, unless any confusion may arise.
Testing \eqref{weakform} by $\xi_{n}(v_{\e})$, we have
\begin{align*}
\int_{\Omega}a(\tfrac{x}{\e}, \tfrac{t}{\e^r})\nabla v_{\e}\cdot \nabla \xi_n(v_{\e})\, dx 
&= \int_{\Omega}\xi_n'(v_{\e})a(\tfrac{x}{\e}, \tfrac{t}{\e^r})\nabla v_{\e}\cdot \nabla v_{\e}\, dx \\
&\stackrel{\eqref{ellip}}{\ge} \lambda \int_{\Omega} |\sqrt{\xi_n'(v_{\e})}\nabla v_{\e}|^2\, dx = \lambda\|\nabla \eta_n(v_{\e})\|_{L^2(\Omega)}^2,  
\end{align*}
where $\eta_n$ is a  smooth Lipschitz continuous  function defined by $\eta_n(s)=\int_0^s\sqrt{\xi_n'(\sigma)}\, d\sigma$. Setting $\zeta_n(s):=\xi_n(s^p)$ and letting $\hat{\zeta}_n$ be a primitive function of $\zeta_n$ such that $\hat{\zeta}_n(0)=0$, i.e., $\hat{\zeta}_n(s)=\int_0^s\zeta_n(\sigma)\, d\sigma$ and $\hat{\zeta}_n'(s)=\zeta_n(s)$, for any $s\in \overline{I}$, we have
\begin{align*}
\int_0^t \langle \partial_tv_{\e}(\tau)^{1/p},\xi_n(v_{\e}(\tau)) \rangle_{H^1_0(\Omega)}\, d\tau
 &= \int_0^t\frac{d}{d\tau} \left(\int_{\Omega}\hat{\zeta}_n\left(v_{\e}(x,\tau)^{1/p}\right)\, dx \right) d\tau\\
 &= \int_{\Omega} \left[ \hat{\zeta}_n(v_{\e}(x,t)^{1/p}) - \hat{\zeta}_n(u^0) \right]\, dx.
\end{align*}

We shall verify additional regularity,
\begin{equation}\label{add_regu}
 v_\e^{ 1/p} \in L^\infty(I;L^{ p\ell+1}(\Omega)), \quad v_\e^{(\ell+1)/2} \in L^2(I;H^1_0(\Omega)),
\end{equation}
provided that $u^0\in L^{p\ell+1}(\Omega)$ and $f_{\e}\in L^{ 1}(I;L^{2/(2-p\ell)}(\Omega))$ for some $\ell\in (0,2/p)$  (i.e., $2-p\ell > 0$).  To do so,  recalling that $v_\e^{1/p} \in C_{\rm w}(\overline I;L^2(\Omega))$ (see Theorem \ref{well-posedness}),  we observe that
\begin{align}
\lefteqn{\int_{\Omega}  \hat{\zeta}_n(v_{\e}(x,t)^{1/p}) \, dx + \lambda \int_0^t \|\nabla \eta_n(v_{\e}(\tau))\|_{L^2(\Omega)}^2\, d\tau}\label{03041} \\
&\le \int_0^t\int_{\Omega} f_{\e}(x,\tau)\xi_n(v_{\e}(x,\tau))\, dxd\tau
 + \int_{\Omega} \hat{\zeta}_n(u^0(x)) \, dx\nonumber\\
&\le \int_0^t\|f_{\e}(\tau)\|_{L^{2/(2-p\ell)}(\Omega)}\|v_{\e}(\tau)^{\ell}\|_{L^{2/p\ell}(\Omega)}\, d\tau +
 \frac 1 {p\ell+1}\int_{\Omega} |u^0(x)|^{p\ell+1} \, dx\nonumber \\
&\le \left( \sup_{t\in I}\|v_{\e}(t)^{1/p}\|^{p\ell}_{L^{2}(\Omega)} \right) \|f_{\e}\|_{L^{1}(I;L^{2/(2-p\ell)}(\Omega))}+\frac{1}{p\ell+1}\|u^0\|_{L^{p\ell+1}(\Omega)}^{p\ell+1}. \nonumber
\end{align}
Here we used the fact that $0 \leq \hat\zeta_n(s) \leq \int^{|s|}_0 \sigma^{p\ell} \, \d \sigma \leq (p\ell+1)^{-1}|s|^{p\ell+1}$ for $s \in \R$.  Moreover, we note that the right-hand side of \eqref{03041} is independent of $n$. Furthermore, we claim that
\begin{equation}\label{fatou}
   \frac{1}{p\ell+1} \int_{\Omega} |v_{\e}(x,t)|^{(p\ell+1)/p} \, dx \le \liminf_{n\to\infty}\int_{\Omega}\hat{\zeta}_n(v_{\e}(x,t)^{1/p})\, dx.
 \end{equation}
Indeed,  the right-hand side of \eqref{fatou} is finite due to \eqref{03041}, and moreover,  we note that
 \begin{equation*}
 \hat{\zeta}_n(v_{\e}(x,t)^{1/p})=\int_0^{|v_{\e}(x,t)|^{1/p}}\xi_n(s^p)\, ds\ge 0
 \end{equation*}
 and
 \begin{align}\label{xiaeconv}
\lefteqn{
  \left|\hat{\zeta}_n(v_{\e}(x,t)^{1/p})-\frac{1}{p\ell+1}|v_{\e}(x,t)|^{(p\ell+1)/p}\right|
   = \left| \int^{v_\e(x,t)^{1/p}}_0 \left( \xi_n(s^p) - s^{p\ell} \right) \, \d s \right|
  }\\
  &\le 
    \int_{0}^{|v_{\e}(x,t)|^{1/p}}|\xi_n(s^p)-s^{p\ell}|\, ds \nonumber\\
  &=
   \int_0^{2/n}|\xi_n(s^p)-s^{p\ell}|\, ds + \chi_{[\,\cdot\,>n]}(|v_{\e}(x,t)|^{1/p})\int_{n}^{|v_{\e}(x,t)|^{1/p}}|\xi_n(s^p)-s^{p\ell}|\, ds\nonumber\\
  &\le
    \int_0^{2/n} s^{p\ell}\, ds + \chi_{[\,\cdot\, >n]}(|v_{\e}(x,t)|^{1/p})\int_{n}^{|v_{\e}(x,t)|^{1/p}}s^{p\ell}\, ds\nonumber\\
  &\le
    \frac{1}{p\ell+1}\left(\frac{2}{n}\right)^{p\ell+1}+\frac{\chi_{[\,\cdot\,>n]}(|v_{\e}(x,t)|^{1/p})}{p\ell+1}|v_{\e}(x,t)|^{(p\ell+1)/p} 
  \to 0\quad \text{ as }\  n\to+\infty\nonumber
 \end{align}
 for a.e.~$(x,t)\in \Omega\times I$. Here $\chi_{[\,\cdot\,>n]}$ denotes the characteristic function supported over the set $\{s\in\R\colon s>n\}$. Hence we observe that
$$
\hat{\zeta}_n(v_{\e}(x,t)^{1/p}) \to \frac{1}{p\ell+1}|v_{\e}(x,t)|^{(p\ell+1)/p}
$$
for a.e.~$(x,t) \in \Omega \times I$. Therefore \eqref{fatou} follows from Fatou's lemma. Therefore, we conclude that $v_{\e}(\cdot,t)^{ 1/p}\in L^{ p\ell+1}(\Omega)$ for a.e.~$t\in I$. Moreover, recalling \eqref{03041} again, we infer that  $v_\e^{ 1/p} \in L^\infty(I;L^{ p\ell+1}(\Omega))$. Furthermore, we see by \eqref{03041} that $(\eta_n(v_{\e}))$ is also bounded in $L^2(I;H^1_0(\Omega))$.
Repeating the same argument as in \eqref{xiaeconv}, we can verify that 
\begin{equation}
 \eta_{n}(v_{\e})\to \frac{2\sqrt{\ell}}{\ell+1}v_{\e}^{(\ell+1)/2} 
\quad \text{ a.e.~in }\ \Omega \times I,\label{03171}
\end{equation}
and therefore, we find by \eqref{03041} that 
\begin{equation}\label{gousei2}
 \eta_{n}(v_{\e})\to  \frac{2\sqrt{\ell}}{\ell+1}v_{\e}^{(\ell+1)/2}  \quad \text{ weakly in }\ L^2(I;H^1_0(\Omega))
\end{equation}
as $n\to +\infty$. Hence $v_{\e}^{(\ell+1)/2}$ turns out to lie on $L^2(I;H^1_0(\Omega))$. Moreover, using the weak lower semicontinuity of norms, one obtains
\begin{equation}\label{lsc}
\frac{4\ell}{(\ell+1)^2}\int^t_0 \|\nabla v_{\e}(\tau)^{(\ell+1)/2}\|_{L^2(\Omega)}^2 \, d\tau
\leq \liminf_{n \to +\infty} \int^t_0 \|\nabla \eta_n(v_{\e}(\tau)) \|_{L^2(\Omega)}^2 \, \d \tau.
\end{equation}
Thus \eqref{add_regu} has been proved.

We further note that
$$
\int_{\Omega} f_{\e}(x,t)\xi_n(v_{\e}(x,t))\, dx \le  \|f_{\e}(t)\|_{L^{p\ell+1}(\Omega)}\|v_{\e}(t)^{\ell}\|_{L^{(p\ell+1)/p\ell}(\Omega)}.
$$
 Hence  combining \eqref{03041} with \eqref{fatou} and \eqref{lsc}, for any $\nu>0$, we can take a constant $C_{\nu}>0$ such that
\begin{align}\label{regu}
\lefteqn{
\frac{1}{p\ell+1} \|v_{\e}(t)^{1/p}\|_{L^{p\ell+1}(\Omega)}^{p\ell+1}
+ \frac{4\ell\lambda}{(\ell+1)^2} \int^t_0 \|\nabla v_{\e}(\tau)^{(\ell+1)/2}\|_{L^2(\Omega)}^2 \, d\tau
}\\
&\le \int_0^t\|f_{\e}(\tau)\|_{L^{p\ell+1}(\Omega)}\|v_{\e}(\tau)^{\ell}\|_{L^{(p\ell+1)/p\ell}(\Omega)}\, d\tau + \int_{\Omega} \bigg( \int_0^{ |u^0(x)|}s^{p\ell}\, ds \bigg) \, dx \nonumber\\
 &\le \Big(\sup_{\tau\in I}\|v_{\e}(\tau)^{1/p}\|^{p\ell}_{L^{p\ell+1}(\Omega)}\Big) \|f_{\e}\|_{L^{1}(I;L^{p\ell+1}(\Omega))}+\frac{1}{p\ell+1}\|u^0\|_{L^{p\ell+1}(\Omega)}^{p\ell+1} \nonumber\\
 &\le \nu\sup_{\tau\in I}\|v_{\e}(\tau)^{1/p}\|^{p\ell+1}_{L^{p\ell+1}(\Omega)}
    + C_{\nu}\|f_{\e}\|^{p\ell+1}_{L^{1}(I;L^{p\ell+1}(\Omega))}+\frac 1{p\ell+1}\|u^0\|_{L^{p\ell+1}(\Omega)}^{p\ell+1}.\nonumber
\end{align}
Now, putting $\ell=1/p\in(0,2/p)$ (then $p\ell+1=2$) and $\nu=1/4$, we have 
\begin{align*}
 \frac{1}{2}\left\|v_{\e}(t)^{1/p}\right\|_{L^2(\Omega)}^2  
 &\le \frac{1}{4}\sup_{ \tau  \in I}\left\|v_{\e}( \tau )^{1/p}\right\|_{L^2(\Omega)}^2 +  C_{1/ 4} \|f_{\e}\|_{L^1(I;L^2(\Omega))}^2 + \frac12 \left\|u^0\right\|_{L^2(\Omega)}^2,
\end{align*}
provided that $u^0\in L^2(\Omega)$ and $f_\e \in L^1(I;L^2(\Omega))$. Thus taking a supremum of both sides for $t \in I$, we obtain
\begin{equation*}
\frac{1}{4}\sup_{ \tau  \in I}\left\|v_{\e}( \tau)^{1/p}\right\|_{L^2(\Omega)}^2
\le C\left(\left\|u^{0}\right\|_{L^2(\Omega)}^{2}+1\right),
\end{equation*}
whenever $(f_\e)$ is bounded in $L^1(I;L^2(\Omega))$.

Finally, in order to prove (iv) for $p\in (0,2)$, we choose $\ell=(2-p)/p\in (0,2/p)$  (then $p\ell+1=3-p$)  in \eqref{regu} and assume that $u^0 \in L^{3-p}(\Omega)$ and $(f_{\e})$ is bounded in $L^1(I;L^{3-p}(\Omega))$. Then we have
\begin{align}\label{regu2}
\lefteqn{
 \frac{1}{3-p} \|v_{\e}(t)^{1/p}\|_{L^{3-p}(\Omega)}^{3-p}
 + \lambda p(2-p) \int^t_0 \|\nabla v_\e(\tau)^{1/p}\|_{L^2(\Omega)}^2 \, d \tau
}\\
  &\le \nu\sup_{\tau \in I}\|v_{\e}(\tau)^{1/p}\|^{3-p}_{L^{3-p}(\Omega)}\nonumber + C_{\nu} \|f_{\e}\|^{3-p}_{L^{1}(I;L^{3-p}(\Omega))} + \frac1{3-p}\|u^0\|_{L^{3-p}(\Omega)}^{3-p},\nonumber 
\end{align}
which implies that $(v_{\e}^{1/p})$ is bounded in $L^{\infty}(I;L^{3-p}(\Omega)) \cap L^2(I;H^1_0(\Omega))$, provided that $u^0 \in L^{3-p}(\Omega)$ and $(f_\e)$ is bounded in $L^{1}(I;L^{3-p}(\Omega))$.  To be precise,  we  here also  used the assumption $f_{\e}\in L^1(I;L^{2/(2-p\ell)}(\Omega))  = L^1(I;L^{2/p}(\Omega))$  (in particular, for $p \in (0,1)$)  to prove the additional regularity \eqref{add_regu} for $\ell \in (0,2/p)$; however, it can be removed for $\ell=(2-p)/p$ under $f_\e \in L^1(I;L^{3-p}(\Omega))$. Indeed, we can also verify \eqref{add_regu} by applying the argument so far to approximations for \eqref{anal} with (bounded) regularizations of $f_\e$ and then by passing to the limit in \eqref{regu2}.
%
%
\end{proof}

 We further  have the following
\begin{lem}[Chain-rule formula]\label{L:chain_rule}
Under the same assumptions for {\rm (iv)} of Lemma \ref{bdd}, it holds that
\begin{equation}\label{gouseidiff}
 \nabla v_{\e}(x,t)^{1/p}= \frac{1}{p} |v_{\e}(x,t)|^{(1-p)/p} \nabla v_{\e}(x,t) \quad \text{ for a.e.~}\ (x,t)\in \Omega\times I 
\end{equation}
for $p \in (0,1)$, and
\begin{equation}\label{gouseidiff2}
 \nabla v_{\e}(x,t)=p |v_{\e}(x,t)|^{(p-1)/p} \nabla v_{\e}(x,t)^{1/p} \quad \text{ for a.e.~}\ (x,t)\in \Omega\times I
\end{equation}
for $p \in (1,2)$. 
\end{lem}

\begin{proof}
We shall give a proof only for the case $p \in (0,1)$; however, the case $p \in (1,2)$ can be also proved in an analogous way. Recall that $\|\nabla\eta_{n}(v_{\e})\|_{L^2(\Omega\times I)}\le C$, where $\eta_n$ is the same smooth function as in the proof of Lemma \ref{bdd}. We have
\[
 \|\eta_n'(v_{\e})\nabla v_{\e}\|_{L^2(\Omega\times I)}=\|\nabla\eta_{n}(v_{\e})\|_{L^2(\Omega\times I)}\le C.
\]
Since $\eta'_n(v_{\e})\to \sqrt{\ell} |v_{\e}|^{(\ell-1)/2}$ a.e.~in $\Omega$  for $\ell \in (1,2/p)$,  we also find that 
\[
  \eta'_n(v_{\e})\nabla v_{\e} \to \sqrt{\ell} |v_{\e}|^{(\ell-1)/2}\nabla v_{\e} \quad \text{ weakly in }\ L^2(\Omega\times I).
\]
Moreover, recalling \eqref{03171}, we find that, for any $\phi\in C_{\rm c}^{\infty}(\Omega\times I)$,
\begin{align*}
\lefteqn{
 \int^T_0 \int_{\Omega} \sqrt{\ell} |v_{\e}(x,t)|^{(\ell-1)/2}\nabla v_{\e}(x,t)\phi(x,t)\, dx dt
}\\
 &= \lim_{n\to\infty} \int^T_0 \int_{\Omega} \eta_n'(v_{\e}(x,t))\nabla v_{\e}(x,t)\phi(x,t)\, dx dt \\
 &= \lim_{n\to\infty} \int^T_0 \int_{\Omega} \nabla \eta_n(v_{\e}(x,t))\phi(x,t)\, dx dt\\
 &= \int^T_0 \int_{\Omega}  \frac{2\sqrt{\ell}}{\ell+1}\nabla v_{\e}(x,t)^{(\ell+1)/2}\phi(x,t)\, dx dt. \nonumber
\end{align*}
Here we used \eqref{gousei2} to derive the last equality. From the arbitrariness of $\phi\in C^{\infty}_{\rm c}(\Omega)$, $\sqrt{\ell} |v_{\e}(x,t)|^{(\ell-1)/2} \nabla v_{\e}(x,t)=\frac{2\sqrt{\ell}}{\ell+1}\nabla v_{\e}(x,t)^{(\ell+1)/2}$ for a.e.~$(x,t)\in\Omega\times I$,
which along with the choice $\ell=(2-p)/p  \in (1,2/p)$ implies \eqref{gouseidiff}.
\end{proof}

By Lemma \ref{bdd}, one can obtain
\begin{lem}[Weak and strong convergences of $(v_{\e})$]\label{strongconvofsol}
Under the same assumptions as in Theorem \ref{well-posedness}, suppose that $(f_{\e})$ is bounded in $L^2(I;H^{-1}(\Omega))$ as $\e\to 0_+$. 
  In addition, for $p\in (0,1)$, assume  the boundedness of $(f_{\e})$ in $L^1(I;L^2(\Omega))$. Let $v_{\e}\in L^2(I;H^1_0(\Omega))$ be the unique weak solution of \eqref{anal}. Then there exist a subsequence $(\e_n)$ of $(\e)$ and $v_{0}\in L^2(I;H^1_0(\Omega))  \cap L^\infty(I;L^{(p+1)/p}(\Omega))$ such that
\begin{alignat}{4}
v_{\e_n} &\to v_{0} \quad &&\text{ weakly in } L^2(I;H^1_0(\Omega)),\label{strongconvsol2}\\
v_{\e_n}^{1/p} &\to v_{0}^{1/p} \quad && \text{ strongly in } C(\overline{I};H^{-1}(\Omega)),\label{strongconvsol3}\\
\partial_tv_{\e_n}^{1/p} &\to \partial_tv_{0}^{1/p} \quad  &&\text{ weakly in }\ L^2(I;H^{-1}(\Omega)), \label{strongcon2sol4}\\
v_{\e_n} &\to v_{0} \quad &&\text{ strongly in } L^\rho(I;L^{(p+1)/p}(\Omega)),\label{strongconvsol4}\\
v_{\e_n}^{1/p} &\to v_{0}^{1/p} \quad &&\text{ strongly in } L^\rho(I;L^{p+1}(\Omega))\label{strongconvsol5}
\end{alignat}
for any $\rho \in [1,+\infty)$.
\end{lem}

\begin{proof}
By (i) of Lemma $\ref{bdd}$, there exist a subsequence $(\e_n)$ of $(\e)$ and $v_{0}\in L^2(I;H^1_0(\Omega))$ such that \eqref{strongconvsol2} holds. We next prove \eqref{strongconvsol3}. The equi-continuity of $t\mapsto v_{\e_n}(t)^{1/p}$ in $H^{-1}(\Omega)$ follows from (ii) of Lemma $\ref{bdd}$. Indeed, 
we see that, for any $t, \tau\in \overline I$ with $t>\tau$, 
\begin{align*}
 \|v_{\e_n}(t)^{1/p}-v_{\e_n}(\tau)^{1/p}\|_{H^{-1}(\Omega)}
  &= \Bigl\|\int_\tau^t \partial_{ t} v_{\e_n}(\sigma)^{1/p}\, d\sigma\Bigl\|_{H^{-1}(\Omega)}\\
  &\le \int_\tau^t \| \partial_{ t} v_{\e_n}(\sigma)^{1/p} \|_{H^{-1}(\Omega)} \, d\sigma\\
  &\le
   \|\partial_{ t}v_{\e_n}^{1/p}\|_{L^2(I;H^{-1}(\Omega))} |t-\tau|^{1/2}.
\end{align*}
On the other hand, we can also derive the relative compactness of $(v_{\e_n}(t)^{1/p})$ on $H^{-1}(\Omega)$ for each $t\in \overline I$ from  (i) and  (iii) of Lemma \ref{bdd}, that is, 
$$
\sup_{t\in  \overline{I} }\left\|v_{\e_n}(t)^{1/p}\right\|_{L^2(\Omega)}< C,
$$
 along with the compact embedding $L^2(\Omega)\hookrightarrow H^{-1}(\Omega)$. 
Hence, by Ascoli's theorem, there exists $\xi\in C(\overline{I};H^{-1}(\Omega))$ such that, up to a (not relabeled) subsequence, 
\begin{equation}
 v_{\e_n}^{1/p}\to \xi\ \quad \text{ strongly in }\ C(\overline{I};H^{-1}(\Omega)).\label{strongconv2xi}
\end{equation}
Let us check $\xi=v_{0}^{1/p}$. 
Let $A:L^{(p+1)/p}(\Omega)\to L^{p+1}(\Omega)$ be an operator defined by 
\begin{equation}\label{appendix}
A(w)=w^{1/p}\quad \text{ for } w\in L^{(p+1)/p}(\Omega).
\end{equation}
Then $A$ turns out to be maximal monotone in $L^{(p+1)/p}(\Omega)\times L^{p+1}(\Omega)$ (see \S \ref{A:S:monotone} in Appendix).
 Here we claim that
 \begin{equation}\label{jeqxi1}
  v_{\e_n}^{1/p}\to \xi\quad  \text{ weakly in }\ L^{p+1}(\Omega\times I). 
 \end{equation}
Indeed, since $(v_{\e}^{1/p})$ is bounded in $L^{p+1}(\Omega\times I)$, there exists $\tilde{\xi}\in L^{p+1}(\Omega\times I)$ such that $v_{\e_n}^{1/p}\to \tilde{\xi}$ weakly in $L^{p+1}(\Omega\times I)$. Due to the uniqueness of weak limits, we find by \eqref{strongconv2xi} that $\xi=\tilde{\xi}$. One can similarly derive that
\begin{equation}\label{c:vLq}
v_{\e_n}\to v_0 \quad\mbox{ weakly in } L^{(p+1)/p}(\Omega\times I) 
\end{equation}
from (i) of Lemma \ref{bdd} along with \eqref{strongconvsol2}. Hence employing \eqref{strongconvsol2} and \eqref{strongconv2xi}, we deduce that
\begin{align}\label{0926}
  \lim_{\e_n\to 0_+}\int_{0}^{T}\left\langle v_{\e_n}(t)^{1/p},v_{\e_n}(t)\right\rangle_{L^{(p+1)/p}(\Omega)} \, dt 
 &=\lim_{\e_n\to 0_+}\int_{0}^{T}\left\langle v_{\e_n}(t)^{1/p},v_{\e_n}(t)\right\rangle_{H^{1}_0(\Omega)}\, dt\\
 &=
  \int_{0}^{T} \left\langle \xi(t),v_0(t)\right\rangle_{H^{1}_0(\Omega)} \, dt \nonumber\\
 &=
  \int_{0}^{T} \left\langle \xi(t), v_0(t)\right\rangle_{L^{(p+1)/p}(\Omega)} \, dt.\nonumber
\end{align} 
Proposition \ref{trick} along with \eqref{jeqxi1}, \eqref{c:vLq} and \eqref{0926} ensures that $\xi=v_{0}^{1/p}$. Thus \eqref{strongconvsol3} is proved. Furthermore, \eqref{strongcon2sol4} follows immediately from (ii) of Lemma \ref{bdd}.

We also observe by \eqref{0926} that
\begin{equation}\label{errorest1term}
 \|v_{\e_n}\|_{L^{(p+1)/p}(\Omega\times I)}^{(p+1)/p}\to \|v_{0}\|_{L^{(p+1)/p}(\Omega\times I)}^{(p+1)/p}\quad \text{  as  }\ \e_n\to 0_+.
\end{equation}
Since $v_{\e_n}\to v_{0}$ weakly in $L^{(p+1)/p}(\Omega\times I)$ and $\|\cdot\|_{L^{(p+1)/p}(\Omega\times I)}$ is uniformly convex,
we conclude that 
\[
v_{\e_n}\to v_{0}\quad \text{ strongly in }\ L^{(p+1)/p}(\Omega\times I).
\]
Hence \eqref{strongconvsol4} follows from the boundedness of $(v_{\e_n})$ in $L^\infty(I;L^{(p+1)/p}(\Omega))$ (see (i) of Lemma \ref{bdd}). Moreover, we can similarly show \eqref{strongconvsol5} by using \eqref{jeqxi1} with $\xi = v_0^{1/p}$ and by noting from \eqref{errorest1term} that
$$
 \|v_{\e_n}^{1/p}\|_{L^{p+1}(\Omega\times I)}^{p+1}\to \|v_{0}^{1/p}\|_{L^{p+1}(\Omega\times I)}^{p+1}\quad \text{  as  }\ \e_n\to 0_+
$$
(see also (i) of Lemma \ref{bdd}).
This completes the proof.
\end{proof}

\begin{remark}[Chain-rule formula for the limit]\label{R:chain_rule}
{\rm The relations in Lemma \ref{L:chain_rule} remain true for the homogenized limit $v_0$. Indeed, one can check \eqref{gouseidiff} (respectively, \eqref{gouseidiff2}) for $v_0$ by testing \eqref{gouseidiff} (respectively, \eqref{gouseidiff2}) with $\e = \e_n$ by an arbitrary smooth test function $\phi \in C^\infty_c(\Omega \times I)$ and then by passing to the limit as $\e_n \to 0_+$.}
\end{remark}

We close this section with the following
\begin{lem}[Weak and strong convergences of $(v_{\e}^{1/p})$ for $0<p<2$]\label{strongconv2}
Under the same assumptions for {\rm (iv)} of Lemma \ref{bdd}, let $v_{\e}\in L^2(I;H^1_0(\Omega))$ be the unique weak solution of \eqref{anal} for $0 < p < 2$ such that $v_{\e_n}\to v_0$ as in \eqref{strongconvsol2} and \eqref{strongconvsol4} as $\e_n\to 0_+$. 
Then $v_{0}^{1/p}$ belongs to $L^2(I;$$H^1_0(\Omega))$ and there exists a {\rm (}not relabeled{\rm\/)} subsequence of $(\e_n)$ such that 
\begin{alignat}{4}
 v_{\e_n}^{1/p} &\to v_{0}^{1/p} \quad &&\text{ weakly in } L^2(I;H^1_0(\Omega)), \label{strongcon2sol2}\\
 v_{\e_n}^{1/p} &\to v_{0}^{1/p} \quad &&\text{ strongly in } L^{2}(\Omega\times I). \label{strongcon2sol3}
\end{alignat}
\end{lem}

\begin{proof}
By (iv) of Lemma \ref{bdd}, there exists a (not relabeled) subsequence of $(\e_n)$ such that
\[
v_{\e_n}^{1/p}\to v_0^{1/p} \quad \text{ weakly in } L^2(I;H^1_0(\Omega)).
\]
Thus \eqref{strongcon2sol2} follows. Finally, we prove \eqref{strongcon2sol3}.
Since $(v_{\e_n}^{1/p})$ is bounded in 
\[
\mathcal{W}:=\left\{v\in L^2(I;H^1_0(\Omega))\colon \partial_t v\in L^2(I; H^{-1}(\Omega)) \right\}
\]
equipped with graph norm $\|\cdot\|_{\mathcal{W}}:=\|\cdot\|_{L^2(I;H^1_0(\Omega))}+\|\partial_t \,\cdot\,\|_{L^2(I; H^{-1}(\Omega))}$,
thanks to the compact embedding $H^1_0(\Omega)\hookrightarrow L^2(\Omega)$, 
the Aubin-Lions lemma (see,  e.g.,~\cite[Corollary 4]{simon}) yields that 
$$
v_{\e_n}^{1/p}\to v_{0}^{1/p}\quad \text{ strongly in } L^{2}(\Omega\times I),
$$
which completes the proof.
\end{proof}

\section{Proof of main results}\label{S:pf}

This section is devoted to proving Theorems \ref{thm1} and \ref{thm2}.

\subsection{Proof of Theorem \ref{thm1}}

In order to prove Theorem {\ref{thm1}, we first set up

\begin{lem}[Two-scale convergence of $(\nabla v_{\e})$]\label{lemma5.1}
Under the same assumptions as in Lemma \ref{strongconvofsol} for $0<p<+\infty$, let $v_{\e}\in L^2(I;H^1_0(\Omega))$ be the unique weak solution to \eqref{anal} and let $v_{0}$ be a weak limit of $(v_{\e_n})$ in $L^2(I;H^1_0(\Omega))$ as  a sequence  $\e_n\to 0_+$. 
Then there exist a {\rm (}not relabeled{\rm\/)} subsequence of $(\e_n)$ and $z\in L^2(\Omega\times I;L^{2}(J;H^{1}_{\mathrm{per}}(\square)/\R))$ such that
\begin{align}
 \nabla v_{\e_n}         &\wtts  \nabla v_{0}+\nabla_y z      &&\text{ in }\ [L^2(\Omega\times I\times \square\times J)]^N, \label{grawttssol}\\
 a\left(\tfrac{x}{\e_n},\tfrac{t}{\e^r_n}\right)\nabla v_{\e_n}  &\wtts  a(y,s)\left(\nabla v_{0}+\nabla_y z\right)  &&\text{ in }\ [L^2(\Omega\times I\times \square\times J)]^N.\label{grawtts}
\end{align}
\end{lem}
\begin{proof}
By (i) of Lemma \ref{bdd} and \eqref{strongconvsol4} of Lemma \ref{strongconvofsol} (see Remark \ref{ysindeprmk}), all the assumptions of Theorem \ref{gradientcpt} are satisfied. Hence the assertion \eqref{grawttssol} follows immediately from Theorem \ref{gradientcpt}. 

In order to prove \eqref{grawtts}, we shall show that
\[
\Phi\left(x,t,y,s\right):=\tenchi a(y,s)\phi(x)B(y)\psi(t)c(s)
\]
is an admissible test function (for the weak space-time two-scale convergence in $[ L^2(\Omega\times I \times \square \times J)]^N$) for each $\phi\in C^{\infty}_{\rm c}(\Omega)$, $B\in [ C^{\infty}_{\mathrm{per}}(\square)]^N$, $\psi\in C^{\infty}_{\rm c}(I)$ and $c\in C^{\infty}_{\mathrm{per}}(J)$.
First, we can check \eqref{ad1} by noting that
\begin{align*}
\lim_{\e\to 0_+} \left\|\Phi\left(x,t,\tfrac{x}{\e},\tfrac{t}{\e^r}\right)\right\|_{L^{2}(\Omega\times I)}^{2}
&= \lim_{\e\to 0_+}\int_0^T\int_{\Omega}\left|\tenchi a\left(\tfrac{x}{\e},\tfrac{t}{\e^r}\right)B\left(\tfrac{x}{\e}\right)c\left(\tfrac{t}{\e^r}\right)\right|^{2}\left|\phi(x)\psi(t)\right|^{2}\,dxdt\\
&\stackrel{\eqref{mean1}}{=} \int_0^T\int_{\Omega}\left\langle|\tenchi a(y,s)B(y)c(s)|^{2}\right\rangle_{y,s}\left|\phi(x)\psi(t)\right|^{2}\, dxdt\\
&= \left\|\Phi\left(x,t,y,s\right)\right\|_{L^{2}(\Omega\times I\times\square\times J)}^{2}.
\end{align*}
Here we also used Proposition \ref{mean}. We next set a separable space $X:=L^2(\square\times J;C(\overline{\Omega\times I})) \subset L^2(\Omega\times I\times \square \times J)$ equipped with norm $\|\cdot\|_X=\|\cdot\|_{L^2(\square\times J;C(\overline{\Omega\times I}))}$.
Then we have
\begin{align*}
\left\|\Phi\left(x,t,\tfrac{x}{\e},\tfrac{t}{\e^r}\right)\right\|_{L^{2}(\Omega\times I)}
&\le
\left\|\tenchi a\left(\tfrac{x}{\e},\tfrac{t}{\e^r}\right)B\left(\tfrac{x}{\e}\right)c\left(\tfrac{t}{\e^r}\right)\right\|_{L^2(\Omega\times I)}
\|\phi(x)\psi(t)\|_{C(\overline{\Omega\times I})}
\\
&\le C( \Omega,I)\|\tenchi a(y,s)B(y)c(s)\|_{L^2(\square\times J)}\|\phi(x)\psi(t)\|_{C(\overline{\Omega\times I})}
\\
&=C( \Omega, I) \left\|\Phi\left(x,t,y,s\right)\right\|_{ X }
\end{align*}
for some constant $C(\Omega, I) > 0$ (see (ii) of Remark \ref{rem2.5}). Thus \eqref{ad2} has been checked. Therefore by Lemma \ref{zhikov} along with \eqref{grawttssol}, we obtain
\begin{align*}
\lefteqn{
\lim_{\e_n\to 0_+} \int_0^T\int_{\Omega}a\left(\tfrac{x}{\e_n},\tfrac{t}{\e^r_n}\right)\nabla v_{\e_n}(x,t)\cdot \phi(x)B\left(\tfrac{x}{\e_n}\right)\psi(t)c\left(\tfrac{t}{\e_n^r}\right)\, dxdt 
}\\
&= \vierint a(y,s)\left(\nabla v_{0}(x,t)+\nabla_y z(x,t,y,s)\right)\cdot \phi(x)B\left(y\right)\psi(t)c\left(s\right)\, dZ.
\end{align*}
This completes the proof.
\end{proof}

We are now in a position to prove Theorem \ref{thm1}. 

\begin{proof}[Proof of Theorem \ref{thm1}]
Recalling $u_{\e}=v_{\e}^{1/p}$ (i.e., $v_\e = u_\e^p = |u_\e|^{p-1}u_\e$)  and setting $u_0:=v_0^{1/p}$, we have already proved \eqref{thm1pf1}--\eqref{thm1pf2} in Lemmas \ref{strongconvofsol} and \ref{lemma5.1}, respectively. 
So it remains to prove that $v_0$ solves \eqref{homeq} with a function $j_{\rm hom}$ given by \eqref{jhom}.  
By \eqref{weakform} , \eqref{strongcon2sol4} and \eqref{grawtts}, we deduce that, for any $\phi\in C^{\infty}_{\rm c}(\Omega)$ and $\psi\in C^{\infty}_{\rm c}(I)$,  
\begin{align*}
\lefteqn{
\int_0^T\int_{\Omega}f(x,t)\phi(x)\psi(t)\, dxdt
}\\
&= \lim_{\e_n\to 0_+}\int_0^T\int_{\Omega}f_{\e_n}(x,t)\phi(x)\psi(t)\, dxdt\nonumber\\
&\stackrel{\eqref{weakform}}{=} \lim_{\e_n\to 0_+}\int_0^T \left\langle \partial_t v_{\e_n}(t)^{1/p},\phi\right\rangle_{H^1_0(\Omega)} \psi(t) \, dt\nonumber\\
&\quad +\lim_{\e_n\to 0_+}\int_0^T\int_{\Omega} a \left(\tfrac{x}{\e_n},\tfrac{t}{\e_n^r} \right) \nabla v_{\e_n}(x,t)\cdot\nabla \phi(x)\psi(t)\, dxdt\nonumber\\
&\stackrel{\eqref{strongcon2sol4}, \eqref{grawtts}}{=} \int_0^T\left\langle\partial_t v_0(t)^{1/p},\phi\right\rangle_{H^1_0(\Omega)}\psi(t)\, dt\nonumber\\
 &\quad +\int_0^T\int_{\Omega}\Bigl\langle a(y,s)\bigl(\nabla v_0(x,t)+\nabla_y z(x,t,y,s)\bigl)\Bigl\rangle_{y,s}\cdot\nabla\phi(x)\psi(t)\, dxdt.\nonumber
\end{align*}
Set
\begin{equation}\label{vjhom}
j_{\rm hom}(x,t):=\int_0^1\int_{\square}a(y,s)\left(\nabla v_0(x,t)+\nabla_yz(x,t,y,s)\right)\, dyds.
\end{equation}
Then $v_0$ turns out to be a weak solution of
\begin{equation}\label{analhomeq}
 \left\{
   \begin{aligned}
    \partial_tv_{0}(x,t)^{1/p}&=\dv\, j_{\rm hom}(x,t)+f(x,t), &&(x,t)\in \Omega\times I,\\
    v_{0}(x,t)&=0, &&(x,t)\in \partial\Omega\times I,\\
    v_{0}(x,0)^{1/p}&=u^0(x), &&x\in \Omega 
   \end{aligned}
 \right.
\end{equation}
by the arbitrariness of $\phi$ and $\psi$ as well as the densely of $C^{\infty}_{\rm c}(\Omega)$ in $H^1_0(\Omega)$. Therefore $u_0=v_0^{1/p}$ satisfies \eqref{homeq}.
\end{proof}

\subsection{Proof of Theorem \ref{thm2}}

We next prove Theorem \ref{thm2}. To this end, we need  the following

\begin{lem}[Two-scale convergence of $(\nabla v_{\e}^{1/p})$ for $p\in(0,2)$]\label{lemma5.2}
Under the same assumptions of Lemma \ref{strongconv2}, let $v_{\e}\in L^2(I;H^1_0(\Omega))$ be the unique weak solution to \eqref{anal} for $0<p<2$ and let $v_{0}$ be a weak limit of $(v_{\e_n}) $ in $L^2(I;H^1_0(\Omega))$ as  a sequence  $\e_n\to 0_+$. Then there exist a {\rm (}not relabeled{\rm\/)} subsequence of $(\e_n)$ and a function $w\in L^{2}(\Omega\times I;L^{2}(J;H^{1}_{\mathrm{per}}(\square)/\R))$ such that 
\begin{equation}\label{grawtts2}
 \nabla v_{\e_n}^{1/p} \wtts \nabla v_{0}^{1/p}+\nabla_y w \text{ in } \left[L^{2}(\Omega\times I\times\square\times J)\right]^N.
\end{equation}
\end{lem}

\begin{proof}
Recall (iv) of Lemma \ref{bdd} and \eqref{strongcon2sol3} of Lemma \ref{strongconv2} (see Remark \ref{ysindeprmk}). Hence employing Theorem \ref{gradientcpt}, 
one can take a (not relabeled) subsequence of $(\e_n)$ and $w\in L^{2}(\Omega\times I;L^{2}(J;H^{1}_{\mathrm{per}}(\square)/\R))$ such that
\[
\nabla v_{\e_n}^{1/p}\wtts \nabla v_{0}^{1/p}+\nabla_y w\quad \text{ in }\ [L^{2}(\Omega\times\ I\times \square\times J)]^N.
\]
This completes the proof. 
\end{proof}

The following lemma will play a crucial role to  derive  a representation of the homogenized matrix $a_{\rm hom}$.
\begin{lem}\label{benri}
Under the same assumption as in Lemma \ref{lemma5.2}, it holds that
\begin{align}\label{mostest}
  \lefteqn{\lim_{\e_n\to 0_+}\e_n^{2-r}\int_0^T\int_{\Omega}\frac{v_{\e_n}(x,t)^{1/p}-v_0(x,t)^{1/p}}{\e_n}\phi(x)b\left(\tfrac{x}{\e_n}\right)
      \psi(t)\partial_s c\left(\tfrac{t}{\e_n^r}\right)\, dxdt}\\
&= \vierint a(y,s) \left( \nabla v_0(x,t) + \nabla_y z(x,t,y,s) \right) \cdot \phi(x) \nabla_yb(y)\psi(t)c(s)\, dZ\nonumber
\end{align}
for any $\phi\in C^{\infty}_{\rm c}(\Omega)$, $b\in C^{\infty}_{\mathrm{per}}(\square)$, $\psi \in C^{\infty}_{\rm c}(I)$ and $c\in C^{\infty}_{\mathrm{per}}(J)$.
\end{lem}
\begin{proof}
Subtracting weak forms for \eqref{anal} and \eqref{analhomeq} and testing it by $\phi(x) b(x/\e_n) \psi(t) c(t/\e_n^r)$, we observe that
\begin{align*}
\lefteqn{\int_0^T\int_{\Omega}\left(f(x,t)-f_{\e_n}(x,t)\right)\phi(x)b\left(\tfrac{x}{\e_n}\right)\psi(t)c\left(\tfrac{t}{\e_n^r}\right)\, dxdt}\\
&= -\int_0^T\int_{\Omega}\left(v_0(x,t)^{1/p}-v_{\e_n}(x,t)^{1/p}\right)\phi(x)b\left(\tfrac{x}{\e_n}\right)\partial_t\left[\psi(t)c\left(\tfrac{t}{\e_n^r}\right)\right]\, dxdt\\
&\quad +\int_0^T\int_{\Omega}\left[ j_{\rm hom}(x,t)-a\left(\tfrac{x}{\e_n},\tfrac{t}{\e_n^{r}}\right)\nabla v_{\e_n}(x,t) \right] \cdot\nabla\left[ \phi(x)b\left(\tfrac{x}{\e_n}\right)\right]\psi(t)c\left(\tfrac{t}{\e_n^r}\right)\, dxdt\\
&= -\int_0^T\int_{\Omega}\frac{v_{0}(x,t)^{1/p}-v_{\e_n}(x,t)^{1/p}}{\e_n}\phi(x)b\left(\tfrac{x}{\e_n}\right)\\
&\qquad \times \left[ \e_n\partial_t\psi(t)c\left(\tfrac{t}{\e_n^r}\right)+\e_n^{1-r}\psi(t)\partial_s c\left(\tfrac{t}{\e_n^r}\right) \right]\, dxdt\\
&\quad + \int_0^T\int_{\Omega}\left[ j_{\rm hom}(x,t)-a\left(\tfrac{x}{\e_n},\tfrac{t}{\e_n^{r}}\right)\nabla v_{\e_n}(x,t) \right] \cdot\left[ \nabla\phi(x)b\left(\tfrac{x}{\e_n}\right) + \phi(x)\e_n^{-1}\nabla_yb\left(\tfrac{x}{\e_n}\right)\right]\\
&\qquad \times\psi(t)c\left(\tfrac{t}{\e_n^r}\right)\, dxdt.
\end{align*}
Hence multiplying both sides by $\e_n$ and rearranging terms, we have
\begin{align*}
  \lefteqn{-\e_n^{2-r}\int_0^T\int_{\Omega}\frac{v_{0}(x,t)^{1/p}-v_{\e_n}(x,t)^{1/p}}{\e_n}\phi(x)b\left(\tfrac{x}{\e_n}\right)
      \psi(t)\partial_s c\left(\tfrac{t}{\e_n^r}\right)\, dxdt}\\
 &\quad
   -\int_0^T\int_{\Omega}a\left(\tfrac{x}{\e_n},\tfrac{t}{\e_n^{r}}\right)\nabla v_{\e_n}(x,t)
        \cdot \phi(x)\nabla_yb\left(\tfrac{x}{\e_n}\right)\psi(t)c\left(\tfrac{t}{\e_n^r}\right)\, dxdt\nonumber\\
 &=
   \e_n\int_0^T\int_{\Omega}\left(f(x,t)-f_{ \e_n }(x,t)\right)\phi(x)b\left(\tfrac{x}{\e_n}\right)\psi(t)c\left(\tfrac{t}{\e_n^r}\right)\, dxdt\nonumber\\ 
 &\quad
   +\e_n\int_0^T\int_{\Omega} (v_{0}(x,t)^{1/p}-v_{\e_n}(x,t)^{1/p})\phi(x)b\left(\tfrac{x}{\e_n}\right)\partial_t\psi(t)c\left(\tfrac{t}{\e_n^r}\right)\, dxdt\nonumber\\
 &\quad
   -\e_n\int_0^T\int_{\Omega}j_{\rm hom}(x,t)\cdot\nabla\phi(x)b\left(\tfrac{x}{\e_n}\right)\psi(t)c\left(\tfrac{t}{\e_n^r}\right)\, dxdt\\
 &\quad
   -\int_0^T\int_{\Omega}j_{\rm hom}(x,t)\cdot\phi(x)\nabla_yb\left(\tfrac{x}{\e_n}\right)\psi(t)c\left(\tfrac{t}{\e_n^r}\right)\, dxdt\\
 &\quad
   +\e_n\int_0^T\int_{\Omega}a\left(\tfrac{x}{\e_n},\tfrac{t}{\e_n^{r}}\right)\nabla v_{\e_n}(x,t)\cdot\nabla\phi(x)b\left(\tfrac{x}{\e_n}\right)\psi(t)c\left(\tfrac{t}{\e_n^r}\right)\, dxdt.
 \end{align*} 
Then, by \eqref{grawtts}, the second term of the left-hand side is convergent as $\e_n\to 0_+$, and moreover, the first three and last terms of the right-hand side vanish as $\e_n\to 0_+$ due to uniform estimates established in Lemma \ref{strongconvofsol}. Taking the limit of both sides as $\e_n \to 0_+$ and employing \eqref{mean1} of Proposition \ref{mean} and \eqref{grawtts} of Lemma \ref{lemma5.1}, we obtain
\begin{align*}
 \lefteqn{\lim_{\e_n\to 0_+}\e_n^{2-r}\int_0^T\int_{\Omega}\frac{v_{\e_n}(x,t)^{1/p}-v_0(x,t)^{1/p}}{\e_n}\phi(x)b\left(\tfrac{x}{\e_n}\right)
\psi(t)\partial_s c\left(\tfrac{t}{\e_n^r}\right)\, dxdt}\\
&\quad -\vierint a(y,s) \left( \nabla v_0(x,t) + \nabla_y z(x,t,y,s) \right) \cdot \phi(x) \nabla_yb(y)\psi(t)c(s)\, dZ 
\nonumber\\
 &= -\lim_{\e_n\to0_+}\int_0^T\int_{\Omega}j_{\rm hom}(x,t)\cdot\phi(x)\nabla_yb\left(\tfrac{x}{\e_n}\right)\psi(t)c\left(\tfrac{t}{\e_n^r}\right)\, dxdt\\
 &\stackrel{\eqref{mean1}}= -\int_0^T\int_{\Omega}j_{\rm hom}(x,t)\cdot\phi(x) \langle \nabla_yb(y)\rangle_y\psi(t)\langle c(s)\rangle_s\, dxdt\ =0.
\end{align*}
Here we also used the fact $\langle \nabla_yb(y) \rangle_y = 0 $ by the periodicity of $b(\cdot)$. 
\end{proof}

The following lemma provides a relation between the two functions $z = z(x,t,y,s)$ and $w = w(x,t,y,s)$ appeared in the weak space-time two-scale convergence of gradients $\nabla v_{\e}$ and $\nabla v_{\e}^ {1/p}$ (see Lemmas \ref{lemma5.1} and \ref{lemma5.2}), respectively, and it will play a  very  crucial role in the proof of Theorem \ref{thm2}, in particular, at the critical scale $r = 2$ to reveal strong interplay between microscopic and macroscopic structures through the homogenization for nonlinear diffusion.

\begin{lem}[Relation between correctors]\label{relationvw}
Under the same assumption as in Lemma \ref{lemma5.2}, 
let $z, w\in L^2(\Omega\times I;L^2(J;H^1_{\mathrm{per}}(\square)/\R))$ be functions appeared in Lemmas \ref{lemma5.1} and \ref{lemma5.2}, respectively. Then it holds that 
\begin{align*}
w(x,t,y,s) &= \frac{1}{p}|v_0(x,t)|^{(1-p)/p}z(x,t,y,s) \quad\text{ for a.e.~} (x,t,y,s)\in \Omega\times I\times \square\times J,
\end{align*}
if $0<p<1$, and moreover,
\begin{align*}
z(x,t,y,s) &=p|v_0(x,t)|^{(p-1)/p}w(x,t,y,s) \quad\text{ for a.e.~} (x,t,y,s)\in \Omega\times I\times \square\times J,
\end{align*}
if $1<p<2$.
\end{lem}

\begin{proof}[Proof of Lemma $\ref{relationvw}$]
In case $0 < p < 1$, for any $\phi\in C^{\infty}_{\rm c}(\Omega)$, $B\in [C^{\infty}_{\mathrm{per}}(\square)]^N$, $\psi\in C^{\infty}_{\rm c}(I)$ and $c\in C^{\infty}_{\mathrm{per}}(J)$, by \eqref{gouseidiff}, we find that
\begin{align}
\lefteqn{
\int_0^T\int_{\Omega} \nabla v_{\e_n}(x,t)^{1/p} \cdot \phi(x) B \left(\tfrac{x}{\e_n}\right)\psi(t) c\left(\tfrac{t}{\e_n^r}\right)\, dxdt
}\label{0702-1}\\
&= \int_0^T\int_{\Omega}\frac{1}{p}|v_{\e_n}(x,t)|^{(1-p)/p}\nabla v_{\e_n}(x,t)\cdot \phi(x)B\left(\tfrac{x}{\e_n}\right)\psi(t)c\left(\tfrac{t}{\e_n^r}\right)\, dxdt.\nonumber
\end{align}
By \eqref{strongcon2sol3} of Lemma \ref{strongconv2}, we deduce that 
\begin{align}\label{0702-2}
\lefteqn{
\left\||v_{\e_n}|^{(1-p)/p}-|v_{0}|^{(1-p)/p}\right\|_{L^{2/(1-p)}(\Omega\times I)}^{2/(1-p)}
}\\
&\le C \int_0^T\int_{\Omega} \left|v_{\e_n}(x,t)^{1/p}-v_{0}(x,t)^{1/p}\right|^{2}\, dxdt \to 0\quad\text{ as } \ \e_n\to 0_+.\nonumber
\end{align}
Moreover, by \eqref{grawttssol}, one has
\begin{equation}
\label{03181}
\nabla v_{\e_n}(x,t)\cdot B\left(\tfrac{x}{\e_n}\right)c\left(\tfrac{t}{\e_n^r}\right) \to \left\langle \left( \nabla v_0(x,t)+\nabla_yz(x,t,y,s) \right) \cdot B(y)c(s) \right\rangle_{y,s}
\end{equation} 
weakly in $L^2(\Omega\times I)$. Hence \eqref{0702-1}, \eqref{0702-2} and \eqref{03181} yield
\begin{align*}
\lefteqn{
\lim_{\e_n\to0_+} \int_0^T\int_{\Omega}\frac{1}{p}|v_{\e_n}(x,t)|^{(1-p)/p}\nabla v_{\e_n}(x,t)\cdot \phi(x)B\left(\tfrac{x}{\e_n}\right)\psi(t)c\left(\tfrac{t}{\e_n^r}\right)\, dxdt
}\nonumber\\
&= \int_0^T\int_{\Omega}\frac{1}{p}|v_{0}(x,t)|^{(1-p)/p}\phi(x)\psi(t)\left\langle \left(\nabla v_{0}(x,t)+\nabla_y z(x,t,y,s)\right) \cdot  B(y) c(s)\right\rangle_{y,s}\, dxdt\\
&= \vierint \frac{1}{p}|v_{0}(x,t)|^{(1-p)/p} \left(\nabla v_{0}(x,t)+\nabla_y z(x,t,y,s)\right)\cdot \phi(x) B(y)\psi(t)c(s)\, dZ.
\end{align*}
On the other hand, we have proved \eqref{grawtts2} in Lemma \ref{lemma5.2}. Thus combining \eqref{0702-1} with all these facts and employing \eqref{gouseidiff} for $v_0$ (see Remark \ref{R:chain_rule}), we have 
$$
\vierint \nabla_y \left(w(x,t,y,s)-\frac{1}{p}|v_{0}(x,t)|^{(1-p)/p}z(x,t,y,s)\right)\cdot \phi(x) B(y)\psi(t)c(s)\, dZ =0.
$$
Due to the arbitrariness of test functions, we obtain
\[
\nabla_y\left(w(x,t,y,s)-\frac{1}{p}|v_0(x,t)|^{(1-p)/p}z(x,t,y,s)\right)=0\quad \text{ a.e.~in }\ \Omega \times I \times \square \times J.
\]
Therefore $w(x,t,\cdot,s)-\frac{1}{p}|v_0(x,t)|^{(1-p)/p}z(x,t,\cdot,s)$ turns out to be identically equal to a constant $C$ a.e.~in $\square$. Since $w(x,t,\cdot,s)-\frac{1}{p}|v_0(x,t)|^{(1-p)/p}z(x,t,\cdot,s)$ has zero mean in $\square$,
the constant $C$ is determined by
\[
C=\int_{\square}\left(w(x,t,y,s)-\frac{1}{p}|v_0(x,t)|^{(1-p)/p}z(x,t,y,s)\right)\, dy=0,
\]
which completes the proof for the case $0<p<1$.

In case $1 < p < 2$, as in the other case, by \eqref{gouseidiff2} and \eqref{strongcon2sol3} of Lemma \ref{strongconv2}, we find that
\begin{align}\label{prszw1}
\lefteqn{
\int_0^T\int_{\Omega}\nabla v_{\e_n}(x,t) \cdot \phi(x)B\left(\tfrac{x}{\e_n}\right)\psi(t)c\left(\tfrac{t}{\e_n^r}\right)\, dxdt  }\\
&= \int_0^T\int_{\Omega} p|v_{\e_n}(x,t)|^{(p-1)/p}\nabla v_{ \e_n  }(x,t)^{1/p} \cdot  \phi(x)B\left(\tfrac{x}{\e_n}\right)\psi(t)c\left(\tfrac{t}{\e_n^r}\right)\, dxdt\nonumber
\end{align} 
and 
\begin{align}\label{prszw2}
\lefteqn{
\left\||v_{\e_n}|^{(p-1)/p}-|v_{0}|^{(p-1)/p}\right\|_{L^{2/(p-1)}(\Omega\times I)}^{2/(p-1)}
}\\
 &\le C\int_0^T\int_{\Omega}\left|v_{\e_n}(x,t)^{1/p}-v_{0}(x,t)^{1/p}\right|^{2}\, dxdt \to 0\quad\text{ as } \ \e_n\to 0_+.\nonumber
\end{align}
Using \eqref{prszw2} and \eqref{grawtts2} of Lemma \ref{lemma5.2}, we can derive from \eqref{prszw1} that
\begin{align*}\label{prszw1}
\lefteqn{
\lim_{\e_n\to 0_+}\int_0^T\int_{\Omega} p|v_{\e_n}(x,t)|^{(p-1)/p}\nabla v_{ \e_n  }(x,t)^{1/p} \cdot  \phi(x)B\left(\tfrac{x}{\e_n}\right)\psi(t)c\left(\tfrac{t}{\e_n^r}\right)\, dxdt
}\nonumber\\
&= \int_0^T\int_{\Omega}p|v_{0}(x,t)|^{(p-1)/p}\phi(x)\psi(t)\left\langle \left(\nabla v_{0}(x,t)^{1/p}+\nabla_y w(x,t,y,s) \right)\cdot  B(y) c(s)\right\rangle_{y,s}\, dxdt\\
&= \vierint p|v_{0}(x,t)|^{(p-1)/p} \left(\nabla v_{0}(x,t)^{1/p}+\nabla_y w(x,t,y,s) \right)\cdot \phi(x)B(y)\psi(t) c(s)\, dZ.
\end{align*} 
Recall \eqref{grawttssol} of Lemma \ref{lemma5.1} along with \eqref{prszw1} and use \eqref{gouseidiff2} for $v_0$ to observe that
\begin{align*}
 \vierint \nabla_y \left(p|v_{0}(x,t)|^{(p-1)/p}w(x,t,y,s)-z(x,t,y,s)\right)\cdot \phi(x) B(y)\psi(t)c(s)\, dZ =0.
\end{align*}
The rest of proof runs as in the case $0<p<1$.
\end{proof}

We are now ready to prove Theorem \ref{thm2}. In what follows, we denote by $\Vs^*$ the dual space of a closed subspace  of $H^1_{\rm per}(\square)$,
$$
\Vs := H^1_{\rm per}(\square) /\R
$$
equipped with  norm  $\|\cdot\|_\Vs = \|\nabla_y \cdot\|_{L^2(\square)}$. Then $\Vs$ and $\Vs^*$ are reflexive Banach spaces such that
\begin{equation}\label{triplet}
\Vs \hookrightarrow H \simeq H^* \hookrightarrow \Vs^*.
\end{equation}
 Here  $H := L^2(\square) / \R$ is a (pivot) Hilbert space which consists of functions $w \in L^2(\square)$ with zero mean (i.e., $\int_\square w(y) \, d y = 0$) and whose dual space can be identified with itself, with densely defined and continuous canonical injections.

\begin{proof}[Proof of Theorem \ref{thm2}]
We first note that \eqref{t2:hyp} along with the assumptions for Theorem \ref{thm1} is equivalent to those for (iv) of Lemma \ref{bdd}. In what follows, we always take $\phi\in C^{\infty}_{\rm c}(\Omega)$, $b\in C^{\infty}_{\mathrm{per}}(\square)/\R$ (i.e., $\langle b(y)\rangle_y=0$), $\psi\in C^{\infty}_{\rm c}(I)$ and $c\in C^{\infty}_{\rm per}(J)$ as test functions, unless otherwise noted, and Lemma \ref{benri} will play a key role. 

In case $0<r<2$, making use of Corollary \ref{veryweak} along with Lemma \ref{lemma5.2}, we find that the first term of \eqref{mostest} vanishes, that is, 
$$
\e_n^{2-r}\int_0^T\int_{\Omega}
\frac{v_{\e_n}(x,t)^{1/p}-v_0(x,t)^{1/p}}{\e_n}\phi(x)b\left(\tfrac{x}{\e_n}\right)\psi(t)\partial_s c\left(\tfrac{t}{\e_n^r}\right)\ dxdt
\to 0\quad \text{ as }\ \e_n\to 0.
$$
Here we used $\langle b(y)\rangle_y=0$ to employ Corollary \ref{veryweak}.
From the arbitrariness of $\phi\in C^{\infty}_{\rm c}(\Omega)$, $\psi\in C^{\infty}_{\rm c}(I)$ and $c\in C^{\infty}_{\mathrm{per}}(J)$, it follows from \eqref{mostest} that
\begin{equation}\label{zform}
\int_{\square}a(y,s)\bigl(\nabla v_{0}(x,t)+\nabla_y z(x,t,y,s)\bigl)\cdot\nabla_yb(y)\, dy
=0\quad\text{ for a.e.~} (x,t,s) \in \Omega\times I\times J
\end{equation}
for $b \in C^\infty_{\rm per}(\square) / \R$ (of course, it also holds for $b\in H^1_{\rm per}(\square)$ by density and $\nabla_y \langle b \rangle_y = 0$).

On the other hand, for each $k=1,2,\ldots, N$, let $\Phi_k\in L^2(J;H^1_{\mathrm{per}}(\square)/\R)$ be the unique weak solution to the following cell problem\footnote{Existence and uniqueness of weak solutions to \eqref{03191} can be checked by using the Lax-Milgram theorem. To this end, we set a bilinear form defined on  a Hilbert space $\mathcal{H}:=L^2(J;H^1_{\rm per}(\square)/\R)$.}\/{\rm :} 
\begin{equation}
-\dv_y \left(a(y,s) [\nabla_y\Phi_k(y,s)+e_{k}]\right) =0\ \text{ in }\ \T^N \times \T
,\label{03191}
\end{equation}
where $e_k$ is the $k$-th vector of the canonical basis of $\R^N$, 
and put
\begin{equation}\label{ahomek}
\tilde{z}(x,t,y,s)=\sum_{k=1}^N\partial_{x_k}v_0(x,t)\Phi_{k}(y,s).
\end{equation}
Then \eqref{zform} holds with $z$ replaced by $\tilde{z}$. Indeed, we see that
\begin{align*}
 \lefteqn{\int_0^1\int_{\square} a(y,s)\bigl(\nabla v_0(x,t)+\nabla_y\tilde{z}(x,t,y,s)\bigl)\cdot \nabla_yb(y)c(s)\, dyds}\\
  &=  \sum_{k=1}^N\partial_{x_k}v_0(x,t)\int_0^1\int_{\square} a(y,s)\bigl(\nabla_y\Phi_k(y,s)+e_k\bigl)\cdot \nabla_yb(y)c(s)\, dyds \stackrel{\eqref{03191}}{=} 0.\nonumber
\end{align*}
From the arbitrariness of $c\in C^{\infty}_{\mathrm{per}}(J)$, one can check \eqref{zform} with $z$ replaced by $\tilde{z}$.

We next claim that $z=\tilde{z}$. By \eqref{zform}, we have
\begin{align}\label{1121}
\int_{\square}a(y,s)\nabla_y\left( z(x,t,y,s)-\tilde{z}(x,t,y,s)\right)\cdot\nabla_yb(y)\, dy=0
\end{align}
for any $b \in C^\infty_{\rm per}(\square)/\R$ and a.e.~$(x,t,s) \in \Omega \times I \times J$. Put $b(\cdot)=(z-\tilde{z})(x,t,\cdot,s)\in H^1_{\rm per}(\square)/\R$ in \eqref{1121} (by density).
To be precise, we need substitute a smooth approximation of $(z-\tilde{z})(x,t,\cdot,s)\in H^{1}_{\rm per}(\square)/\R$ to $b$ and then take a limit.
Then by virtue of the Poincar\'e-Wirtinger inequality, it follows that
\begin{align}
0 &= \int_{\square}a(y,s)\nabla_y\bigl( z(x,t,y,s)-\tilde{z}(x,t,y,s)\bigl)
 \cdot\nabla_y\bigl(z(x,t,y,s)-\tilde{z}(x,t,y,s)\bigl)\, dy\label{ztilz}\\
 &\stackrel{\eqref{ellip}}{\ge} 
   {\lambda}\|\nabla_y\bigl(z(x,t,s)-\tilde{z}(x,t,s)\bigl)\|^2_{L^2(\square)}
 \ge \frac{\lambda}{C}  \|z(x,t,s)-\tilde{z}(x,t,s)\|^2_{L^2(\square)},\nonumber
\end{align}
which implies $z=\tilde{z}$.

One can take a constant matrix $a_{\rm hom}\in \R^{N\times N}$ such that
\begin{equation}
a_{\rm hom}\nabla v_0(x,t)=j_{\rm hom} \label{jhomahom}.
\end{equation} 
Indeed, recalling Theorem \ref{thm1}, we see that
\begin{align*}
j_{\rm hom}(x,t) &\stackrel{\eqref{jhom}}{=} \int_0^1\int_{\square}a(y,s)\left(\nabla v_0(x,t)+\nabla_yz(x,t,y,s)\right)\, dyds\\
&\stackrel{\eqref{ahomek}}{=} \int_0^1\int_{\square}a(y,s)\Bigl(\nabla v_0(x,t)+\sum_{k=1}^N\partial_{x_k}v_0(x,t)\nabla_y\Phi_{k}(y,s)\Bigl)\, dyds\\
&= \sum_{k=1}^N\Bigl(\int_0^1\int_{\square}a(y,s)\left(\nabla_y\Phi_{k}(y,s)+e_k\right)\, dyds\Bigl)\partial_{x_k}v_0(x,t).
\end{align*}
Define a constant matrix $a_{\rm hom}\in \R^{N\times N}$ by
\[
  a_{\rm hom}e_k=\int_0^1\int_{\square}a(y,s)\bigl(\nabla_y\Phi_k(y,s)+e_k\bigl)\, dyds,\quad k=1,2,\ldots, N,
\]
which is independent of $v_0$ and $z$. Then the homogenized flux $j_{\rm hom}$ can be written as \eqref{jhomahom}.
Moreover, the homogenized limit $v_0=v_0(x,t)$ is the unique weak solution of the Cauchy-Dirichlet problem
\begin{equation*}
 \left\{
  \begin{aligned}
   \partial_tv_{0}^{1/p} &=\dv\left( a_{\rm hom}\nabla v_{0} \right)+f &&\text{ in } \Omega\times I, \\
   v_{0}&=0 &&\text{ on } \partial\Omega\times I , \\
   v_{0}^{1/p}&= u^0 &&\text{ in } \Omega \times \{0\}.
  \end{aligned}
 \right.
\end{equation*}
Hence the limit of $(v_{\e_n})$ turns out to be independent of the choice of the subsequence $(\e_n)$, that is, it is no longer necessary to extract any subsequence.

In case $r=2$, the first term of \eqref{mostest} does not vanish any more as $\e_n\to 0_+$. Thanks to Corollary \ref{veryweak} along with Lemma \ref{lemma5.2}, we obtain
\begin{align*}
\lim_{\e_n\to 0_+} \int_0^T\int_{\Omega} \frac{v_{\e_n}(x,t)^{1/p}-v_0(x,t)^{1/p}}{\e_n} \phi(x)b\left(\tfrac{x}{\e_n}\right)\psi(t)\partial_s c\left(\tfrac{t}{\e_n^r}\right)\, dxdt
\\
= \vierint w(x,t,y,s) \phi(x)b(y)\psi(t)\partial_s c(s)\, dZ.
\end{align*}
Hence we see that
\begin{align}
\lefteqn{
\vierint \Big[ w(x,t,y,s)b\left(y\right)\partial_s c\left(s\right)
}\label{03194}\\
 &\quad -a(y,s)\bigl(\nabla v_{0}(x,t)+\nabla_yz(x,t,y,s)\bigl)\cdot\nabla_yb\left(y\right)c\left(s\right) \Bigl] \phi(x)\psi(t)\, dZ=0.\nonumber
\end{align}
In case $0<p<1$, we recall by Lemma \ref{relationvw} that $w  \in L^2(\Omega\times I;L^2(J;V))$ coincides with $(1/p)|v_0|^{(1-p)/p}z$. By the arbitrariness of $\phi\in C^{\infty}_c(\Omega)$ and $\psi\in C^{\infty}_c(I)$, we observe that
\begin{align}\label{wform2}
\lefteqn{
\int_0^1\int_{\square} \Bigl[\frac{1}{p}|v_0(x,t)|^{(1-p)/p}z(x,t,y,s)b\left(y\right)\partial_s c\left(s\right)
}\\
&-a(y,s)\bigl(\nabla v_{0}(x,t)+\nabla_yz(x,t,y,s)\bigl)\cdot\nabla_yb\left(y\right)c\left(s\right)\Bigl]\, dyds\nonumber =0
\end{align}
for a.e.~$(x,t) \in \Omega\times I$ and all $b \in \Vs$ and $c \in C^\infty_{\rm per}(J)$. We here claim that
\begin{equation}\label{cl:1}
\frac 1 p |v_0|^{(1-p)/p} z \in L^2(\Omega \times I ; W^{1,2}(J;\Vs^*) ).
\end{equation}
Indeed,  define $\xi(x,t,\cdot,\cdot) \in L^2(J;\Vs^*)$ by
\begin{align}\label{df:xi}
\lefteqn{
\int^1_0 \left\langle \xi(x,t,\cdot,s), w(\cdot,s) \right\rangle_\Vs \, d s 
}\\
&= \int^1_0 \int_\square a(y,s) \left( \nabla v_0(x,t) + \nabla_y z(x,t,y,s) \right) \cdot \nabla_y w(y,s) \, dy ds\nonumber
\end{align}
for $w \in L^2(J;V)$. Then $\xi : \Omega \times I \to L^2(J;V^*)$ turns out to be weakly measurable, and moreover, it is strongly measurable by Pettis's theorem.  Since $v_0 \in L^2(I;H^1_0(\Omega))$ and $\nabla_y z \in L^2(\Omega \times I \times \square \times J)$, one can verify that $\xi \in L^2(\Omega \times I ; L^2(J;\Vs^*))$. Moreover, we infer by \eqref{wform2} that
\begin{equation*}
\int^1_0 \frac 1 p |v_0(x,t)|^{(1-p)/p} z(x,t,\cdot,s) \partial_s c(s) \, ds = \int^1_0 \xi(x,t,\cdot,s) c(s) \, d s \ \mbox{ in } \Vs^*, 
\end{equation*}
which along with the arbitrariness of $c \in C^\infty_{\rm per}(J)$ implies
\begin{equation}\label{v0dz}
\frac 1 p |v_0(x,t)|^{(1-p)/p} \partial_s z(x,t,\cdot,s) = - \xi(x,t,\cdot,s) \ \mbox{ in } \Vs^*
\end{equation}
 in the distributional sense  for a.e.~$(x,t,s) \in \Omega \times I \times J$. Thus \eqref{cl:1} follows. We next claim that
\begin{equation}\label{cl:2}
\frac1p|v_0(x,t)|^{(1-p)/p}z(x,t,\cdot,1) = \frac1p|v_0(x,t)|^{(1-p)/p}z(x,t,\cdot,0) \ \mbox{ in } \ \Vs^*
\end{equation}
for a.e.~$(x,t) \in \Omega \times I$.  Indeed, setting $c \equiv 1$ in \eqref{wform2}, we see that
\begin{align*}
\lefteqn{
\left\langle \frac 1 p |v_0(x,t)|^{(1-p)/p} \left(z(x,t,\cdot,1) - z(x,t,\cdot,0) \right), b \right\rangle_\Vs
}\\
&\stackrel{\eqref{v0dz}}= \left\langle - \int^1_0 \xi(x,t,\cdot,s) \, d s, b \right\rangle_\Vs\\
&\stackrel{\eqref{df:xi}}= - \int^1_0 \int_\square a(y,s) \left( \nabla v_0 + \nabla_y z \right) \cdot \nabla_y b \, d y d s
\stackrel{\eqref{wform2}} = 0
\end{align*}
for all $b \in \Vs$ and a.e.~$(x,t) \in \Omega \times I$. From the arbitrariness of $b \in \Vs$, we obtain \eqref{cl:2}. 

On the other hand, for each $k=1,2,\ldots, N$, let $\Phi_k(x,t,\cdot,\cdot)\in L^2(J;H^1_{\mathrm{per}}(\square)/\R)$ be a weak solution to the following cell problem (see Lemma \ref{A:L:1} in \S \ref{A:S:reg} for more details)\/{:} 
\begin{equation*}
\frac{1}{p}|v_0(x,t)|^{(1-p)/p}\partial_s\Phi_k(x,t,y,s)=\dv_y \left(a(y,s)[\nabla_y\Phi_k(x,t,y,s)+e_{k}]\right) \ \mbox{ in } \T^N \times \T
\end{equation*}
and set $\tilde{z}$ as in \eqref{ahomek} with $\Phi_k = \Phi_k(x,t,y,s)$. Then we have \eqref{wform2} with $z=\tilde{z}$. Moreover, as in the last case, we find by \eqref{wform2} that
\begin{align}\label{1123}
0 
 &= -\frac{1}{p}|v_0(x,t)|^{(1-p)/p} \int_0^1 \int_\square \left(z(x,t,y,s)-\tilde{z}(x,t,y,s)\right) b(y) \partial_s c(s)\, dy ds\\
 &\quad +\int_0^1\int_{\square}a(y,s)\nabla_y\left( z(x,t,y,s)-\tilde{z}(x,t,y,s)\right)\cdot\nabla_yb(y)c(s)\, dyds. \nonumber
\end{align}
Put $(z-\tilde{z})(x,t,\cdot,\cdot)$ in place of the product between $b\in C^{\infty}_{\rm per}(\square)/ \R$ and $c\in C^{\infty}_{\rm per}(J)$ in \eqref{1123}.
Indeed, it is possible by density argument. 
By exploiting the Poincar\'e-Wirtinger inequality and the uniform ellipticity \eqref{ellip}, we deduce that
\begin{align}
0&\stackrel{\eqref{1123}}= -\frac 1 {2p} \underbrace{|v_0(x,t)|^{(1-p)/p}\int_0^1 \dfrac d {ds} \|(z-\tilde{z})(x,t,\cdot,s)\|_{L^2(\square)}^2 \, ds}_{= \, 0 \ \text{ by \eqref{cl:2}}} \label{03196}\\
 &\quad + \int_0^1\int_{\square}a(y,s)\nabla_y (z-\tilde{z})(x,t,y,s) \cdot\nabla_y (z-\tilde{z})(x,t,y,s)\, dyds \nonumber\\
 &\ge
   \lambda\|\nabla_y(z-\tilde{z})(x,t)\|^2_{L^2(\square\times J)}
 \ge
   \frac{\lambda}{C}\|(z-\tilde{z})(x,t)\|^2_{L^2(\square\times J)},\nonumber
\end{align}
whence follows that $z=\tilde{z}$. 

In case ($r=2$ and) $1<p<2$, by Lemma \ref{relationvw}, we first note that $z(x,t,\cdot,\cdot) \equiv 0$ in $\square \times J$ whenever $v_0(x,t) = 0$, and  then,  we observe immediately that
$$
a_{\rm hom}(x,t) e_k = \int^1_0 \int_\square a(y,s) e_k \, dy ds \quad \mbox{ for } \ k=1,2,\ldots,N.
$$
Hence we shall restrict ourselves to the case that $v_0(x,t)\neq 0$ below. By \eqref{03194} and Lemma \ref{relationvw}, we can deduce that
\begin{align}
\lefteqn{
\int_0^1\int_{\square}\Bigl[w(x,t,y,s)b\left(y\right)\partial_s c\left(s\right)
}\label{wform2d}\\
&-a(y,s)\bigl(\nabla v_{0}(x,t)+p|v_0(x,t)|^{(p-1)/p}\nabla_yw(x,t,y,s)\bigl)\cdot\nabla_yb\left(y\right)c\left(s\right)\Bigl]\, dyds\nonumber =0
\end{align}
for a.e.~$(x,t) \in \Omega \times I$. Repeating a similar argument to the case $0 < p < 1$, one can verify that
\begin{align}
w &\in L^2(\Omega\times I;W^{1,2}(J;\Vs^*)),\nonumber\\
w(x,t,\cdot,1) &= w(x,t,\cdot,0) \ \mbox{ in } \ \Vs^* \ \mbox{ for a.e.~} (x,t) \in \Omega \times I.\label{cl:4}
\end{align}
Furthermore, for each $k=1,2,\ldots, N$, let $\Psi_k = \Psi_k(x,t,\cdot,\cdot)\in L^2(J;H^1_{\mathrm{per}}(\square)/\R)$ be the solution of \eqref{local21} and set $\tilde{w}=\sum_{k=1}^N\partial_{x_k}v_0\Psi_k$. Then we find by \eqref{wform2d} that 
\begin{align}\label{03195}
\lefteqn{
\qquad 0 = -\int_0^1 \int_\square \left(w(x,t,y,s)-\tilde{w}(x,t,y,s)\right) b(y) \partial_s c(s)\, dy ds
}\\
 &+p|v_0(x,t)|^{(p-1)/p}\int_0^1\int_{\square}a(y,s)\nabla_y\left( w(x,t,y,s)-\tilde{w}(x,t,y,s)\right)\cdot\nabla_yb(y)c(s)\, dyds.\nonumber
\end{align}
As in \eqref{03196}, put $(w-\tilde{w})(x,t,\cdot,\cdot)$ in place of the product between $b\in C^{\infty}_{\rm per}(\square)$ and $c\in C^{\infty}_{\rm per}(J)$ in  \eqref{03195}.  Then we observe by the Poincar\'e-Wirtinger inequality and \eqref{ellip} that
\begin{align*}
0&\stackrel{\eqref{03195}}= - \frac{1}{2} \underbrace{\int_0^1 \frac d{ds} \left\| (w-\tilde{w})(x,t,\cdot,s)\right\|_{L^2(\square)}^2\, ds}_{=\,0 \ \text{ by periodicity in $s$}}\\
&\quad + p|v_0(x,t)|^{(p-1)/p} \int_0^1\int_{\square}a(y,s)\nabla_y (w-\tilde{w})(x,t,y,s) \cdot\nabla_y (w-\tilde{w})(x,t,y,s)\, dyds\\
&\ge \lambda p|v_0(x,t)|^{(p-1)/p}\|\nabla_y (w-\tilde{w})(x,t)\|^2_{L^2(\square\times J)}\\
&\ge \frac{\lambda p|v_0(x,t)|^{(p-1)/p}}{C}\|(w-\tilde{w})(x,t)\|^2_{L^2(\square\times J)},
\end{align*}
which implies $w=\tilde{w}$,  when $v_0(x,t)\neq 0$.  Hence by Lemma \ref{relationvw}, we have
$$
z=p|v_0|^{(p-1)/p}w = p|v_0|^{(p-1)/p} \sum_{k=1}^N \left(\partial_{x_k}v_0\right) \Psi_k.
$$
Hence it may be convenient to define $\Phi_k$ by \eqref{ahom:c:pm}. Finally, repeating the same argument as in the case $0 < r < 2$, we can verify that $a_ {\rm hom} $ can be written as \eqref{ahom:c:fd} (respectively, that with \eqref{ahom:c:pm}) for $0 < p < 1$ (respectively, $1 < p < 2$).

In case $2<r<+\infty$, before passing to the limit as $\e_n\to 0_+$, multiply both sides of \eqref{mostest} by $\e_n^{r-2}$. Then by Corollary \ref{veryweak} along with Lemma \ref{lemma5.2}, we have
\begin{align*}
0&=
\lim_{\e_n\to 0_+}\int_0^T\int_{\Omega}
\frac{v_{\e_n}(x,t)^{1/p}-v_0(x,t)^{1/p}}{\e_n}\phi(x)b\left(\tfrac{x}{\e_n}\right)\psi(t)\partial_s c\left(\tfrac{t}{\e_n^r}\right)\, dxdt\\
&=\vierint w(x,t,y,s)\phi(x)b(y)\psi(t)\partial_sc(s)\, dZ,
\end{align*}
which along with the arbitrariness of  $\phi\in C^{\infty}_{\rm c}(\Omega)$ and $\psi\in C^{\infty}_{\rm c}(I)$  yields
$$
\int^1_0 \int_{\square}w(x,t,y,s)b(y)\partial_sc(s)\, dy ds = 0 \quad \mbox{ for a.e.~} (x,t) \in \Omega \times I
$$
for any $b \in C^\infty_{\rm per}(\square)$ with zero mean and $c \in C^\infty_{\rm per}(J)$. Moreover, we may also assure that the relation above holds true for all $b \in C^\infty_{\rm per}(\square)$ with (possibly) non-zero mean; indeed, since $w(x,t,\cdot,s)$ has zero mean in $\square$, we observe that, for general $b \in C^\infty_{\rm per}(\square)$,
\begin{align*}
\lefteqn{
\int^1_0\int_{\square}w(x,t,y,s)b(y)\partial_sc(s)\, dyds
}\\
&= \int^1_0\int_{\square}w(x,t,y,s)\left( b(y) - \langle b \rangle_y \right)\partial_sc(s)\, dyds\\
&\quad + \langle b \rangle_y \int^1_0\langle w(x,t,\cdot,s) \rangle_y  \partial_s c(s) \,ds = 0.
\end{align*}
Thus the (distributional) derivative $\partial_s w(x,t,y, \cdot)$ turns out to be zero in $J$ for a.e.~$(x,t,y) \in \Omega \times I \times \square$. Therefore, $w(x,t,y,s)$ is independent of $s\in J$, and hence, so is $z(x,t,y,s)$ by Lemma \ref{relationvw}.
Now, we choose $c(s)$ in \eqref{mostest} as a constant function (without any multiplication of $\e_n$) to obtain 
\[
 \int_{\square}\int_0^T\int_{\Omega}
 \Bigl(\int_0^1a(y,s)\, ds\Bigl)\bigl(\nabla v_{0}(x,t)+\nabla_y z(x,t,y)\bigl)\cdot\phi(x)\nabla_yb(y)\psi(t)\, dxdtdy = 0
\]
for $\phi \in C^\infty_c(\Omega)$, $b \in C^\infty_{\rm per}(\square)$ and $\psi \in C^\infty_c(I)$. The arbitrariness of $\phi\in C^{\infty}_{\rm c}(\Omega)$ and $\psi\in C^{\infty}_{\rm c}(I)$ implies 
\begin{equation}\label{r3eq}
 \int_{\square}\Bigl(\int_0^1a(y,s)\, ds\Bigl)\bigl(\nabla v_{0}(x,t)+\nabla_y z(x,t,y)\bigl)\cdot\nabla_yb(y)\, dy=0 \quad \text{ a.e.~in }\ \Omega\times I
\end{equation}
for $b \in H^1_{\rm per}(\square)$. As in (i), we set
\begin{equation}\label{zform3}
 \tilde{z}(x,t,y)=\sum_{k=1}^N\partial_{x_k}v_0(x,t)\Phi_{k}(y)\quad \text{ for } k=1,2,\ldots, N,
\end{equation}
where $\Phi_k\in H^1_{\mathrm{per}}(\square)/\R$ is the unique (weak) solution to the following cell problem\/{:} 
$$
  -\dv_y \biggl(\Bigl(\int_0^1a(y,s)\, ds\Bigl)(\nabla_y\Phi_k(y)+e_{k})\biggl) =0\ 
\mbox{ in } \T^N.
$$
Then \eqref{r3eq} holds with $z$ replaced by $\tilde{z}$. 
Furthermore, an estimate similar to \eqref{ztilz} yields $z=\tilde{z}$. Recalling Theorem \ref{thm1}, we deduce that
\begin{align}
j_{\rm hom}(x,t) &\stackrel{\eqref{jhom}}{=} \int_0^1\int_{\square}a(y,s)\left(\nabla v_0(x,t)+\nabla_yz(x,t,y)\right)\, dyds\label{03192}\\
&\stackrel{\eqref{zform3}}{=} \int_{\square}\Bigl(\int_0^1a(y,s)\, ds\Bigl) \Bigl(\nabla v_0(x,t)+\sum_{k=1}^N\partial_{x_k}v_0(x,t)\nabla_y\Phi_{k}(y)\Bigl)\, dy\nonumber\\
&= \sum_{k=1}^N \biggl[ \int_{\square}  \Bigl( \int_0^1a(y,s)\, ds  \Bigl) \left(\nabla_y\Phi_{k}(y)+e_k\right)\, dy \biggl] \partial_{x_k}v_0(x,t).\nonumber
\end{align}
Therefore, the homogenized matrix $a_{\rm hom}\in \R^{N\times N}$ is represented by \eqref{a_hom3},
which along with \eqref{03192} implies \eqref{jhomahom}. Thus we have proved the assertion for the case $r > 2$.
\end{proof}

\section{Proof of Theorem \ref{T:cor}}\label{S:cor}

We first set up the following lemma  (see \S \ref{pf-pc} in Appendix for a proof): 
\begin{lem}[Pointwise convergence of $(v_\e^{1/p})$]\label{aeconv}
Under the same assumption as in Theorem \ref{T:cor}, it holds that 
\begin{align*}
  v_{\e}(t)^{1/p}\to v_{ 0}(t)^{1/p}\quad \text{ weakly in } L^{p+1}(\Omega)\quad \text{ \underline{for all} }\  t\in \overline{I}.
\end{align*}
\end{lem} 

Let us move on to a proof for Theorem \ref{T:cor}.

\begin{proof}[Proof of Theorem \ref{T:cor}]
 For simplicity, we write $\e=\e_n$ and set  $v_{\e}=|u_{\e}|^{p-1}u_{\e}$, $v_{0}=|u_{0}|^{p-1}u_{0}$ as before. We start with an approximation  $z_\ell(x,t,y,s)$  of the function $z(x,t,y,s)$ so that  its  gradient $\nabla_y z_\ell$ is an admissible test function (for the weak space-time two-scale convergence in $L^2(\Omega\times I \times \square \times J)$). Recalling that $z$ lies on $L^2(\Omega \times I ; L^2(J;H^1_{\rm per}(\square)/\R))$, one can approximate $z$ with a sequence $(s_\ell)$ of simple functions,
$$
s_\ell(x,t,y,s) = \sum_{k=1}^{n_\ell} w_{\ell,k}(y,s) \chi_{A_{\ell,k}}(x,t),
$$
where $n_\ell \in \N$, $\chi_{A_{\ell,k}} : \R^{N+1} \to \{0,1\}$  are  characteristic functions supported over measurable sets $A_{\ell,k}$ in $\Omega \times I$ and $w_{\ell,k} \in L^2(J;H^1_{\rm per}(\square)/\R)$, such that
$$
\left\| s_\ell - z \right\|_{L^2(\Omega \times I ; L^2(J;H^1_{\rm per}(\square)/\R))} < \ell^{-1}.
$$
Moreover, set
$$
z_\ell(x,t,y,s) = \sum_{k=1}^{n_\ell} w_{\ell,k}(y,s) \phi_{\ell,k}(x,t),
$$
where $\phi_{\ell,k} : \R^{N+1} \to \R$ is a smooth function given by
$$
\phi_{\ell,k}(x,t) := \left( \chi_{A_{\ell,k}} * \rho_\ell \right)(x,t)
$$
and $\rho_\ell$ stands for a standard mollifier such that
$$
\sum_{k=1}^{n_\ell} \|w_{\ell,k}\|_{L^2(J;H^1_{\rm per}(\square)/\R)} \|\chi_{A_{\ell,k}}-\phi_{\ell,k}\|_{L^2(\Omega \times I)} < \ell^{-1}.
$$
Then we observe that
\begin{align*}
\lefteqn{
\left\| z_\ell - z \right\|_{L^2(\Omega \times I ; L^2(J;H^1_{\rm per}(\square)/\R))}
}\\
&\leq \left\| z_\ell - s_\ell \right\|_{L^2(\Omega \times I ; L^2(J;H^1_{\rm per}(\square)/\R))} + \left\| s_\ell - z \right\|_{L^2(\Omega \times I ; L^2(J;H^1_{\rm per}(\square)/\R))}\\
&< \sum_{k=1}^{n_\ell} \|w_{\ell,k}\|_{L^2(J;H^1_{\rm per}(\square)/\R)} \|\chi_{A_{\ell,k}}-\phi_{\ell,k}\|_{L^2(\Omega \times I)} + \ell^{-1}
< 2 \ell^{-1}.
\end{align*}
Hence $z_{\ell}$ converges to $z$ strongly in $L^2(\Omega \times I ; L^2(J;H^1_{\rm per}(\square)/\R))$ as $\ell \to +\infty$. Here we also remark that $z_\ell$ does not depend on $s$, whenever so does $z$ (see the case $r > 2$ in Theorem \ref{thm2}). Indeed, one can then take $(w_{\ell,k})$ independent of $s$.

Moreover, we claim that $\nabla_y z_{\ell,k}(x,t,y,s) := \phi_{\ell,k}(x,t) \nabla_y w_{\ell,k}(y,s)$  are  admissible test functions. Indeed, since $w_{\ell,k}$ is (identified with) a periodic function in $(y,s)$, by Proposition \ref{mean}, it follows that  
\begin{align*}
\lim_{\e\to 0_+}\|\nabla_y z_{\ell,k}(x,t,\tfrac x\vep, \tfrac t{\vep^r})\|_{L^2(\Omega\times I)}^2=\|\nabla_y z_{\ell,k}\|_{L^2(\Omega\times I\times \square\times J)}^2,
\end{align*}
which implies \eqref{ad1}. Moreover, setting $X= L^2(J \times \square ; C(\overline{\Omega \times I}))$ equipped with 
$$
\|w\|_X := \bigg( \int^1_0 \int_\square \sup_{(x,t) \in \Omega \times I}|w(x,t,y,s)|^2 \, \d y \d s\bigg)^{1/2}
\quad \mbox{ for } \ w \in X
$$ (then $X$ is a separable normed space due to the separability of $C(\overline{\Omega \times I})$), we see that
\begin{align*}
\|\nabla_y z_{\ell,k}(x,t,\tfrac x \vep,\tfrac t{\vep^r})\|_{L^2(\Omega\times I)}^2
&\le \|\phi_{\ell,k}\|_{C(\overline{\Omega \times I})}^2 \|\nabla_y w_{\ell,k}(\tfrac x \vep, \tfrac t{\vep^r})\|_{L^2(\Omega\times I)}^2\\
&\leq C(\Omega  ,I ) \|\phi_{\ell,k}\|_{C(\overline{\Omega \times I})}^2 \|\nabla_y w_{\ell,k}\|_{L^2(\square\times J)}^2
=C(\Omega  ,I ) \|\nabla_y z_{\ell,k}\|_X^2
\end{align*}
for some constant $C(\Omega ,I )>0$ (see (ii) of Remark \ref{rem2.5}). Thus \eqref{ad2} has been checked. 

We next claim that
\begin{equation*}
\limsup_{\e\to 0_+}
\int_{0}^T\int_{\Omega}\left|\nabla v_{\e}  - \left(\nabla  v_0+\nabla_y z_{\ell}(x,t,\tfrac{x}{\e},\tfrac{t}{\e^r})\right)\right|^2\, dxdt\le C\|\nabla_y(z_{\ell}-z) \|_{L^2(\Omega\times I\times \square\times J)}.
\end{equation*}
Indeed, it follows from the uniform ellipticity and {symmetry} of $a(y,s)$ that
\begin{align}\label{errorest1}
\lefteqn{
\lambda\int_{0}^T\int_{\Omega}\left|\nabla v_{\e}-\left(\nabla  v_0+\nabla_yz_{\ell}(x,t,\tfrac{x}{\e},\tfrac{t}{\e^r})\right)\right|^2\, dxdt
}\\
&\stackrel{\eqref{ellip}}{\le}
\int_{0}^T\int_{\Omega} a_{\e}\left(\nabla v_{\e}-\left(\nabla  v_0+\nabla_yz_{\ell,\e} \right)\right) \cdot\left(\nabla v_{\e}-\left(\nabla  v_0+\nabla_y z_{\ell,\e} \right)\right)\, dxdt\nonumber\\
&=\int_{0}^T\int_{\Omega} a_{\e}\nabla v_{\e}\cdot \nabla v_{\e} \, dxdt - 2\int_{0}^T\int_{\Omega} a_{\e}\nabla v_{\e}\cdot \left(\nabla  v_0+\nabla_yz_{\ell}(x,t,\tfrac{x}{\e},\tfrac{t}{\e^r})\right) \, dxdt\nonumber\\
&\quad+\int_{0}^T\int_{\Omega} a_{\e}\left(\nabla  v_0+\nabla_yz_{\ell}(x,t,\tfrac{x}{\e},\tfrac{t}{\e^r})\right)\cdot \left(\nabla  v_0+\nabla_yz_{\ell}(x,t,\tfrac{x}{\e},\tfrac{t}{\e^r})\right)\, dxdt.\nonumber
\end{align}
Here and henceforth, we simply write $a_{\e}=a(\tfrac x{\e}, \tfrac t{\e^r})$ and $z_{\ell,\e} = z_\ell(x,t,\tfrac x \e, \tfrac t {\e^r})$. We shall  next  estimate each term in the right-hand side.

Due to the weak forms of \eqref{anal} and \eqref{analhomeq} along with \eqref{vjhom}, the first term reads,
\begin{align*}
\lefteqn{\limsup_{\e\to 0_+}\int_0^T\int_{\Omega}a_{\e}\nabla v_{\e}\cdot \nabla v_{\e}\, dxdt}\\
&\stackrel{\eqref{anal}}{=} 
\limsup_{\e\to 0_+}\left(\int_0^T\int_{\Omega}f_{\e}v_{\e}\, dxdt-\int_0^T \langle \partial_t v_{\e}^{1/p}, v_{\e}\rangle_{H^1_0(\Omega)}\, dt\right)\\
&= \limsup_{\e\to 0_+}\left[\int_0^T\int_{\Omega}f_{\e}v_{\e}\, dxdt-\frac{1}{p+1}\left(\|v_{\e}(T)^{1/p}\|_{L^{p+1}(\Omega)}^{p+1} - \|u^0\|_{L^{p+1}(\Omega)}^{p+1}\right)\right]\\
&\le \int_0^T\int_{\Omega}f_{0}v_{0}\, dxdt - \frac{1}{p+1}\left(\|v_0(T)^{1/p}\|_{L^{p+1}(\Omega)}^{p+1} - \|u^0\|_{L^{p+1}(\Omega)}^{p+1}\right)\\
&= \int_0^T\int_{\Omega}f_{0}v_{0}\, dxdt-\int_0^T \langle \partial_t v_{0}^{1/p}, v_{0}\rangle_{H^1_0(\Omega)}\, dt\\
&\stackrel{\eqref{analhomeq}}{=}\vierint a(y,s)(\nabla v_0+\nabla_y z)\cdot \nabla v_0\, dZ.
\end{align*}
Here we also used Lemmas \ref{strongconvofsol} and \ref{aeconv} along with the (additional) assumption on $(f_{\e})$.

As for the second term (see \eqref{errorest1}),  recall \eqref{grawttssol}, i.e., 
$\nabla v_{\e}\wtts\nabla v_{0}+\nabla_y z$ in $[L^2(\Omega\times I\times \square\times J)]^N$. Since $\nabla_y z_\ell$ is the finite sum of the admissible test functions $\nabla_y z_{\ell,k}$ for $k = 1,2,\ldots,n_\ell$, passing to the limit as $\e\to 0_+$, we have
\begin{align}
\lefteqn{\int_0^T\int_{\Omega} a_{\e}\nabla v_{\e}\cdot \left(\nabla  v_0+\nabla_yz_{\ell}(x,t,\tfrac{x}{\e},\tfrac{t}{\e^r})\right)\, dxdt}\label{ce:2t}\\
&= \int_0^T\int_{\Omega}\nabla v_{\e}\cdot \tenchi a_{\e} \left(\nabla v_{0}+\nabla_y z_{\ell}(x,t,\tfrac{x}{\e},\tfrac{t}{\e^r}) \right)\, dxdt\nonumber\\
&\to \vierint (\nabla v_0+\nabla_y z)\cdot \tenchi a(y,s) \left(\nabla v_0+\nabla_y z_{\ell}\right)\, dZ\nonumber\\
&= \vierint a(y,s)\left(\nabla v_0+\nabla_y z\right)\cdot \nabla v_0\, dZ\nonumber\\
&\quad +\vierint a(y,s) \left(\nabla v_0+\nabla_y z\right) \cdot \nabla_y z_\ell\, dZ,\nonumber
\end{align}
since $\tenchi a\nabla v_0$ and $\tenchi a\nabla_y z_{\ell,k}$ are also admissible test functions. Indeed, one can check \eqref{ad1} by noting from \eqref{mean1} of Proposition \ref{mean} with $a\in L^{\infty}(\square\times J)$ that
\begin{align*}
\lim_{\e\to 0_+}\int_0^T\int_{\Omega}|\tenchi a_{\e}\nabla v_0|^2\, dxdt
&= \vierint |\tenchi a(y,s)\nabla v_0|^2\, dZ,\\
\lim_{\e\to 0_+}\int_0^T\int_{\Omega}|\tenchi a_{\e}\nabla_y z_{\ell,k}(x,t,\tfrac{x}{\e},\tfrac{t}{\e^r})|^2\, dxdt
&= \vierint |\tenchi a(y,s)\nabla_y z_{\ell,k}|^2\, dZ.
\end{align*}
Moreover, setting $X_1=L^2(\Omega\times I\times\square\times J)$ and $X_2=L^{2}(\square\times J;C(\overline{\Omega\times I}))$, we see that
\begin{align*}
\|\tenchi a_{\e}\nabla v_0\|_{L^2(\Omega\times I)}^2
\stackrel{\eqref{rayleigh}}{\le} \frac{1}{\|\tenchi a\|_{L^2(\square\times J)}^2}\|\tenchi a\|_{L^2(\square\times J)}^2\|\nabla v_0\|_{L^2(\Omega\times I)}^2
=\frac{1}{\|\tenchi a\|_{L^2(\square\times J)}^2}\|\tenchi a\nabla v_0\|_{X_1}^2
\end{align*}
and
\begin{align*}
\|\tenchi a_{\e}\nabla_y z_{\ell,k}(x,t,\tfrac{x}{\e},\tfrac{t}{\e^r})\|_{L^2(\Omega\times I)}^2
&\le \|\phi_{\ell,k}\|_{C(\overline{\Omega\times I})}^2 \int_0^T\int_{\Omega} \bigl| \tenchi a_{\e} \nabla_y w_{\ell, k }(\tfrac{x}{\e},\tfrac{t}{\e^r}) \bigl|^2\, dxdt\\
&\le C(\Omega  ,I ) \|\phi_{\ell,k}\|_{C(\overline{\Omega\times I})}^2 \int_0^1\int_{\square} \bigl| \tenchi a(y,s)\nabla_y w_{\ell, k }(y,s) \bigl|^2 \, dyds\\
&\leq C(\Omega ,I )\|\tenchi a\nabla_y z_{\ell,k}\|_{X_2}^2.
\end{align*}
Thus \eqref{ad2} follows. Now, we shall handle the second term of the right-hand side of \eqref{ce:2t}  depending on $r$.  In case $0 < r < 2$, we observe that
\begin{align*}
 \mathcal{I}:= \int_0^T\int_{\Omega} \underbrace{
 \int_0^1 \int_{\square} a(y,s)(\nabla v_0+\nabla_y z) \cdot \nabla_y z_\ell\, dy ds}_{\quad =\, 0 \text{ by \eqref{zform}}} dx dt = 0.
\end{align*}
In case $2 < r < +\infty$, since both $z$ and $z_\ell$ are independent of $s$, it follows that
\begin{align*}
 \mathcal{I}= \int_0^T\int_{\Omega} \underbrace{
 \int_{\square} \Big( \int^1_0 a(y,s) \, ds \Big) (\nabla v_0+\nabla_y z) \cdot \nabla_y z_\ell\, dy }_{\quad =\, 0 \text{ by \eqref{r3eq}}} dx dt = 0.
\end{align*}
In case $r = 2$, one has
\begin{align*}
\mathcal{I} &= \vierint a(y,s)(\nabla v_0+\nabla_y z) \cdot \nabla_y z\, d Z\\
&\quad + \vierint a(y,s)(\nabla v_0+\nabla_y z) \cdot \nabla_y (z_\ell - z)\, d Z =: \mathcal{J}_1 + \mathcal{J}_2.
\end{align*}
Here we can derive that
$$
|\mathcal{J}_2| \leq C \|\nabla_y (z-z_\ell)\|_{L^2(\Omega \times I \times \square \times J)}.
$$
For the case where $0 < p < 1$, since $z(x,t,\cdot,\cdot) \in L^2(J;\Vs)$ and $(1/p)|v_0(x,t)|^{\frac{1-p}p} z(x,t,\cdot,\cdot) \in W^{1,2}(J ;\Vs^*)$, noting  by \eqref{triplet}  that
$$
 L^2(J;\Vs) \cap W^{1,2}(J;\Vs^*) \subset C(\overline J ; L^2(\square)/\R), 
$$
we deduce that
\begin{align*}
 \mathcal{J}_1 &\stackrel{\eqref{wform2}}= \int^T_0 \int_\Omega \left( \frac 1 p |v_0(x,t)|^{(1-p)/p} \int^1_0 \left\langle \partial_s z(x,t,\cdot,s),z(x,t,\cdot,s) \right\rangle_{\Vs} \, ds \right)\, dx dt\\
&= \int^T_0 \int_\Omega \frac 1 {2p} |v_0(x,t)|^{(1-p)/p} \left( \|z(x,t,\cdot,1)\|_{L^2(\square)}^2 - \|z(x,t,\cdot,0)\|_{L^2(\square)}^2 \right) \ dx dt
 \stackrel{\eqref{cl:2}}= 0.
\end{align*}
For the case where $1 < p < 2$, we can also observe that
\begin{align*}
 \mathcal{J}_1 &\stackrel{\eqref{wform2d}}= \int^T_0 \int_\Omega \int^1_0 \left\langle \partial_s  w(x,t,\cdot,s), z (x,t,\cdot,s) \right\rangle_{\Vs} \, ds dx dt\\
&= \int^T_0 \int_\Omega \int^1_0 \left\langle p|v_0(x,t)|^{(p-1)/p} \partial_s w(x,t,\cdot,s),w(x,t,\cdot,s) \right\rangle_{\Vs} \, ds dx dt\\
&= \int^T_0 \int_\Omega \frac p 2 |v_0(x,t)|^{(p-1)/p} \left( \|w(x,t,\cdot,1)\|_{L^2(\square)}^2 - \|w(x,t,\cdot,0)\|_{L^2(\square)}^2 \right) \ dx dt
\stackrel{\eqref{cl:4}}= 0.
\end{align*}

Concerning the third term  (see \eqref{errorest1}),  by \eqref{mean1} of Proposition \ref{mean}, we can derive that
\begin{align*}
\lim_{\e\to 0_+} \lefteqn{\int_0^T\int_{\Omega} a_{\e}\left(\nabla v_0+\nabla_y z_{\ell}(x,t,\tfrac{x}{\e},\tfrac{t}{\e^r})\right) \cdot \left(\nabla  v_0+\nabla_yz_{\ell}(x,t,\tfrac{x}{\e},\tfrac{t}{\e^r})\right)\, dxdt}\\
&= \lim_{\e\to 0_+} \int_0^T\int_{\Omega}\Bigl[ a_{\e} \nabla v_{0}\cdot \nabla v_{0} + 2a_{\e}\nabla_y z_{\ell,\e}\cdot\nabla v_0 + a_{\e}\nabla_y z_{\ell,\e}\cdot\nabla_y z_{\ell,\e}\Bigl]\, dxdt\\
&= \vierint \left[ a(y,s)\nabla v_{0}\cdot \nabla v_{0}+2a(y,s)\nabla_yz_{\ell}\cdot\nabla v_0 +a(y,s)\nabla_y z_{\ell}\cdot\nabla_y z_{\ell} \right]\, dZ\\
 &= \vierint a(y,s)(\nabla v_0+\nabla_y z_{\ell})\cdot (\nabla v_0+\nabla_y z_{\ell})\, dZ\\
 &= \vierint a(y,s)(\nabla v_0+\nabla_y z_{\ell})\cdot \nabla v_0\, dZ\\
 &\quad + \underbrace{\vierint a(y,s)(\nabla v_0+\nabla_y z)\cdot \nabla_y z_\ell\, dZ}_{\ = \, \mathcal{I}}\\
 &\quad + \vierint a(y,s)\nabla_y (z_{\ell}-z)\cdot \nabla_y z_{\ell}\, dZ\\ 
 &\le \vierint a(y,s)(\nabla v_0+\nabla_y z_{\ell})\cdot \nabla v_0\, dZ + \mathcal{I}\\
 &\quad + \|\nabla_y (z_{\ell}-z)\|_{L^2(\Omega\times I\times\square\times J)} \|\tenchi a\nabla_y z_{\ell}\|_{L^2(\Omega\times I\times\square\times J)}.
\end{align*}
Consequently, combining all these estimates with \eqref{errorest1}, we obtain
\begin{align*}
&\limsup_{\e\to 0_+}\lambda \int_0^T\int_{\Omega} \left|\nabla v_{\e}-\left(\nabla v_0+\nabla_y z_{\ell}(x,t,\tfrac{x}{\e},\tfrac{t}{\e^r})\right) \right|^2\, dxdt\\
&\quad \le C\|\nabla_y (z_{\ell}-z)\|_{L^2(\Omega\times I\times\square\times J)}.
\end{align*}

 We shall  postpone discussing the measurability of the function $(x,t) \mapsto \nabla_y z(x,t,\tfrac x\e, \tfrac t{\e^r})$ for a while (indeed, it is not trivial at all). To complete the proof, recalling $v_{\e}=|u_{\e}|^{p-1}u_{\e}$ and $v_{0}=|u_{0}|^{p-1}u_{0}$ along with $z=\sum_{k=1}^N (\partial_{x_k}v_0) \Phi_k$, 
we conclude that, for $\ell>0$ large enough,  
\begin{align*}
\lefteqn{
\limsup_{\e\to 0_+}\int_{0}^T\int_{\Omega} \Bigl| \nabla |u_{\e}|^{p-1}u_{\e} - \nabla |u_0|^{p-1}u_0 - \sum_{k=1}^N\left(\partial_{x_k} |u_0|^{p-1}u_0\right)\nabla_y\Phi_k\left(x,t,\tfrac{x}{\e},\tfrac{t}{\e^r}\right) \Bigl|^2\, dxdt
}\\
&=\limsup_{\e\to 0_+} \int_{0}^T \int_{\Omega} \left| \nabla v_{\e} - \nabla v_0 - \nabla_y z(x,t,\tfrac{x}{\e},\tfrac{t}{\e^r}) \right|^2\, dxdt \hspace{65mm}\\
&\le \limsup_{\e\to 0_+}  2 \int_{0}^T\int_{\Omega} \left| \nabla v_{\e} - \nabla  v_0 - \nabla_y z_{\ell}(x,t,\tfrac{x}{\e},\tfrac{t}{\e^r}) \right|^2 \, dxdt\\
&\quad + \limsup_{\e\to 0_+}  2 \int_{0}^T\int_{\Omega} \left| \nabla_y \left( z_{\ell} - z\right)(x,t,\tfrac{x}{\e},\tfrac{t}{\e^r}) \right|^2\, dxdt \\
&\le C\|\nabla_y (z_{\ell}-z)\|_{L^2(\Omega\times I\times\square\times J)}.
\end{align*}
Here we also used the fact that
\begin{align*}
\lefteqn{
\limsup_{\e\to 0_+} \int_{0}^T\int_{\Omega} \left| \nabla_y \left( z_{\ell} - z\right)(x,t,\tfrac{x}{\e},\tfrac{t}{\e^r}) \right|^2\, dxdt
}\\
&= \vierint \left| \nabla_y \left( z_{\ell} - z\right)(x,t,y,s) \right|^2\, dZ,
\end{align*}
which still remains to be proved. To this end, we need assume some regularity for the matrix field $a(y,s)$ (see Assumption (ii) of Theorem \ref{T:cor}),  which also enables us to check the measurability of $(x,t) \mapsto \nabla_y z(x,t,\tfrac x \e, \tfrac t{\e^r})$ in $\Omega \times I$.  In case $r \in (0,2)$, we have already known that $\Phi_k(y,s)$ is independent of $(x,t)$ and $\nabla_y z(x,t,y,s) = \sum_{k=1}^N \partial_{x_k} v_0(x,t) \nabla_y \Phi_k(y,s)$, i.e., $\nabla_y z$ is a finite sum of (multiplicatively) separable functions in micro- and macroscopic variables. Hence, the measurability of the function $(x,t) \mapsto \nabla_y z(x,t,\tfrac x \e, \tfrac t {\e^r})$ in $\Omega \times I$ follows immediately.  Moreover, we  claim  that
\begin{equation}
\nabla_y\Phi_k \in L^{\infty}(\square \times J). \label{CP1regu}
\end{equation}
Indeed, since $a \in L^\infty(J;C^\alpha_{\rm per}(\square))$ for some $\alpha \in (0,1)$ by assumption, we can guarantee, thanks to~\cite[Theorem 1.1]{arm2}, that
$
\nabla_y\Phi_k(\cdot,s) \in L^{\infty}(\square)
$
(see~\cite[Section 4]{arm2} for details).
Furthermore, we see by~\cite[Lemma 3.5]{arm2} along with the weak-star lower semicontinuity of norm that
$$
\|\nabla_y \Phi_k(s)\|_{L^{\infty}(\square)}\le C\quad \text{ for all }\ s\in J,
$$
where $C>0$ is independent of $s\in J$.  Hence we obtain \eqref{CP1regu}. We write $z_{\ell,\vep} = z_\ell(x,t,\tfrac x \e, \tfrac t {\e^r})$ and $z_\e = z(x,t,\tfrac x \e, \tfrac t {\e^r})$ below. It then follows that
\begin{align*}
\lefteqn{
 \int^T_0 \int_\Omega \left| \nabla_y \left( z_{\ell} - z\right)(x,t,\tfrac{x}{\e},\tfrac{t}{\e^r}) \right|^2 \, dxdt
}\\
&= \int^T_0 \int_\Omega \left( |\nabla_y z_{\ell,\vep}|^2 - 2 \nabla_y z_{\ell,\vep} \cdot \nabla_y z_{\vep} + |\nabla_y z_\vep|^2 \right) \, dx dt\\
&= \int^T_0 \int_\Omega \Big( \sum_{i=1}^{n_\ell}\sum_{j=1}^{n_\ell} \nabla_y w_{\ell,i}(\tfrac x \e, \tfrac t {\e^r}) \cdot \nabla_y w_{\ell,j}(\tfrac x \e, \tfrac t {\e^r}) \phi_{\ell,i}(x,t) \phi_{\ell,j}(x,t)\\
&\quad -2 \sum_{i=1}^{n_\ell} \sum_{k=1}^N \nabla_y w_{\ell,i}(\tfrac x \e, \tfrac t {\e^r}) \cdot \nabla_y \Phi_k(\tfrac x \e, \tfrac t {\e^r}) \phi_{\ell,i}(x,t) \partial_{x_k}v_0(x,t)\\
&\quad + \sum_{h=1}^N \sum_{k=1}^N \nabla_y \Phi_h(\tfrac x \e, \tfrac t {\e^r})\cdot \nabla_y \Phi_k(\tfrac x \e, \tfrac t {\e^r}) \partial_{x_h}v_0(x,t) \partial_{x_k}v_0(x,t) \Big) \, dxdt\\
&\to \vierint \Big( \sum_{i=1}^{n_\ell}\sum_{j=1}^{n_\ell} \nabla_y w_{\ell,i}(y,s) \cdot \nabla_y w_{\ell,j}(y,s) \phi_{\ell,i}(x,t) \phi_{\ell,j}(x,t)\\
&\quad -2 \sum_{i=1}^{n_\ell} \sum_{k=1}^N \nabla_y w_{\ell,i}(y,s) \cdot \nabla_y \Phi_k(y,s) \phi_{\ell,i}(x,t) \partial_{x_k}v_0(x,t)\\
&\quad + \sum_{h=1}^N \sum_{k=1}^N \nabla_y \Phi_h(y,s) \cdot \nabla_y \Phi_k(y,s) \partial_{x_h}v_0(x,t) \partial_{x_k}v_0(x,t) \Big) \, dZ\\
&= \vierint |\nabla_y(z_\ell-z)(x,t,y,s)|^2 \, dZ. 
\end{align*}
One can similarly prove the assertion for the case $r \in (2,+\infty)$, where $\Phi_k$ depends only on $y$. In case $r = 2$, we can prove that
\begin{equation}\label{sc:reg}
\nabla_y \Phi_k \in L^\infty(\Omega \times I ; C_{\rm per}(\square \times J)) \ \mbox{ for } \ k=1,2,\ldots,N
\end{equation}
for smooth $a(y,s)$ (see \S \ref{A:S:reg} in Appendix for more details).  Hence, since $\nabla_y \Phi_k$ is a Carath\'eodory function, the function $(x,t) \mapsto \nabla_y z(x,t,\tfrac x \e, \tfrac t {\e^r})$ is measurable in $\Omega\times I$ (see Remark \ref{R:meas-osci}). Moreover,  noting by \eqref{sc:reg} that
$$
\nabla_y \Phi_h(x,t,y,s) \cdot \nabla_y \Phi_k(x,t,y,s) \partial_{x_h}v_0(x,t) \partial_{x_k}v_0(x,t) \in L^1(\Omega\times I;C_{\rm per}(\square \times J))
$$
and exploiting  Lemma 1.3 of~\cite{al1} (see also~Theorem 3.3 of~\cite{Z2003})  with obvious modification, one can verify that
\begin{align*}
\lefteqn{
\int^T_0 \int_\Omega \nabla_y \Phi_h(x,t,\tfrac x \e,\tfrac t {\e^r}) \cdot \nabla_y \Phi_k(x,t,\tfrac x \e,\tfrac t {\e^r}) \partial_{x_h}v_0(x,t) \partial_{x_k}v_0(x,t) \, dxdt
}\\
&\to \vierint \nabla_y \Phi_h(x,t,y,s) \cdot \nabla_y \Phi_k(x,t,y,s) \partial_{x_h}v_0(x,t) \partial_{x_k}v_0(x,t) \, dZ,
\end{align*}
and moreover,
\begin{align*}
\lefteqn{
\int^T_0 \int_\Omega \nabla_y w_{\ell,i}(\tfrac x \e,\tfrac t {\e^r}) \cdot \nabla_y \Phi_k(x,t,\tfrac x \e,\tfrac t {\e^r}) \phi_{\ell,i}(x,t) \partial_{x_k}v_0(x,t) \, dxdt
}\\
&\to \vierint \nabla_y w_{\ell,i}(y,s) \cdot \nabla_y \Phi_k(x,t,y,s) \phi_{\ell,i}(x,t) \partial_{x_k}v_0(x,t) \, dZ
\end{align*}
as $\e \to 0_+$. Therefore the rest of proof runs as before.
\end{proof}

We next prove Corollary \ref{C:cor}.

\begin{proof}[Proof of Corollary \ref{C:cor}]
Write $\e = \e_n$, $a_\e = a(\tfrac{x}{\e},\tfrac{t}{\e^r})$ for simplicity and let $j_\e := a_\e \nabla v_\e$ be the diffusion flux of \eqref{anal}. Recall that $z(x,t,y,s) = \sum_{k=1}^N \partial_{x_k} v_0(x,t) \Phi_k(x,t,y,s)$. Thanks to Theorem \ref{T:cor}, one observes by \eqref{rayleigh} that
\begin{align*}
\lefteqn{
\left\|
j_{\vep} - a_\e \left( \nabla v_{0}+\nabla_y z \left(x,t,\tfrac{x}{\e},\tfrac{t}{\e^r}\right) \right)
\right\|_{L^2(\Omega \times I)}
}\\
&\leq C \left\|
\nabla v_{\vep} - \nabla v_{0} - \nabla_y z \left(x,t,\tfrac{x}{\e},\tfrac{t}{\e^r}\right)
\right\|_{L^2(\Omega \times I)} \to 0
\end{align*}
(cf.~see \eqref{grawtts}). Therefore, although $j_\e \to j_{\rm hom}$ weakly in $L^2(\Omega \times I)$ as $\e \to 0_+$, we deduce that
\begin{align*}
\lim_{\vep \to 0_+}\left\|
j_{\vep} - j_{\rm hom} - \left[ a_\e \left( \nabla v_{0}+\nabla_y z \left(x,t,\tfrac{x}{\e},\tfrac{t}{\e^r}\right) \right) - j_{\rm hom} \right]
\right\|_{L^2(\Omega \times I)} = 0,
\end{align*}
where $j_{\rm hom} = j_{\rm hom}(x,t)$ is the homogenized flux defined by \eqref{jhom}. Here it is noteworthy that the corrector term
$$
a_\e \left( \nabla v_{0}+\nabla_y z \left(x,t,\tfrac{x}{\e},\tfrac{t}{\e^r}\right) \right) - j_{\rm hom}
$$
converges to zero \emph{weakly} in $L^2(\Omega \times I)$ as $\e \to 0_+$ (by \eqref{jhom} along with Proposition \ref{mean} for $r\neq2$ and Lemma 1.3 of~\cite{al1} for $r=2$),  but it does \emph{not converge strongly} in general. Indeed,  as in the proof of Theorem \ref{T:cor}, we can verify that
\begin{align*}
\lefteqn{
\left\|
a_\e \left( \nabla v_{0}+\nabla_y z \left(x,t,\tfrac{x}{\e},\tfrac{t}{\e^r}\right) \right) - j_{\rm hom} 
\right\|_{L^2(\Omega\times I)}^2
}\\
&= \int^T_0 \int_\Omega \left| a_\e \left( \nabla v_{0}(x,t)+\nabla_y z \left(x,t,\tfrac{x}{\e},\tfrac{t}{\e^r}\right) \right) - j_{\rm hom}(x,t) \right|^2 \, \d x \d t\\
&\to \vierint \left| a(y,s)\left( \nabla v_{0}(x,t)+\nabla_y z (x,t,y,s) \right) - j_{\rm hom}(x,t) \right|^2 \, \d Z\\
&\neq 0,
\end{align*}
unless the function $a(y,s) \left( \nabla v_0(x,t) + \nabla_y z (x,t,y,s) \right)$ is constant for $(y,s) \in \square \times J$ (see Remark \ref{R:corr-non0} below).

Furthermore, noting that
\begin{equation}\label{div-H-1}
\|\mathrm{div} \,\varphi\|_{H^{-1}(\Omega)} = \sup_{\|\nabla w\|_{L^2(\Omega)} = 1} \int_\Omega \varphi \cdot \nabla w \, dx \leq \|\varphi\|_{L^2(\Omega)}
\quad \mbox{ for } \ \varphi \in [L^2(\Omega)]^N,
\end{equation}
we can derive that
\begin{align*}
\lefteqn{
 \left\| \partial_t v_{\e}^{1/p} - \partial_t v_0^{1/p} - \mathrm{div} \left[ a_\e \left( \nabla v_{0}+\nabla_y z \left(x,t,\tfrac{x}{\e},\tfrac{t}{\e^r}\right) \right) - j_{\rm hom} \right] \right\|_{L^2(I;H^{-1}(\Omega))}
}\\
&\leq \left\| j_{\e} - j_{\rm hom} - \left[ a_\e \left( \nabla v_{0}+\nabla_y z \left(x,t,\tfrac{x}{\e},\tfrac{t}{\e^r}\right) \right) - j_{\rm hom} \right] \right\|_{L^2(\Omega\times I)}
\to 0
\end{align*}
as $\e \to 0_+$.
\end{proof}

We close this section with the following remark:

\begin{remark}[The corrector does not vanish generally]\label{R:corr-non0}
{\rm
\begin{enumerate}
 \item[(i)] Let us consider the case $N \geq 2$ and assume for simplicity that $r \neq 2$. Then $\Phi_k$ is independent of macroscopic variables $x,t$, and in particular, it is irrelevant to $v_0(x,t)$. Moreover, it is not always true that
\begin{equation*}
a(y,s) \left( \nabla_y \Phi_k (y,s) + e_k \right) = b_k 
\quad \mbox{ for a.e.~} (y,s) \in \square \times J, \quad k = 1,2,\ldots,N 
\end{equation*}
for some constant vectors $b_k \in \R^N$, and hence, if it is not true, the corrector term never vanishes as $\e \to 0_+$. Then the flux $j_{\e}$ cannot converge strongly.
 \item[(ii)] In the one-dimensional case, for $r \in (0,2)$, the cell problem reads,
$$
\partial_y \left( a(y,s)\left[ \partial_y \Phi(y,s) + 1 \right]\right) = 0 \ \mbox{ in } \, \square \times J,
$$
which implies
$$
a(y,s) \left[ \partial_y \Phi(y,s) + 1 \right] = c(s) \ \mbox{ in } \, \square \times J
$$
for some $c(s) \in \R$ independent of $y \in \square = (0,1)$ but (possibly) depending on $s \in J$. Then the periodicity of $\Phi(y,s)$ in $y$ yields
$$
c(s) = \left( \int^1_0 a(y,s)^{-1} \, \d y \right)^{-1} = \langle a(\cdot,s)^{-1} \rangle_y^{-1},
$$
which is not constant for $s \in J$ in general. If $c(s)$ is not constant, $j_{\e}$ cannot then converge to $j_{\rm hom}$ strongly in $L^2(\Omega\times I)$. Furthermore, by $N=1$, \eqref{div-H-1} holds with an equality, that is, $\|\varphi'\|_{H^{-1}(\Omega)} = \|\varphi\|_{L^2(\Omega)}$ for $\varphi \in L^2(\Omega)$. Therefore $\partial_t u_\e$ cannot converge to $\partial_t u_0$ strongly in $L^2(I;H^{-1}(\Omega))$ as well.
\end{enumerate}
}
\end{remark}

\section{Proof of Proposition \ref{P:ahom}}\label{S:hmat}

We first prove (i). In case $r=2$ and $0<p<1$, for each $\xi=[\xi_k]_{k=1,2,\ldots,N}\in\R^N$, there exists a weak solution $\Phi_{\xi}=\sum_{k=1}^N \xi_k \Phi_k$ to 
\begin{align}\label{relocal2}
\begin{cases}
\frac{1}{p}|v_0|^{(1-p)/p}\partial_s\Phi_{\xi}(x,t,y,s)-\dv_y\left(a(y,s)[\nabla_y \Phi_{\xi}(x,t,y,s)+\xi] \right)=0 &\mbox{ in } \T^N \times \T,\\
\Phi_\xi(x,t,y,0)=\Phi_\xi(x,t,y,1) &\mbox{ in } \T^N.
\end{cases}
\end{align}
Using  \eqref{relocal2}  and \eqref{ellip}, we have
\begin{align*}
a_{\rm hom}(x,t)\xi\cdot\xi
&= \int_{0}^1\int_{\square} a(y,s)(\nabla_y\Phi_{\xi}+\xi)\cdot \xi\, dyds\\
&\stackrel{\eqref{relocal2}}{=} \int_{0}^1\int_{\square} a(y,s)(\nabla_y\Phi_{\xi}+\xi)\cdot(\nabla_y\Phi_{\xi}+\xi) \, dy ds \\
&\quad + \frac{1}{2p} |v_0(x,t)|^{(1-p)/p} \int^1_0 \dfrac d{ds} \|\Phi_{\xi}( x,t,\cdot,s)\|_{L^2(\square)}^2 \, ds\\
 &\stackrel{\eqref{ellip}}{\ge} \lambda \int_{0}^1\int_{\square}|\xi+\nabla_y\Phi_{\xi}|^2\, dyds\\
&= \lambda \sum_{k=1}^N \left( 1 + \int_{0}^1\int_{\square} |\nabla_y\Phi_k|^2 \, dyds \right) |\xi_k|^2.
\end{align*}
Here we used the fact that $\langle\nabla_y\Phi_{\xi}\rangle_y=0$ and $\Phi_{\xi}|_{s=0}=\Phi_{\xi}|_{s=1}$ by the periodicity of $\Phi_{\xi}$ in $(y,s)\in\square\times J$. Furthermore, we also find that
\begin{align}\label{ahom-up}
a_{\rm hom}(x,t)\xi\cdot\xi
\leq \sum_{k=1}^N \left( 1 + \int_{0}^1\int_{\square} |\nabla_y\Phi_k|^2 \, dyds \right) |\xi_k|^2.
\end{align}
In case $r=2$ and $1<p<2$,  let $(x,t) \in \Omega\times I$ be such that $v_0(x,t)\neq 0$. Setting $\Psi_{\xi}=\sum_{k=1}^N\xi_k \Psi_{k}$ and noting that $\Phi_k = p|v_0|^{(p-1)/p}\Psi_k$ solves \eqref{relocal2},  we have
\begin{align*}
 a_{\rm hom}(x,t)\xi\cdot\xi
 &=
 \int_{0}^1\int_{\square} a(y,s)(\nabla_y\Phi_{\xi}+\xi)\cdot \xi\, dyds\\
 &\stackrel{\eqref{relocal2}}{=}
 \int_{0}^1\int_{\square} a(y,s)\left(p|v_0|^{(p-1)/p}\nabla_y\Psi_{\xi}+\xi\right)\cdot\left(p|v_0|^{(p-1)/p}\nabla_y\Psi_{\xi}+\xi\right)\, dy ds\\
 &\quad + \frac{p}{2} |v_0(x,t)|^{(p-1)/p} \int^1_0 \frac d {ds} \|\Psi_{\xi}(s)\|_{L^2(\square)}^2\, ds\\
 &\stackrel{\eqref{ellip}}{\ge}
 \lambda \int_{0}^1\int_{\square} \left|\xi+p|v_0|^{(p-1)/p}\nabla_y\Psi_{\xi}\right|^2\, dyds\\
 &= \lambda \sum_{k=1}^N \left( 1 + \int_{0}^1\int_{\square} |\nabla_y\Phi_k|^2 \, dyds \right) |\xi_k|^2.
\end{align*}
Moreover, \eqref{ahom-up} also follows.  Concerning $(x,t) \in \Omega \times I$ for which $v_0(x,t)$ vanishes, we see that $\Phi_k(x,t,\cdot,\cdot) \equiv 0$ in $\square \times J$, and hence, the same conclusion follows immediately.  In case $r\neq 2$, one can prove the assertion in a similar way to the case $0<p<1$.

We next prove (ii) for $r\in (0,2)$. Since $a(y,s)$ is symmetric, it follows that
\begin{align*}
[a_{\rm hom}]_{ j,k}
 &= a_{\rm hom} e_k \cdot e_j\\
 &\stackrel{\eqref{ahomfast}}{=} \int_0^1\int_{\square} a(y,s)(\nabla_y\Phi_k+e_k)\cdot e_j\, dyds\\
 &\quad +\underbrace{\int_0^1\int_{\square}a(y,s)(\nabla_y\Phi_k+e_k)\cdot\nabla_y\Phi_j\, dyds}_{=\,0}\\
 &= \int_0^1\int_{\square}a(y,s)(\nabla_y\Phi_j+e_j)\cdot(\nabla_y\Phi_k+e_k)\, dyds\\
 &= \int_0^1\int_{\square} a(y,s)(\nabla_y\Phi_j+e_j)\cdot e_k\, dyds\\
 &\quad +\underbrace{\int_0^1\int_{\square} a(y,s)(\nabla_y\Phi_j+e_j)\cdot\nabla_y\Phi_k\, dyds}_{=\,0}
 = [a_{\rm hom}]_{ k,j}.
\end{align*}
One can also verify the symmetry for $r\in (2,+\infty)$ in a similar fashion.

Finally, we shall discuss the critical case $r = 2$. For the case where $p \in (0,1)$, one observes that 
\begin{align*}
[a_{\rm hom}(x,t)]_{ j,k}
 &= a_{\rm hom}(x,t) e_k \cdot e_j\\
 &= \int_0^1 \int_{\square} a(y,s)(\nabla_y\Phi_k+e_k)\cdot e_j\, dyds\\
 &\quad+\underbrace{\int_0^1 \Big[\int_{\square} a(y,s)(\nabla_y\Phi_k+e_k)\cdot\nabla_y \Phi_{j}\, dy +\frac{1}{p}|v_0|^{(1-p)/p} \left\langle \partial_s \Phi_k, \Phi_j \right\rangle_V \Big] \, ds}_{=\,0}\\
 & = \int_0^1 \int_{\square} a(y,s)e_j\cdot (\nabla_y\Phi_k+e_k)\, dyds \\
 & \quad+\int_0^1 \Big[\int_{\square} a(y,s)\nabla_y \Phi_{j}\cdot(\nabla_y\Phi_k+e_k)\, dy +\frac{1}{p}|v_0|^{(1-p)/p} \left\langle \partial_s \Phi_k, \Phi_j \right\rangle_V \Big] \, ds \\
 & =\int_0^1 \Big[\int_{\square} a(y,s)(\nabla_y \Phi_{j}+e_j)\cdot(\nabla_y\Phi_k+e_k)\, dy +\frac{1}{p}|v_0|^{(1-p)/p} \left\langle \partial_s \Phi_k, \Phi_j \right\rangle_V \Big] \, ds \\
 &= \int_0^1\int_{\square} a(y,s)(\nabla_y \Phi_{j}+e_{j})\cdot e_k\, dyds\\
 &\quad +\int_0^1\Bigl[\int_{\square} a(y,s)(\nabla_y \Phi_{j}+e_{j})\cdot\nabla_y\Phi_k\, dy + \frac{1}{p}|v_0|^{(1-p)/p}\left\langle\partial_s\Phi_k,\Phi_{j}\right\rangle_{\Vs}\, \Bigl]\, ds\\
 &= \underbrace{a_{\rm hom}(x,t)e_{j}\cdot e_k}_{\ = \, [a_{\rm hom}(x,t)]_{ k,j}} + \frac{1}{p}|v_0|^{(1-p)/p}\int_0^1 \left( \langle \partial_s\Phi_k, \Phi_{j} \rangle_{\Vs}  - \langle \partial_s \Phi_{j}, \Phi_k \rangle_{\Vs} \right)\, ds
\end{align*}
for $ j,k = 1,2,\ldots, N$. Hence $a_{\rm hom}(x,t)$  is not  symmetric,  unless the second term of the right-hand side vanishes for all $j,k$ different each other.  For the case where $p\in(1,2)$, assume that $v_0(x,t) \neq 0$. Then we deduce that
\begin{align*}
\lefteqn{[a_{\rm hom}(x,t)]_{ j,k}}\\
 &= a_{\rm hom}(x,t) e_k \cdot e_{j}\\
 &= \int_0^1\int_{\square} a(y,s)\left(p|v_0|^{(p-1)/p}\nabla_y\Psi_k+e_k\right)\cdot e_{j}\, dyds\\
 &\quad + p|v_0|^{(p-1)/p} \underbrace{\int_0^1\Bigl[\int_{\square} a(y,s) \left(p|v_0|^{(p-1)/p}\nabla_y\Psi_k+e_k\right)\cdot\nabla_y\Psi_{j}\, dy +\left\langle\partial_s\Psi_k,\Psi_{j}\right\rangle_{\Vs}\, \Bigl]\, ds}_{=\,0}\\
 & = \int_0^1\int_{\square} a(y,s)e_{j}\cdot \left(p|v_0|^{(p-1)/p}\nabla_y\Psi_k+e_k\right)\, dyds \\
 & \quad +p|v_0|^{(p-1)/p}\int_0^1\Bigl[\int_{\square} a(y,s) \nabla_y\Psi_{j}\cdot\left(p|v_0|^{(p-1)/p}\nabla_y\Psi_k+e_k\right)\, dy +\left\langle\partial_s\Psi_k,\Psi_{j}\right\rangle_{\Vs}\, \Bigl]\, ds \\
 & =\int_0^1\Bigl[\int_{\square} a(y,s) (p|v_0|^{(p-1)/p}\nabla_y\Psi_{j}+e_j)\cdot\left(p|v_0|^{(p-1)/p}\nabla_y\Psi_k+e_k\right)\, dy +p|v_0|^{(p-1)/p}\left\langle\partial_s\Psi_k,\Psi_{j}\right\rangle_{\Vs}\, \Bigl]\, ds \\
 &= \int_0^1\int_{\square}  a(y,s) \left( p|v_0|^{(p-1)/p} \nabla_y\Psi_{j}+e_{j} \right)\cdot e_k\, dyds\\
 &\quad + p|v_0|^{(p-1)/p}\int_0^1\Bigl[\int_{\square} a(y,s)( p|v_0|^{(p-1)/p} \nabla_y\Psi_{j}+e_{j})\cdot\nabla_y\Psi_k\, dy + \left\langle\partial_s\Psi_k,\Psi_{j}\right\rangle_{\Vs}\, \Bigl]\, ds\\
 &= \underbrace{a_{\rm hom}(x,t) e_{j}\cdot e_k}_{\ = \, [a_{\rm hom}(x,t)]_{k,j}} +  p|v_0|^{(p-1)/p}\int_0^1 \left( \langle \partial_s\Psi_k, \Psi_{j} \rangle_{\Vs}  - \langle \partial_s \Psi_{j}, \Psi_k \rangle_{\Vs} \right)\, ds,
\end{align*}
which implies the same assertion as in the case $0<p<1$. Furthermore, when $v_0(x,t)=0$, we have already seen that $a_{\rm hom}(x,t)$ is symmetric, because so is $a(y,s)$ in (ii) of Remark \ref{R:interpre}. \qed

\begin{remark}[Skew-symmetric part makes no contribution to the diffusion]\label{R:no-contr}
{\rm
As mentioned in Proposition \ref{P:ahom}, the homogenized matrix $a_{\rm hom}$ may be asymmetric at the critical case $r = 2$. According to the proof above for $0<p<1$, the skew-symmetric part of $a_{\rm hom}$ reads,
\begin{align*}
\left(\frac{a_{\rm hom} - \tenchi a_{\rm hom}}2\right)_{jk}
&= \frac 1{2p}|v_0|^{(1-p)/p}\int_0^1 \left( \langle \partial_s\Phi_k, \Phi_{j} \rangle_{\Vs} - \langle \partial_s \Phi_{j}, \Phi_k \rangle_{\Vs} \right)\, ds\\
&= \frac 1p|v_0|^{(1-p)/p}\int_0^1 \langle \partial_s\Phi_k, \Phi_{j} \rangle_{\Vs}\, ds \quad \mbox{ for } \ j,k=1,2,\ldots,N.
\end{align*}
On the other hand, the skew-symmetric part seems to make no contribution to the homogenized diffusion. Indeed, suppose that $v_0 \neq 0$ and $\Phi_k$, $k=1,2,\ldots,N$, are all smooth enough. Then the diffusion term of the \eqref{analhomeq} reads,
\begin{align*}
 \mathrm{div} (a_{\rm hom} \nabla v_0)
&= \mathrm{div} \left( \frac{a_{\rm hom} + \tenchi a_{\rm hom}}2 \nabla v_0\right)
+ \mathrm{div} \left( \frac{a_{\rm hom} -  \tenchi a_{\rm hom}}2 \nabla v_0\right),
\end{align*}
and we here observe that
\begin{align*}
 \mathrm{div} \left( \frac{a_{\rm hom} -  \tenchi a_{\rm hom}}2 \nabla v_0\right)
&= \frac{1-p}{p^2} v_0^{(1-2p)/p} \partial_{ x_j} v_0 \left( \int^1_0 \int_\square (\partial_s \Phi_k) \Phi_j \, d y d s \right) \partial_{ x_k} v_0\\
&\quad + \frac 1p |v_0|^{(1-p)/p} \left[ \int^1_0 \int_\square \left\{ (\partial_{ x_j}\partial_s \Phi_k) \Phi_j + (\partial_s \Phi_k) (\partial_{ x_j} \Phi_j) \right\} \, d y d s \right] \partial_{ x_k} v_0\\
&\quad + \frac 1p |v_0|^{(1-p)/p} \left( \int^1_0 \int_\square (\partial_s \Phi_k) \Phi_j \, d y d s \right) \partial^2_{x_j x_k} v_0
\end{align*}
(here we used Einstein's summation convention). However, all the terms of the right-hand side vanish due to integration by parts and the symmetry of the Hessian. Hence the homogenized matrix $a_{\rm hom}$ may be asymmetric and the skew-symmetric part  is still alive  in the homogenized diffusion flux. However, the asymmetry will finally disappear in the homogenized diffusion.
}
\end{remark}

\section*{Acknowledgment}

The first author is supported by JSPS KAKENHI Grant Number JP20H01812, JP18K18715, 16H03946 and JP17H01095. The second author is partially supported by Division for Interdisciplinary Advanced Research and Education, Tohoku University and  Grant-in-Aid for JSPS Fellows (No. 20J10143).

\appendix

\section{Regularity of solutions to cell problems}\label{A:S:reg}

This section is devoted to discussing existence, uniqueness and regularity of weak solutions to cell-problems at the critical ratio $r = 2$. In what follows, we may simply write $w(y,s)$ for functions $w = w(x,t,y,s)$ by omitting variables $x,t$, unless any confusion may arise.

In case $r = 2$ and $p \in (0,1)$, for each $(x,t) \in \Omega \times (0,T)$, the cell-problem reads,
\begin{equation}\label{cp-2-1}
\left\{
\begin{array}{ll}
\frac1p |v_0|^{(1-p)/p}\partial_s\Phi_k(y,s)-\dv_y\left(a(y,s)\left[\nabla_y\Phi_k(y,s)+e_k\right]\right)=0 &\mbox{ in } \T^N\times\T,\\
\Phi_k(y,0) = \Phi_k(y,1) &\mbox{ in } \T^N
\end{array}
\right.
\end{equation}
such that $\langle \Phi_k(\cdot,s) \rangle_y = 0$ for $s \in \T$. Since $v_0 = v_0(x,t)$ depends only on $(x,t)$, it can be regarded as a constant to discuss existence, uniqueness and regularity of weak solutions to \eqref{cp-2-1} for each $(x,t)$ fixed. Moreover, the divergence of the vector field $a(y,s) e_k$ acts as a forcing term. In case $v_0(x,t) \neq 0$, assuming
\begin{equation}\label{ae-1}
a(y,s) e_k \in [L^2(J;L^2_{\rm per}(\square))]^N, \quad k=1,2,\ldots,N,
\end{equation}
one can construct a unique weak solution $\Phi_k = \Phi_k(x,t,\cdot,\cdot) \in L^2(J;V) \cap W^{1,2}(J;V^*)$, where $V := H^1_{\rm per}(\square) \setminus \R$, of the periodic problem \eqref{cp-2-1} by applying a general theory on non-autonomous evolution equations. The uniqueness of solutions follows from the strict monotonicity of the elliptic operator along with the zero mean condition $\langle \Phi_k(\cdot,s) \rangle_y = 0$. In case $v_0(x,t) = 0$, it suffices to use an elliptic theory instead of the parabolic one.

Thus the existence and uniqueness of the weak solution $\Phi_k(x,t,\cdot,\cdot)$ to the cell-problem \eqref{cp-2-1}  have  been proved for a.e.~$(x,t) \in \Omega \times I$, and moreover, as we shall see, it can also be proved that $\Phi_k(x,t,\cdot,\cdot)$ complies with a classical regularity, when $a(y,s)e_k$ is smooth enough. On the other hand, the regularity of $\Phi_k$ in $x,t$ seems more delicate; indeed, it depends on the regularity of the homogenized limit $v_0$ (see \eqref{cp-2-1}), whose regularity also relies on the smoothness of $a_{\rm hom}(x,t)$ consisting of the solutions $\{\Phi_k(x,t,y,s)\}_{k=1,2,\ldots,N}$. However, as will be shown below, the boundedness at least can be proved. 

\begin{lem}[Strong measurability in $(x,t)$]\label{A:L:1}
Assume $r=2$, $p \in (0,1)$ and \eqref{ae-1}. For $k=1,2,\ldots,N$, the function $(x,t) \mapsto \Phi_k(x,t,\cdot,\cdot)$  {\rm (}resp., $|v_0(x,t)|^{(1-p)/p} \Phi_k(x,t,\cdot,\cdot)${\rm )}  is strongly measurable in $\Omega \times (0,T)$ with values in $L^2(J;V)$  {\rm (}resp., in $W^{1,2}(J;V^*)${\rm )}.   Moreover, $\Phi_k \in L^\infty(\Omega\times I;L^2(J;V))$  and $|v_0|^{(1-p)/p} \Phi_k \in L^\infty(\Omega\times I;W^{1,2}(J;V^*))$. 
\end{lem}

The lemma mentioned above is enough to discuss a characterization of the homogenized matrix as in Theorem \ref{thm2}.

\begin{proof}
Since $|v_0|^{(1-p)/p}$ lies in $L^{(p+1)/(1-p)}(\Omega\times I)$, one can take a sequence $(\sigma_n)$ of step functions from $\Omega\times I$ into $\R$ such that $\sigma_n(x,t) \to (1/p)|v_0(x,t)|^{(1-p)/p}$ for $(x,t) \in Q_0$, where $Q_0$ is a measurable set in $\Omega\times I$ satisfying $|(\Omega\times I) \setminus Q_0| = 0$, as $n \to +\infty$. Here one can assume  that $\sigma_n(x,t) > 0$  without any loss of generality.  Fix $(x,t)\in Q_0$ and let $P_n(x,t,\cdot,\cdot) \in L^2(J;V) \cap W^{1,2}(J;V^*)$ be the (unique) weak solution to  
\begin{equation}\label{cp-2-1n}
\left\{
\begin{array}{ll}
 \sigma_n \partial_s P_n(y,s)-\dv_y\left(a(y,s)\left[\nabla_y P_n(y,s)+e_k\right]\right)=0 &\mbox{ in } \T^N\times\T,\\
  P_n(y,0) = P_n(y,1) &\mbox{ in } \T^N
\end{array}
\right.
\end{equation}
such that $\langle P_n(\cdot,s) \rangle_y = 0$ for all $s \in \T$. Moreover, we note that the vector-valued function $(x,t) \mapsto P_n(x,t,\cdot,\cdot)$ is a simple function defined over $\Omega\times I$. Test \eqref{cp-2-1n} by $P_n$. We observe by \eqref{ellip} that
\begin{align*}
\lefteqn{
\frac{ \sigma_n }2 \frac{d}{d s} \|P_n(s)\|_{L^2(\square)}^2 + \lam \|\nabla_y P_n(s)\|_{L^2(\square)}^2
}\\
&\leq - \int_\square a(y,s) e_k \cdot \nabla_y P_n(y,s) \, d y
\leq \frac\lam 2 \|\nabla_y P_n(s)\|_{L^2(\square)}^2 + C \|a(s)e_k\|_{L^2(\square)}^2.
\end{align*}
Integrate both sides over $(0,1)$ and employ the periodicity, $P_n(\cdot,0)=P_n(\cdot,1)$ in $\T^N$. It then follows that
$$
\frac\lam2\int^1_0 \|\nabla_y P_n(s)\|_{L^2(\square)}^2\, ds \leq C \int^1_0 \|a(s)e_k\|_{L^2(\square)}^2 \, ds.
$$
Then one can also obtain
$$
\sigma_n^{2}\int^1_0 \|\partial_s P_n(s)\|_{V^*}^2 \, \d s \leq C \int^1_0 \|a(s)e_k\|_{L^2(\square)}^2 \, ds.
$$
Therefore we can take a (not relabeled) subsequence of $(n)$ and a limit $P(x,t,\cdot,\cdot) \in L^2(J;V)$ such that  $|v_0(x,t)|^{(1-p)/p}P(x,t,\cdot,\cdot) \in W^{1,2}(J;V^*)$ and
\begin{alignat*}{4}
P_n(x,t,\cdot,\cdot) &\to P(x,t,\cdot,\cdot) \quad &&\mbox{ weakly in } L^2(J;V),\\
 \sigma_n(x,t) P_n(x,t,\cdot,\cdot) &\to \tfrac 1 p|v_0(x,t)|^{(1-p)/p}P(x,t,\cdot,\cdot) \quad && \mbox{ weakly in } W^{1,2}(J;V^*).
\end{alignat*}
Hence $(x,t) \mapsto P(x,t,\cdot,\cdot)$ is weakly measurable in $\Omega\times I$ with values in $L^2(J;V)$,  and therefore, it is also strongly measurable due to Pettis's theorem (one can also apply the same argument to $\sigma_n(x,t) \partial_s P_n(x,t,y,s)$).  Furthermore, recalling the convergence $\sigma_n(x,t) \to (1/p)|v_0(x,t)|^{(1-p)/p}$ a.e.~in $\Omega \times I$ as $n \to +\infty$, one can verify that $P$ solves \eqref{cp-2-1} for a.e.~$(x,t) \in \Omega \times I$, and hence, the uniqueness of solutions yields $P = \Phi_k$. Finally, the a priori bounds along with the weak lower-semicontinuity of norms yield that $\Phi_k \in L^\infty(\Omega\times I ;  L^2(J;V))$  and $|v_0|^{(1-p)/p}\Phi_k \in L^\infty(\Omega\times I;W^{1,2}(J;V^*))$. 
\end{proof}

To discuss further regularity of $\Phi_k$, which is needed in the proof of Theorem \ref{T:cor},  let us assume that
\begin{equation}\label{ae-2}
a(y,s) e_k \ \mbox{ is smooth and $(\square \times J)$-periodic for }  k = 1,2,\ldots,N.
\end{equation}
Then using a general theory based on a Hilbert space setting, one can assure that $\Phi_k$ is a strong solution of \eqref{cp-2-1} in the frame of $L^2_{\rm per}(\square)$. Here we may also use the fact that $H^2_{\rm loc}(\R^N) \cap H^1_{\rm per}(\square) = H^2_{\rm per}(\square)$.\footnote{Indeed, let $u \in W^{m,q}_{\rm loc}(\R^N)$ be such that $u(\cdot + h) = u(\cdot)$ for any $h \in \Z^N$ and let $\rho_n$ be a mollifier. Then $u_n:=u*\rho_n$ turns out to be an element of $C^\infty_{\rm per}(\R^N)$, and it converges to $u$ strongly in $W^{m,q}(\square)$.} One can further employ a classical regularity theory for linear parabolic equations (see, e.g.,~\cite{Friedman,LSU}) to obtain interior classical regularity of the strong solution, and moreover, it can be extended to $\T^N$; indeed, by periodicity, the $L^2$ solution of \eqref{cp-2-1} is also an $L^2$ solution of the same PDE on any cubic domain of the form $\square +  y_0$ for $ y_0  \in \R^N$ (in another word, it is due to the periodic structure of $\T^N$). Furthermore, one can verify that $\partial_s \Phi_k$ is also periodic  in $s$,  that is, $\partial_s \Phi_k(\cdot,0)=\partial_s \Phi_k(\cdot,1)$ in $\square$. Indeed, $s \mapsto \Phi_k(x,t,y,s)$  can be extended onto $[0,+\infty)$ by periodically concatenating its orbit. Then the extended orbit solves the same PDE with the smooth periodic coefficient fields at least in the strong sense. Therefore applying a regularity theory again, we obtain smoothness (e.g., $C^1$ regularity) of $s \mapsto \Phi_k(x,t,\cdot,s)$, say in $L^2(\square)$, whenever $s \mapsto \mathrm{div}_y\,a(\cdot,s)e_k$  lies on $C(\overline J;L^2(\square))$,  and thus, the periodicity of $\partial_s \Phi_k$ follows.

We further prove that
\begin{lem}[Regularity for $p\in(0,1)$]\label{A:L:2}
Assume $r=2$, $p \in (0,1)$ and \eqref{ae-2}. For $k=1,2,\ldots,N$, the function $\nabla_y\Phi_k$ lies on $[L^\infty(\Omega\times I ; C_{\rm per}(\square \times J))]^N$.
\end{lem}

\begin{proof}
As in the proof of Lemma \ref{A:L:1}, we can derive from \eqref{cp-2-1} that
\begin{equation}\label{A:e0}
\frac\lam2\int^1_0 \|\nabla_y \Phi_k(s)\|_{L^2(\square)}^2\, ds \leq C \int^1_0 \|a(s)e_k\|_{L^2(\square)}^2 \, ds
\end{equation}
for a.e.~$(x,t) \in \Omega \times I$. Differentiate \eqref{cp-2-1} in $y_i$ and set $\Phi_{k,i} = \partial_{y_i} \Phi_k$. Then $\Phi_{k,i}$ solves
\begin{align}
\lefteqn{
\frac1p |v_0|^{(1-p)/p}\partial_s \Phi_{k,i}(y,s)-\dv_y\left(a(y,s) \nabla_y \Phi_{k,i}(y,s)\right)
}\label{cp-2-1dy}
\\
&=\dv_y\left(\partial_{y_i} a(y,s)e_k\right) + \dv_y\left(\partial_{y_i} a(y,s)\nabla_y \Phi_k(y,s)\right) \mbox{ in } \T^N\times\T.\nonumber
\end{align}
Test it by $\Phi_{k,i}$ to get
\begin{align*}
\lefteqn{
\frac1{2p}|v_0|^{(1-p)/p} \frac{d}{ds} \|\Phi_{k,i}(s)\|_{L^2(\square))}^2 + \lam \|\nabla_y \Phi_{k,i}(s)\|_{L^2(\square)}^2
}\\
&\leq - \int_\square \partial_{y_i} a(y,s)e_k \cdot \nabla_y \Phi_{k,i}(y,s) \, dy - \int_\square \partial_{y_i} a(y,s) \nabla_y \Phi_k(y,s) \cdot \nabla_y \Phi_{k,i}(y,s) \, dy\\
&\leq \frac \lam 2 \|\nabla_y \Phi_{k,i}(s)\|_{L^2(\square)}^2 + C \left( \|\partial_{y_i} a(s) e_k\|_{L^2(\square)}^2 + \|\partial_{y_i}a\|_{L^\infty(\square \times J)}^2 \|\nabla_y \Phi_k(s)\|_{L^2(\square)}^2 \right),
\end{align*}
which along with \eqref{A:e0} yields 
\begin{equation}\label{A:e1}
\esssup_{(x,t) \in \Omega\times I} \left( \int^1_0 \|\nabla_y \Phi_{k,i}(x,t,\cdot,s)\|_{L^2(\square)}^2 \, ds \right) \leq C.
\end{equation}
Hence $\Phi_k$ turns out to be an element of $L^\infty(\Omega \times I ; L^2(J;H^2(\square)))$, since the strong measurability of $\Phi_k$ over $\Omega \times I$ with values in $L^2(J;H^2(\square))$ can be proved as in the proof of Lemma \ref{A:L:1}. Differentiate both sides of \eqref{cp-2-1dy} in $y_j$ and set $\Phi_{k,ij} := \partial^2_{ij} \Phi_k$. Then $\Phi_{k,ij}$ solves
\begin{align*}
\lefteqn{
\frac1p |v_0|^{(1-p)/p}\partial_s \Phi_{k,ij}(y,s)-\dv_y\left(a(y,s) \nabla_y \Phi_{k,ij}(y,s)\right)
}
\\
&=\dv_y\left(\partial^2_{y_iy_j} a(y,s)e_k\right) + \dv_y\left(\partial_{y_i} a(y,s)\nabla_y \Phi_{k,j}(y,s)\right)\nonumber\\
&\quad + \dv_y\left( \partial^2_{y_iy_j} a(y,s)\nabla_y \Phi_k(y,s)\right)
+ \dv_y \left( \partial_{y_j} a(y,s) \nabla_y \Phi_{k,i}(y,s) \right)
\mbox{ in } \T^N\times\T.\nonumber
\end{align*}
Test it again by $\Phi_{k,ij}$ to see that
\begin{align*}
\lefteqn{
\frac1{2p} |v_0|^{(1-p)/p} \frac d{ds} \|\Phi_{k,ij}(s)\|_{L^2(\square)}^2 + \lam \| \nabla_y \Phi_{k,ij}(s) \|_{L^2(\square)}^2
}\\
&\leq \frac \lam 2 \| \nabla_y \Phi_{k,ij}(s) \|_{L^2(\square)}^2
+ C \Big( \|\partial^2_{y_iy_j} a(s)e_k \|_{L^2(\square)}^2  +  \|\partial_{y_i} a\|_{L^\infty(\square \times J)}^2 \|\nabla_y \Phi_{k,j}(s)\|_{L^2(\square)}^2\\
&\quad + \|\partial^2_{y_iy_j} a\|_{L^\infty(\square\times J)}^2 \|\nabla_y \Phi_k(s)\|_{L^2(\square)}^2 + \| \partial_{y_j} a \|_{L^\infty(\square \times J)}^{ 2} \| \nabla_y \Phi_{k,i}(s)\|_{L^2(\square)}^2 \Big),
\end{align*}
which together with \eqref{A:e0} and \eqref{A:e1} implies
\begin{equation*}
\esssup_{(x,t) \in \Omega\times I} \left( \int^1_0 \|\nabla_y \Phi_{k,ij}(x,t,\cdot,s)\|_{L^2(\square)}^2 \, ds \right) \leq C.
\end{equation*}
These procedures can be also performed by differentiating equations in $s$. For instance, differentiate both sides of \eqref{cp-2-1} in $s$. Then $\Phi_{k,s} := \partial_s \Phi_k$ solves
\begin{align*}
\lefteqn{
\frac1p |v_0|^{(1-p)/p} \partial_s \Phi_{k,s}(y,s)-\dv_y\left(a(y,s) \nabla_y \Phi_{k,s}(y,s)\right)
}
\\
&=\dv_y\left(\partial_s a(y,s)e_k\right) + \dv_y\left(\partial_s a(y,s)\nabla_y \Phi_k(y,s)\right) \mbox{ in } \T^N\times\T\nonumber
\end{align*}
along with the periodicity $\Phi_{k,s}(\cdot,0) = \Phi_{k,s}(\cdot,1)$ in $\square$. Testing it by $\Phi_{k,s}$, we observe that
\begin{align*}
\lefteqn{
\frac1{2p} |v_0|^{(1-p)/p} \frac{d}{d s} \|\Phi_{k,s}(s)\|_{L^2(\square)}^2 + \lam \|\nabla_y \Phi_{k,s}(s)\|_{L^2(\square)}^2
}\\
&\leq - \int_\square \partial_s a(y,s) e_k \cdot \nabla_y \Phi_{k,s}(y,s) \, d y
- \int_\square \partial_s a(y,s) \nabla_y \Phi_k(y,s) \cdot \nabla_y \Phi_{k,s}(y,s) \, d y\\
&\leq \frac\lam 2 \|\nabla_y \Phi_{k,s}(s)\|_{L^2(\square)}^2
+ C \left( \|\partial_s a(s)e_k\|_{L^2(\square)}^2 + \|\partial_s a(s)\|_{L^\infty(\square)}^2 \|\nabla_y \Phi_k(s)\|_{L^2(\square)}^2 \right).
\end{align*}
Therefore it follows that
\begin{equation*}
\int^1_0 \|\nabla_y \Phi_{k,s}(s)\|_{L^2(\square)}^2 \, ds
\leq C.
\end{equation*}
Hence iterating these procedures in finite time and using the Sobolev embedding theorem $H^m(\square) \hookrightarrow C(\overline\square)$ for $m > N/2$, we can finally arrive at
\begin{equation*}
\esssup_{(x,t) \in \Omega\times I} \left( \int^1_0 \|\partial_s \nabla_y\Phi_k(x,t,\cdot,s)\|_{C_{\rm per}(\square)}^2 \, ds \right) \leq C,
\end{equation*}
whence follows that
$$
\partial_{y_j}\Phi_k \in L^\infty(\Omega\times I ; W^{1,2}(J;C_{\rm per}(\square))) \hookrightarrow L^\infty(\Omega\times I; C_{\rm per}(\square \times J))
$$
for $j=1,2,\ldots,N$. This completes the proof.
\end{proof}

As we have seen, the regularity obtained in Lemma \ref{A:L:2} is not at all optimal and one may prove better ones in $y,s$ (but still bounded in $x,t$) for $\Phi_k$. However, we just proved the assertion necessary for proving Theorem \ref{T:cor}.

In case $r = 2$ and $p \in (1,2)$, it suffices to consider the case that $v_0(x,t) \neq 0$ only. For each $(x,t) \in  [v_0\neq0] := \{(x,t) \in \Omega \times (0,T) \colon v_0(x,t) \neq 0\}$,  under \eqref{ae-1},   we can verify existence and uniqueness of a weak solution $\Psi_k(x,t,\cdot,\cdot)  \in L^2(J;V) \cap W^{1,2}(J;V^*)$  to the cell-problem,
\begin{equation}\label{cp-2-2}
\left\{
\begin{array}{ll}
\partial_s\Psi_k(y,s)-\dv_y\left(a(y,s)\left[ p |v_0|^{(p-1)/p}\nabla_y\Psi_k(y,s)+e_k \right] \right)=0 &\mbox{ in } \T^N\times\T,\\
\Psi_k(y,0) = \Psi_k(y,1) &\mbox{ in } \T^N
\end{array}
\right.
\end{equation}
such that $\langle \Psi_k(\cdot,s) \rangle_y = 0$ for $s \in \T$ as in Lemma \ref{A:L:1}.  Moreover, we  claim that
\begin{equation*}
 |v_0|^{(p-1)/{ p}} \Psi_k \in L^{ \infty}( [v_0\neq0];L^2(J;V)),
\quad \Psi_k\in L^{ \infty}( [v_0\neq0];W^{1,2}(J;V^*)). 
\end{equation*}
Indeed, testing \eqref{cp-2-2} by  $|v_0|^{(p-1)/p}\Psi_k$,  we  can verify by \eqref{ellip} that $|v_0|^{ (p-1)/p} \Psi_k\in L^{ \infty}([v_0\neq0];L^2(J;V))$, and  moreover,  it follows that
\begin{align*}
\lefteqn{
 \int^1_0  \langle\partial_s\Psi_k(x,t,\cdot,s),\phi(s) \rangle_{V}  \, \d s 
}\\
& \le \left\|a \left[p|v_0(x,t)|^{(p-1)/p}\nabla_y\Psi_k(x,t,\cdot,\cdot)+e_k\right]\right\|_{L^2(\square\times J)}\|\phi\|_{L^2(J;V)}
\quad \mbox{ for } \ \phi \in L^2(J;V),
\end{align*}
which implies that  $\Psi_k \in L^\infty([v_0\neq0]; W^{1,2}(J; V^*))$.  
 We shall then prove that

\begin{lem}[Regularity for $p\in(1,2)$]\label{A:L:3}
Assume $r=2$, $p\in(1,2)$ and \eqref{ae-2}. For $k=1,2,\ldots,N$, the function 
$$
 \nabla_y \Phi_k(x,t,y,s):=\begin{cases}
			   p|v_0(x,t)|^{(p-1)/p}\nabla_y  \Psi_k(x,t,y,s) &\mbox{ if } \ v_0(x,t) \neq 0,\\
			   0 &\mbox{ otherwise}
			  \end{cases}
$$
belongs to $L^\infty(\Omega\times I ; C_{\rm per}(\square \times J))$.
\end{lem}

\begin{proof}
Suppose that $v_0(x,t) \neq 0$. Testing \eqref{cp-2-2} by $\Phi_k = p|v_0|^{(p-1)/p}\Psi_k$, we have
\begin{align*}
\lefteqn{
 \frac p2 |v_0|^{(p-1)/p} \frac d{ds} \|\Psi_k(s)\|_{L^2(\square)}^2 + \lam \|\nabla_y \Phi_k(s)\|_{L^2(\square)}^2
}\\
&\leq - \int_\square a(y,s) e_k \cdot \nabla_y \Phi_k(y,s) \, \d y
\leq \frac \lam 2 \|\nabla_y \Phi_k(s)\|_{L^2(\square)}^2 + C \|a(s) e_k\|_{L^2(\square)}^2,
\end{align*}
which implies
$$
\frac \lam 2 \int^1_0 \|\nabla_y \Phi_k(s)\|_{L^2(\square)}^2 \, ds \leq C \int^1_0 \|a(s) e_k\|_{L^2(\square)}^2 \, d s
$$
for a.e.~$(x,t)\in\Omega\times I$. Differentiate both sides of \eqref{cp-2-2} in $y_i$ and set $\Psi_{k,i} := \partial_{y_i} \Psi_k$ and $\Phi_{k,i} := \partial_{y_i} \Phi_k$. Then we see that
\begin{align*}
\lefteqn{
\partial_s\Psi_{k,i}(y,s)-\dv_y\left(a(y,s) \nabla_y\Phi_{k,i}(y,s)\right)
}\\
&= \dv_y\left(\partial_{y_i} a(y,s) e_k \right) + \dv_y \left( \partial_{y_i} a(y,s) \nabla_y \Phi_k(y,s) \right)\ \mbox{ in } \T^N\times\T.\nonumber
\end{align*}
Testing it by $\Phi_{k,i}$, we can derive
\begin{align*}
\lefteqn{
\int^1_0 \|\nabla_y \Phi_{k,i}(s)\|_{L^2(\square)}^2 \, ds
}\\
&\leq C \left( \int^1_0 \|\partial_{y_i} a(s)\|_{L^2(\square)}^2 \, d s + \|\partial_{y_i} a\|_{L^\infty(\square \times J)}^2 \int^1_0 \|\nabla_y \Phi_k(s)\|_{L^2(\square)}^2 \, ds \right)
\end{align*}
for a.e.~$(x,t)\in\Omega\times I$. 
Thus we get
$$
\esssup_{(x,t) \in \Omega\times I} \left( \int^1_0 \|\nabla_y \Phi_{k,i}(x,t,\cdot,s)\|_{L^2(\square)}^2 \, ds \right) \leq C.
$$
Repeating the argument so far and also applying it to the differentiation in $s$ as well, we can finally obtain
$$
\esssup_{(x,t) \in \Omega\times I} \left( \int^1_0 \|\partial_s \nabla_y \Phi_k(x,t,\cdot,s)\|_{C_{\rm per}(\square)}^2 \, ds \right) \leq C,
$$
and therefore,
$$
\nabla_y \Phi_k \in L^\infty(\Omega\times I;C_{\rm per}(\square \times J)). 
$$
This completes the proof.
\end{proof}

\section{Maximal monotonicity of the operator $A$}\label{A:S:monotone}

We show that the operator $A:L^{(p+1)/p}(\Omega)\to L^{p+1}(\Omega)$ defined by \eqref{appendix} is maximal monotone in $L^{(p+1)/p}(\Omega)\times L^{p+1}(\Omega)$. Indeed, define a functional $\phi:L^{(p+1)/p}(\Omega)\to [0,+\infty)$ by
$$
\phi(w)=\frac{p}{p+1}\int_{\Omega}|w|^{(p+1)/p}\, dx\quad \text{ for } \ w\in L^{(p+1)/p}(\Omega).
$$
Then $\phi$ is convex and continuous on $L^{(p+1)/p}(\Omega)$, and moreover, $\partial\phi(w)=\{w^{1/p}\}$ for $w\in L^{(p+1)/p}(\Omega)$.
Indeed, the continuity and convexity of $\phi$ follow from the continuity of norms, Minkowski's inequality and the convexity of the power function $|\cdot|^{(p+1)/p}$. Moreover, let $w \in L^{(p+1)/p}(\Omega)$. Then, for $\zeta\in\partial\phi(w)$, we have
\begin{equation}\label{subd-def}
\phi(u)-\phi(w)\ge\langle\zeta, u-w\rangle_{L^{(p+1)/p}(\Omega)}\quad\text{for all } \ u\in D(\phi).
\end{equation}
The left-hand side  of \eqref{subd-def}  is majorized as follows\/:
\begin{align*}
\lefteqn{
 \phi(u)-\phi(w)
}\\
 &= \frac{p}{p+1}\int_{\Omega}\left(|u|^{(p+1)/p}-|w|^{(p+1)/p}\right)\, dx\nonumber\\
 &\le \int_{\Omega}u^{1/p}(u-w)\, dx\nonumber\\
 &= \int_{\Omega}w^{1/p}(u-w)\, dx + \int_{\Omega}\left(u^{1/p}-w^{1/p}\right)(u-w)\, dx\nonumber\\ 
 &\le \int_{\Omega}w^{1/p}(u-w)\, dx\nonumber\\
 &\quad +
     \begin{cases}
      C\int_{\Omega}\left(|w|^{(1-p)/p}+|u|^{(1-p)/p}\right)|u-w|^2\, dx\quad &\text{ if }\ 0<p\le1, \vspace{3mm}\nonumber \\
      C\int_{\Omega}|u-w|^{(p+1)/p}\, dx\quad &\text{ if }\ 1<p<+\infty.\nonumber
     \end{cases}
\end{align*}
Substituting $u=w+\theta v$ for $\theta\in (0,1)$, $v\in D(\phi)$, we obtain
\begin{align*}
\lefteqn{
\phi(u)-\phi(w) \le \theta\int_{\Omega}w^{1/p}v\, dx
}\\ 
&\quad+
\begin{cases}
C\theta^2\int_{\Omega}\left(|w|^{(1-p)/p}+|w+\theta v|^{(1-p)/p}\right)v^2\, dx\ &\text{ if }\ 0<p\le1, \vspace{3mm}\\
C\theta^{(p+1)/p}\int_{\Omega}|v|^{(p+1)/p}\, dx &\text{ if }\ 1<p<+\infty.
\end{cases}
\end{align*}
On the other hand, the right-hand side  of \eqref{subd-def}  reads 
$$
\langle\zeta,u-w\rangle_{L^{(p+1)/p}(\Omega)}=\theta\int_{\Omega}\zeta v\, dx.
$$
Hence dividing both sides by $\theta$ and letting $\theta\to 0_+$ in both sides, we obtain 
$$
\int_{\Omega}w^{1/p}v\ dx\ge\int_{\Omega}\zeta v\ dx\quad \text{ for all } \ v\in D(\phi).
$$
Repeating the same argument for $u=w-\theta v$, we can verify that
$$
\int_{\Omega}w^{1/p}v\, dx=\int_{\Omega}\zeta v\, dx\quad \text{ for all } \ v\in D(\phi).
$$
Thus we obtain $\zeta=w^{1/p}$, and hence, $A=\partial\phi$.
Therefore the operator $A:L^{(p+1)/p}(\Omega)\to L^{p+1}(\Omega)$ turns out to be maximal monotone in $L^{(p+1)/p}(\Omega)\times L^{p+1}(\Omega)$
(see Theorem \ref{minty}).

\section{Proof of Lemma \ref{aeconv}}\label{pf-pc}

It follows from (i) of Lemma \ref{bdd} (i.e., $\sup_{t\in \overline{I}}\|v_{\e}(t)^{1/p}\|_{L^{p+1}(\Omega)}\le C$) that, for each $t\in \overline{I}$, up to a (not relabeled) subsequence of $(\e)$,
\begin{align}
v_{\e}(t)^{1/p}\to Z_t\quad \text{ weakly in }\ L^{p+1}(\Omega)\label{2002261}
\end{align}
for some $Z_t\in L^{p+1}(\Omega)$. 
Hence it suffices to prove that
\[
  Z_t=v_0(t)^{1/p}\quad \text{ for all }\ t\in \overline{I}.
\]
Since $v_{\e}^{1/p}\to v_{0}^{1/p}$ in $L^{p+1}(\Omega\times I)$ (see \eqref{strongconvsol5}), one can take a (not relabeled) subsequence of $(\e)$ such that
\begin{align}
 v_{\e}(t)^{1/p}\to v_{0}(t)^{1/p}\quad \text{ strongly in }\ L^{p+1}(\Omega)\label{2002262}
\end{align}
for a.e.~$t\in I$. 
Hence $Z_t=v_0(t)^{1/p}$ for all $t\in I\setminus I_0$ for some $I_0\subset I$ satisfying $|I_0|=0$.
Now, fix $t\in\overline{I}$ arbitrary. 
Then one can take a sequence $(t_n)$ in $I\setminus I_0$ such that $(t_n)\to t$.
 Moreover,  we have 
\begin{align}
\langle Z_{t_n},\phi\rangle_{L^{(p+1)/p}(\Omega)}=\langle v_0(t_n)^{1/p},\phi\rangle_{L^{(p+1)/p}(\Omega)}\quad \text{ for all }\ \phi\in L^{(p+1)/p}({\Omega}).\label{lem6.2.1}
\end{align}
Fix $\phi \in H^1_0(\Omega)\cap L^{(p+1)/p}(\Omega)$ arbitrarily. It follows that
\begin{align*}
\lefteqn{|\langle Z_t-v_{0}(t)^{1/p}, \phi \rangle_{H^1_0(\Omega)\cap L^{(p+1)/p}(\Omega)}|}\\
 &\le
     |\langle Z_t-v_{\e}(t)^{1/p}, \phi \rangle_{H^1_0(\Omega)\cap L^{(p+1)/p}(\Omega)}|+|\langle v_{\e}(t)^{1/p}-v_{\e}(t_n)^{1/p}, \phi \rangle_{H^1_0(\Omega)\cap L^{(p+1)/p}(\Omega)}|\\
 &\quad+
     |\langle v_{\e}(t_n)^{1/p}-v_{0}(t_n)^{1/p}, \phi \rangle_{H^1_0(\Omega)\cap L^{(p+1)/p}(\Omega)}|+|\langle v_{0}(t_n)^{1/p}-v_{0}(t)^{1/p}, \phi \rangle_{H^1_0(\Omega)\cap L^{(p+1)/p}(\Omega)}|\\
 &=:I_1(\e)+I_2(\e,n)+I_3(\e,n)+I_4(n).
\end{align*}
By (ii) of Lemma \ref{bdd}, we observe that
\begin{align*}
   \lefteqn{I_2(\e,n)+I_4(n)}\\
 &\le
   \|\phi\|_{H^1_0(\Omega)}(\|v_{\e}(t)^{1/p}-v_{\e}(t_n)^{1/p}\|_{H^{-1}(\Omega)} + \|v_{0}(t)^{1/p}-v_{0}(t_n)^{1/p}\|_{H^{-1}(\Omega)})\\
 &\le
   \|\phi\|_{H^1_0(\Omega)}(\|\partial_t v_{\e}^{1/p}\|_{L^2(I;H^{-1}(\Omega))} + \|\partial_t v_{0}^{1/p}\|_{L^2(I;H^{-1}(\Omega))} )|t-t_n|^{1/2}\\
 &\le
   \|\phi\|_{H^1_0(\Omega)}(C + \|\partial_t v_{0}^{1/p}\|_{L^2(I;H^{-1}(\Omega))} )|t-t_n|^{1/2}
=: \delta |t-t_n|^{1/2}.
\end{align*}
For any $\nu>0$, one can take $n_{\nu}\in\N$ large enough that $|t-t_{n_{\nu}}|< \nu^2/(2\delta)^2$, and hence, we have
\begin{align}\label{202003032}
\sup_{\e\in (0,1)} \bigl(I_2(\e,n_{\nu})+I_4(n_{\nu})\bigl)<\frac{\nu}{2}.
\end{align}
By \eqref{2002261} and \eqref{2002262} along with \eqref{lem6.2.1}, 
we can take a sequence $\e_{\nu}>0$ small enough that 
\begin{align}\label{202003031}
I_1(\e)+I_3(\e,n_{\nu})< \frac{\nu}{2}\quad \text{ for any }\ \e\in(0,\e_{\nu}) .
\end{align}
Therefore, we conclude by \eqref{202003032} and \eqref{202003031} that
\begin{align*}
|\langle Z_t-v_{0}(t)^{1/p}, \phi \rangle_{H^1_0(\Omega)\cap L^{(p+1)/p}(\Omega)}| < \frac{\nu}{2}+\frac{\nu}{2}=\nu.
\end{align*}
Thus we obtain $\langle Z_t - v_0(t)^{1/p}, \phi \rangle_{H^1_0(\Omega)\cap L^{(p+1)/p}(\Omega)} = 0$, which along with the arbitrariness of $\phi \in H^1_0(\Omega) \cap L^{(p+1)/p}(\Omega)$ and $t \in \overline{I}$ yields that
$$
Z_t = v_0(t)^{1/p} \quad \mbox{ for all } \ t \in \overline{I}. 
$$
 Consequently,  \eqref{2002261} implies the assertion. \qed

\end{document}